\pdfoutput=1

\documentclass[12pt,reqno]{amsart}

\usepackage{a4wide}
\usepackage{amsmath}
\usepackage{amssymb}
\usepackage{amsxtra}
\usepackage{amsthm}
\usepackage{mathtools}
\usepackage{extpfeil}
\usepackage{enumitem}
\usepackage[mathscr]{eucal}
\usepackage{graphicx}
\usepackage{comment}
\usepackage{asymptote}
\usepackage{proof}
\usepackage{xstring}
\usepackage[T1]{fontenc}

\usepackage[all]{xy}
\SelectTips{cm}{}
\ifx\pdfoutput\undefined
  \xyoption{dvips}
\fi

\usepackage{hyperref}
\usepackage{slashed}
\usepackage{xr}

\newif\ifarxiv

\arxivfalse

\newcommand{\ifundef}[1]{\expandafter\ifx\csname#1\endcsname\relax}

\DeclareMathAlphabet{\mathbbe}{U}{bbold}{m}{n}

\makeatletter

\def\re@DeclareMathSymbol#1#2#3#4{%
    \let#1=\undefined
    \DeclareMathSymbol{#1}{#2}{#3}{#4}}

\ifundef{Top}
  \DeclareSymbolFont{tcSyC}{U}{txsyc}{m}{n}
  \SetSymbolFont{tcSyC}{bold}{U}{txsyc}{bx}{n}
  \DeclareFontSubstitution{U}{txsyc}{m}{n}

  \re@DeclareMathSymbol{\Top}{\mathord}{tcSyC}{120}
  \re@DeclareMathSymbol{\Bot}{\mathord}{tcSyC}{121}
\fi

\ifundef{righthalfcup}
  \DeclareFontFamily{U}{MnSymbolC}{}
  \DeclareSymbolFont{mnSyC}{U}{MnSymbolC}{m}{n}
  \SetSymbolFont{mnSyC}{bold}{U}{MnSymbolC}{b}{n}
  \DeclareFontShape{U}{MnSymbolC}{m}{n}{
      <-6>  MnSymbolC5
     <6-7>  MnSymbolC6
     <7-8>  MnSymbolC7
     <8-9>  MnSymbolC8
     <9-10> MnSymbolC9
    <10-12> MnSymbolC10
    <12->   MnSymbolC12}{}
  \DeclareFontShape{U}{MnSymbolC}{b}{n}{
      <-6>  MnSymbolC-Bold5
     <6-7>  MnSymbolC-Bold6
     <7-8>  MnSymbolC-Bold7
     <8-9>  MnSymbolC-Bold8
     <9-10> MnSymbolC-Bold9
    <10-12> MnSymbolC-Bold10
    <12->   MnSymbolC-Bold12}{}
  
  \re@DeclareMathSymbol{\righthalfcup}{\mathord}{mnSyC}{184}
  \re@DeclareMathSymbol{\lefthalfcap}{\mathord}{mnSyC}{185}
\fi

\DeclareFontFamily{U}{MnSymbolA}{}
\DeclareSymbolFont{mnSyA}{U}{MnSymbolA}{m}{n}
\SetSymbolFont{mnSyA}{bold}{U}{MnSymbolA}{b}{n}
\DeclareFontShape{U}{MnSymbolA}{m}{n}{
    <-6>  MnSymbolA5
   <6-7>  MnSymbolA6
   <7-8>  MnSymbolA7
   <8-9>  MnSymbolA8
   <9-10> MnSymbolA9
  <10-12> MnSymbolA10
  <12->   MnSymbolA12}{}
\DeclareFontShape{U}{MnSymbolA}{b}{n}{
    <-6>  MnSymbolA-Bold5
   <6-7>  MnSymbolA-Bold6
   <7-8>  MnSymbolA-Bold7
   <8-9>  MnSymbolA-Bold8
   <9-10> MnSymbolA-Bold9
  <10-12> MnSymbolA-Bold10
  <12->   MnSymbolA-Bold12}{}

\re@DeclareMathSymbol{\twoheadedswarrow}{\mathord}{mnSyA}{30}

\makeatother

\newcommand{\mlaux}[3]{\setbox0=\hbox{$\mathsurround=0pt #2{#3}$}%
  \dimen0=\dp0\advance\dimen0 by \ht0\lower#1\dimen0\box0}

\newcommand{\makellapm}[2]{\hbox to 0pt{\hss$\mathsurround=0pt #1{#2}$}}

\newcommand{\makerlapm}[2]{\hbox to 0pt{$\mathsurround=0pt #1{#2}$\hss}}

\newcommand{\makelapm}[2]{\hbox to 0pt{\hss$\mathsurround=0pt #1{#2}$\hss}}

\newcommand{\makeushort}[3]{%
	\setbox0=\hbox{$\mathsurround=0pt #2{#3}$}%
	\hbox to 1\wd0{\hss\underbar{\hbox to #1\wd0{\hss\box0\hss}}\hss}}

\def\makebigger#1#2#3{\scalebox{#1}{$\mathsurround=0pt #2{#3}$}}
\def\bigger#1#2{{\relax\mathpalette{\makebigger{#1}}{#2}}}

\def\scaleuphalf{1.0954}

\newcommand{\op}{^{\mathord{\text{\rm op}}}}

\newcommand{\coop}{^{\mathord{\text{\rm coop}}}}

\renewcommand{\th}{^{\text{th}}}

\newcommand{\defeq}{\mathrel{:=}}

\makeatletter

\def\newmop{\@ifstar{\@newmop m}{\@newmop o}}
\def\@newmop#1{\@ifnextchar[{\@@newmop #1}{\@@@newmop #1}}
\def\@@newmop#1[#2]{\@declmathop #1#2}
\def\@@@newmop#1#2{\expandafter\@declmathop\expandafter #1\csname #2\endcsname{#2}}

\makeatother

\newdir{ |}{{}*!/-5pt/@{|}}

\newmop{im}
\newmop{coim}
\newmop{dom}
\newmop{cod}
\newmop{id}
\newmop{Map}

\newmop{obj}
\newmop{arr}
\newmop{sq}
\newmop{norm}

\newmop{el}

\newmop{ev}

\newcommand{\h}{\mathfrak{h}}

\newcommand{\qop}[1]{{\mathord{\text{\normalfont{\textsf{#1}}}}}}
\newcommand{\Fun}{\qop{Fun}}

\newmop{Ext}
\newmop{icon}
\newmop{pbk}
\newmop{icom}

\newmop{cell}
\newmop{cof}
\newmop{fib}

\newmop{diag}

\newmop*{colim}
\newmop*{holim}

\newmop[\lan]{lan}
\newmop[\ran]{ran}

\newcommand{\comma}{\mathbin{\downarrow}}

\makeatletter
\newcommand{\rotatemath}[2]{\rotatebox[origin=c]{180}{$\m@th #1{#2}$}}
\makeatother

\newmop[\pr]{pr}
\newmop[\dm]{d} 

\newmop[\cpr]{in}

\newcommand{\pocorner}{\hbox to 8pt{{\vrule height8pt depth0pt width0.5pt}%
    \vbox to 8pt{{\hrule height0.5pt width7.5pt depth0pt}\vfill}}}

\newcommand{\pbcorner}{\vbox to 0pt{\kern 4pt\hbox to 0pt{\kern 4pt%
      \vbox{{\hrule height0.5pt width7.5pt depth0pt}}%
      {\vrule height8pt depth0pt width0.5pt}\hss}\vss}}
\newcommand{\pbexcursion}{\save[]+<5pt,-5pt>*{\pbcorner}\restore}
\newcommand{\pbdiamond}{\save[]+<0pt,-5pt>*{\rotatebox{-45}{$\pbcorner$}}\restore}

\newcommand{\pwr}{\pitchfork}

\newmop{cls}

\newcommand{\leib}[1]{\mathbin{\widehat{#1}}}

\newcommand{\pbtimes}[1]{\mathbin{\mathop{\times}_{#1}}}

\newcommand{\category}[1]{\underline{\smash[b]{\text{\rm{#1}}}}}

\newcommand{\lcat}{\mathcal}

\newcommand{\catthree}{{\bigger{1.12}{\mathbbe{3}}}}
\newcommand{\cattwo}{{\bigger{1.12}{\mathbbe{2}}}}
\newcommand{\catone}{{\bigger{1.16}{\mathbbe{1}}}}
\newcommand{\iso}{{\mathbb{I}}}

\makeatletter

\def\Del@Sym{{\bigger\scaleuphalf{\mathbbe{\Delta}}}}

\def\del@fn{\futurelet\del@next}
\def\del@dn{\def\del@next}

\def\parsedel@{%
  \ifx +\del@next \del@dn+{\Del@Sym_{\mathord{+}}}%
  \else \del@dn {\del@fn\parsedel@@}%
  \fi\del@next}

\def\parsedel@@{%
  \ifx\space@\del@next \expandafter\del@dn\space{\del@fn\parsedel@@}%
  \else\ifx [\del@next \del@dn[{\del@fn\parsedel@@@}%
  \else\ifx _\del@next \del@dn{\Delta}%
  \else\ifx ^\del@next \del@dn{\Delta}%
  \else \del@dn{\Del@Sym}%
  \fi\fi\fi\fi\del@next}

\def\parsedel@@@{%
  \ifx\space@\del@next \expandafter\del@dn\space{\del@fn\parsedel@@@}%
  \else\ifx t\del@next \del@dn t{\Del@Sym_{\top}\del@fn\parsedel@@@@}%
  \else\ifx b\del@next \del@dn b{\Del@Sym_{\bot}\del@fn\parsedel@@@@}%
  \else \del@dn{\errmessage{unexpected modifier}}%
  \fi\fi\fi\del@next}

\def\parsedel@@@@{%
  \ifx\space@\del@next \expandafter\del@dn\space{\del@fn\parsedel@@@@}%
  \else\ifx ]\del@next \del@dn]{}%
  \else \del@dn{\errmessage{expecting close of option block}}%
  \fi\fi\del@next}

\def\Del{\del@fn\parsedel@}

\makeatother

\newcommand{\Horn}{\Lambda}

\newcommand{\Cat}{\category{Cat}}

\newcommand{\sSet}{\category{sSet}}
\newcommand{\qCat}{\category{qCat}}
\newcommand{\Segal}{\category{Segal}}
\newcommand{\CSS}{\category{CSS}}

\newcommand{\MMod}[1]{\category{Mod}_{\lcat{#1}}}

\newcommand{\dmod}[3]{\xymatrix@=1.25em{{#2} \ar[r]|\mid^{ {#1}} & {#3}}}

\newcommand{\face}{\delta}

\newcommand{\degen}{\sigma}

\newcommand{\join}{\star}

\newmop{dec}
\newmop{fatdec}
\newmop{slc}
\newmop{fatslc}

\newcommand{\slice}{/}

\newcommand{\slicer}[2]{{#1\mkern-1mu}_{{}\slice{#2}}}

\newmop{ir} 
\newmop{incl}

\newcommand{\ho}{h}

\newcommand{\boundary}{\partial}

\def\reedyfilt#1_#2{#1_{\leq #2}}
\newmop{sk}
\newmop{cosk}
\newmop{res}

\newmop{map}
\newmop{Ho}

\newmop[\pth]{path}
\newmop{cyl}

\newmop[\const]{c}
  
\newmop{coll}
\newmop{wgt}

\newdir{ >}{{}*!/-7pt/@{>}}
\newdir{u(}{{}*!/-4pt/@^{(}}
\newdir{d(}{{}*!/-4pt/@_{(}}
\newdir{|>}{%
  !/4.5pt/@{|}*:(1,-.2)@^{>}*:(1,+.2)@_{>}*+@{}}

\makeatletter
\def\makeslashed#1#2#3#4#5{#1{\mathpalette{\sla@{#2}{#3}{#4}}{#5}}}

\def\@mathlower#1#2#3{\setbox0=\hbox{$\m@th#2#3$}\lower#1\ht0\box0}
\def\mathlower#1#2{\mathpalette{\@mathlower{#1}}{#2}}
\makeatother

\newcommand{\inc}{\hookrightarrow}

\newcommand{\tfib}{\twoheadrightarrow}

\newcommand{\xtfib}[1]{\xtwoheadrightarrow{#1}}

\newcommand{\longtwoheadrightarrow}{\mathrel{\mathord{-}\mkern-3mu\mathord\twoheadrightarrow}}

\newcommand{\we}{\xrightarrow{\mkern10mu{\smash{\mathlower{0.6}{\sim}}}\mkern10mu}}

\newcommand{\trvfib}{\stackrel{\smash{\mkern-2mu\mathlower{1.5}{\sim}}}\longtwoheadrightarrow}

\newcommand{\To}{\Rightarrow}

\makeatletter

\def\tens@fn{\futurelet\tens@next}
\def\tens@dn{\def\tens@nextcont}
\newtoks\tens@toks
\def\addtotens@toks#1{\tens@toks=\expandafter{\the\tens@toks#1}}

\def\parsetens@@{%
    \ifx\space@\tens@next \expandafter\tens@dn\space{\tens@fn\parsetens@@}%
    \else\ifx ^\tens@next \tens@dn ^##1{\parsetens@procsep^\addtotens@toks{##1}%
      \tens@fn\parsetens@@}%
    \else\ifx _\tens@next \tens@dn _##1{\parsetens@procsep_\addtotens@toks{##1}%
      \tens@fn\parsetens@@}%
    \else\tens@dn{\ifx *\tens@last \else\addtotens@toks\egroup\fi\the\tens@toks}%
    \fi\fi\fi\tens@nextcont}

\def\parsetens@procsep#1{%
  \ifx *\tens@last \addtotens@toks{#1}\addtotens@toks\bgroup%
  \else\ifx \tens@last\tens@next \addtotens@toks,%
  \else \addtotens@toks\egroup\addtotens@toks\bgroup%
    \addtotens@toks\egroup\addtotens@toks{#1}\addtotens@toks\bgroup%
  \fi\fi\let\tens@last\tens@next}

\newcommand{\tn}[1]{\let\tens@last=*\tens@toks={#1}\tens@fn\parsetens@@}

\makeatother

\def\adjdisplay#1-|#2:#3->#4.{{%
    \xymatrix@R=0em@!C=2.5em{%
      *+[l]{#3} \ar@/_0.55pc/[rr]_-{#2} & {\bot} &
      *+[r]{#4}\ar@/_0.55pc/[ll]_-{#1}}}}

\def\adjdisplaytwo#1-|#2:#3->#4.{{%
\xymatrix@=1.2em{
      {#3}\ar@/_1.5ex/[rr]_-{#2}^-{}="one"
      & & {#4}
      \ar@/_1.5ex/[ll]_-{#1}^-{}="two" 
      \ar@{}"one";"two"|{\bot}
    }}}

\def\tripleadjdisplay#1-|#2-|#3:#4->#5.{{%
\xymatrix@=2.4em{ 
{#4}\ar[r]|{#2} &
{#5} \ar@/_3ex/[l]_{#1}^{\bot} \ar@/^3ex/[l]_{\bot}^{#3}}
}}

\def\adjinline#1-|#2:#3->#4.{{#1}\dashv{#2}:#3\to #4}

\newcommand{\pent}[1]{
  \xybox{
    \POS (0,-15)*+{\a}="0", 
         (-14,-5)*+{\b}="1", 
         (-9,12)*+{\c}="2", 
         (9,12)*+{\d}="3", 
         (14,-5)*+{\e}="4"
    \POS"0" \ar "1"^{\labelstyle \ab}|{}="01"
    \POS"1" \ar "2"^{\labelstyle \bc}|{}="12"
    \POS"2" \ar "3"^{\labelstyle \cd}|{}="23"
    \POS"3" \ar "4"^{\labelstyle \de}|{}="34"
    \POS"0" \ar "4"_{\labelstyle \ae}|{}="04"
    \ifcase #1
    \POS"0" \ar "2"|{\labelstyle \ac}="02"
    \POS"0" \ar "3"|{\labelstyle \ad}="03"
    \POS"02";"1"**{}, ?(0.3) \ar@{=>} ?(0.7)^{\labelstyle \abc}
    \POS"03";"2"**{}, ?(0.25) \ar@{=>} ?(0.5)_{\labelstyle \acd}
    \POS"04";"3"**{}, ?(0.2) \ar@{=>} ?(0.4)_{\labelstyle \ade}
    \or
    \POS"1" \ar "3"|{\labelstyle \bd}="13"
    \POS"1" \ar "4"|{\labelstyle \be}="14"
    \POS"13";"2"**{}, ?(0.3) \ar@{=>} ?(0.7)_{\labelstyle \bcd}
    \POS"14";"3"**{}, ?(0.25) \ar@{=>} ?(0.5)_{\labelstyle \bde}
    \POS"04";"1"**{}, ?(0.25) \ar@{=>} ?(0.5)_{\labelstyle \abe}
    \or
    \POS"2" \ar "4"|{\labelstyle \ce}="24"
    \POS"0" \ar "2"|{\labelstyle \ac}="02"
    \POS"02";"1"**{}, ?(0.3) \ar@{=>} ?(0.7)^{\labelstyle \abc}
    \POS"04";"2"**{}, ?(0.2) \ar@{=>} ?(0.35)_{\labelstyle \ace}
    \POS"24";"3"**{}, ?(0.2) \ar@{=>} ?(0.6)^{\labelstyle \cde}
    \or
    \POS"1" \ar "3"|{\labelstyle \bd}="13"
    \POS"0" \ar "3"|{\labelstyle \ad}="03"
    \POS"04";"3"**{}, ?(0.2) \ar@{=>} ?(0.4)_{\labelstyle \ade}
    \POS"13";"2"**{}, ?(0.3) \ar@{=>} ?(0.7)_{\labelstyle \bcd}
    \POS"03";"1"**{}, ?(0.25) \ar@{=>} ?(0.5)^{\labelstyle \abd}
    \or
    \POS"2" \ar "4"|{\labelstyle \ce}="24"
    \POS"1" \ar "4"|{\labelstyle \be}="14"
    \POS"24";"3"**{}, ?(0.2) \ar@{=>} ?(0.6)^{\labelstyle \cde}
    \POS"04";"1"**{}, ?(0.25) \ar@{=>} ?(0.5)_{\labelstyle \abe}
    \POS"14";"2"**{}, ?(0.25) \ar@{=>} ?(0.5)^{\labelstyle \bce}
    \else\fi
  }
}

\newcommand{\pentofpent}[1]{
  \def\baselen{#1}
  \begin{xy}
    0;<\baselen,0mm>:
    *{\xybox{
        \POS(0,-4)*[o]{\pent 0}="zero"
        \POS(16,40)*[o]{\pent 3}="three"
        \POS(72,40)*[o]{\pent 1}="one"
        \POS(88,-4)*[o]{\pent 4}="four"
        \POS(44,-36)*[o]{\pent 2}="two"
        \ar@<1ex>"zero";"three"^-{\objectstyle\abcd}
        \ar@<1ex>"three";"one"^-{\objectstyle\abde}
        \ar@<1ex>"one";"four"^-{\objectstyle\bcde}
        \ar@<-1ex>"zero";"two"_-{\objectstyle\acde}
        \ar@<-1ex>"two";"four"_-{\objectstyle\abce}
        \ar@{=>}(44,-5);(44,+15)^{\objectstyle\abcde}
     }}
  \end{xy}
}

\setlist{}
\setenumerate{leftmargin=*,labelindent=\parindent}
\setitemize{leftmargin=*,labelindent=0.5\parindent}
\setdescription{leftmargin=1em}

\swapnumbers
	
\theoremstyle{plain}
\newtheorem{thm}{Theorem}[subsection]
\newtheorem{lem}[thm]{Lemma}
\newtheorem{cor}[thm]{Corollary}

\newtheorem{prop}[thm]{Proposition}

\theoremstyle{definition}
\newtheorem{defn}[thm]{Definition}
\newtheorem{ex}[thm]{Example}

\theoremstyle{remark}
\newtheorem{obs}[thm]{Observation}
\newtheorem{rec}[thm]{Recall}
\newtheorem{rmk}[thm]{Remark}
\newtheorem{exs}[thm]{Exercise}

\makeatletter
\let\c@equation\c@thm
\makeatother
\numberwithin{equation}{subsection}

\renewcommand{\Cat}{\mathcal{C}\mathrm{at}}
\renewcommand{\sSet}{\mathrm{s}\mathcal{S}\mathrm{et}}
\renewcommand{\qCat}{\mathrm{q}\mathcal{C}\mathrm{at}}
\renewcommand{\ho}{\mathrm{ho}}
\renewcommand{\CSS}{\mathcal{C}\mathrm{SS}}
\renewcommand{\Segal}{\mathcal{S}\mathrm{egal}}
\newcommand{\Sp}{\mathcal{S}\mathrm{p}}
\newcommand{\Rezk}{\mathcal{R}\mathrm{ezk}}

\usepackage{xr}

\newcommand{\extRef}[3]{%
  {\protect\IfBeginWith{#3}{itm:}{}{#2.}}\ref*{#1:#3}}

\newcommand{\refI}{\extRef{found}{I}}
\newcommand{\refII}{\extRef{cohadj}{II}}

\newcommand{\refIV}{\extRef{yoneda}{IV}}
\newcommand{\refV}{\extRef{equipment}{V}}

\title{Infinity category theory from scratch}

\author[Riehl]{Emily Riehl}
\address{
  Department of Mathematics \\
Johns Hopkins University \\
Baltimore, MD 21218\\
  USA
}
\email{eriehl@math.jhu.edu}
\thanks{\textit{Acknowledgments:} The first-named author wishes to thank the organizers of the Young Topologists' Meeting 2015 for providing her with the opportunity to speak about this topic; the Hausdorff Research Institute for Mathematics, which hosted her during the period in which these lectures were prepared; and the Simons Foundation and National Science Foundation for financial support through the AMS-Simons Travel Grant program and DMS-1509016. Edoardo Lanari pointed out a circularity in the interpretation of the axiomatization for an $\infty$-cosmos with all objects cofibrant, now corrected. The authors are also grateful for insightful suggestions from the referee that were incorporated into the published manuscript.}

\author[Verity]{Dominic Verity}
\address{
  Centre of Australian Category Theory \\
  Macquarie University \\
  NSW 2109 \\
  Australia
}
\email{dominic.verity@mq.edu.au}

\date{\today}

\subjclass[2010]{%
  18G55, 55U35
}

\begin{document}
  
  \ifpdf
  \DeclareGraphicsExtensions{.pdf, .jpg, .tif}
  \else
  \DeclareGraphicsExtensions{.eps, .jpg}
  \fi
  
\begin{abstract} 
These lecture notes were written to accompany a mini course given at the 2015 Young Topologists' Meeting at \'{E}cole Polytechnique F\'{e}d\'{e}rale de Lausanne, videos of which can be found at  \href{http://hessbellwald-lab.epfl.ch/ytm2015}{hessbellwald-lab.epfl.ch/ytm2015}. We use the terms $\infty$-\emph{categories} and $\infty$-\emph{functors} to mean the objects and morphisms in an $\infty$-\emph{cosmos}: a simplicially enriched category satisfying a few axioms, reminiscent of an enriched category of ``fibrant objects.'' Quasi-categories, Segal categories, complete Segal spaces, naturally marked simplicial sets, iterated complete Segal spaces, $\theta_n$-spaces, and fibered versions of each of these are all $\infty$-categories in this sense. We show that the basic category theory of $\infty$-categories and $\infty$-functors can be developed from the axioms of an $\infty$-cosmos; indeed, most of the work is internal to a strict 2-category of $\infty$-categories, $\infty$-functors, and natural transformations. In the $\infty$-cosmos of quasi-categories, we recapture precisely the same category theory developed by Joyal and Lurie, although in most cases our definitions, which are 2-categorical rather than combinatorial in nature, present a new incarnation of the standard concepts.

In the first lecture, we define an $\infty$-cosmos and introduce its \emph{homotopy 2-category}, the strict 2-category mentioned above. We illustrate the use of formal category theory to develop the basic theory of equivalences of and adjunctions between $\infty$-categories. In the second lecture, we study limits and colimits of diagrams taking values in an $\infty$-category and relate these concepts to adjunctions between $\infty$-categories. In the third lecture, we define comma $\infty$-categories, which satisfy a particular weak 2-dimensional universal property in the homotopy 2-category. We illustrate the use of comma $\infty$-categories to encode the universal properties of (co)limits and adjointness. Because comma $\infty$-categories are preserved by all functors of $\infty$-cosmoi and created by certain weak equivalences of $\infty$-cosmoi, these characterizations form the foundations for ``model independence'' results. In the fourth lecture, we introduce (co)cartesian fibrations, a certain class of $\infty$-functors, and also consider the special case with groupoidal fibers. We then describe the calculus of \emph{modules} between $\infty$-categories ---  comma $\infty$-categories being the prototypical example --- and use this framework to introduce the Yoneda lemma and develop the theory of pointwise Kan extensions of $\infty$-functors.
\end{abstract}
  
  \maketitle

\tableofcontents

\renewcommand\thesection{Lecture~\arabic{section}}
\section{Basic 2-category theory in an \texorpdfstring{$\infty$}{infinity}-cosmos}\label{sec:2-cat}
\renewcommand\thesection{\arabic{section}}

The goal of these lectures is to give a very precise account of the foundations of $\infty$-category theory, paralleling the development of ordinary category theory.

To explain what we  mean by the term ``$\infty$-categories,'' it is useful to distinguish between schematic definitions of infinite-dimensional categories and models of infinite-dimensional categories. For instance, $(\infty,1)$-\emph{category} is a schematic notion, describing a category that is weakly enriched in $\infty$-\emph{groupoids}, itself a schematic term for something like topological spaces. The main source of ambiguity in this definition is meaning of ``weakly enriched,'' which is not precisely defined.

Models of $(\infty,1)$-categories are precise mathematical objects meant to embody this schema. These include \emph{quasi-categories} (originally called \emph{weak Kan complexes}), \emph{complete Segal spaces}, \emph{Segal categories}, and \emph{naturally marked simplicial sets} (\emph{1-trivial saturated weak complicial sets}), among others.

To develop the theory of $(\infty,1)$-categories in a non-schematic way, it makes sense to choose a model. Andr\'e Joyal pioneered and Jacob Lurie extended a wildly successful project to extend basic category theory from ordinary categories to quasi-categories. A natural question is then: does this work extend to other models of $(\infty,1)$-categories? And to what extent are basic categorical notions invariant under change of models? 

The general consensus is that the choice of model should not matter so much, but one obstacle to proving results of this kind is that, to a large extent, precise versions of the categorical definitions that have been established for quasi-categories had not been given for the other models. 

Here we will use the term $\infty$-\emph{category} to refer to any of the models of $(\infty,1)$-categories listed above: quasi-categories, complete Segal spaces, Segal categories, or naturally marked simplicial sets.\footnote{Later, we will explain that the results we prove about $\infty$-\emph{categories} also apply to  other varieties of infinite-dimensional categories, but let us not get into this now.} Note that, for us ``$\infty$-category'' will not function as a schematic term: several well-behaved models of $(\infty,1)$-categories will be included in the scope of its meaning, but others are excluded.  With this interpretation in mind, let us describe a few features of our project to extend the basic theory of categories to $\infty$-categories.
\begin{itemize}
\item It is blind to which model of $\infty$-categories is being considered.
\end{itemize}
That is, our definitions of basic categorical notions will be stated and our theorems will be proven without reference to particular features of any model of $\infty$-categories and will apply, simultaneously, to all of them. 
\begin{itemize}
\item It is compatible with the Joyal/Lurie theory of quasi-categories.
\end{itemize}
In the special case of quasi-categories, our presentation of the basic categorical notions will necessarily differ from existing one, as our formalism is unaware of the fact that quasi-categories are simplicial sets. Nonetheless the categorical concepts so-defined are precisely equivalent to the Joyal/Lurie theory, and so our approach can be mixed with the existing one. This is good news because, at present, the Joyal/Lurie theory of quasi-categories is considerably more expansive.\footnote{But we are not finished yet.}
\begin{itemize}
\item It is invariant under change of models between $(\infty,1)$-categories.
\end{itemize}
Independent work of Bertrand T\"{o}en and of Clark Barwick and Chris Schommer-Pries proves that all models of $(\infty,1)$-categories ``have the same homotopy theory,'' in the sense of being connected by a zig-zag of Quillen equivalences of model categories  \cite{Toen:2005vu}  or having equivalent quasi-categories  \cite{BSP:2011ot}. In lectures 3 and 4, we will see that the models of $(\infty,1)$-categories that fall within the scope of this theory all ``have the same basic category theory,'' in the sense that the answers to basic categorical questions are invariant under change of model. Finally:
\begin{itemize}
\item It is as simple as possible.
\end{itemize}
Our original aim in this project was to provide a streamlined, from the ground up, development of the foundations of quasi-category theory that more closely parallels classical category theory and is relatively easy to work with, at least for those with some affinity for abstract nonsense. With a bit of effort, these ideas could be extended to cover  a wider variety of models of $(\infty,1)$-categories, but to do so would come at the cost of complicating the proofs. 

The plan for the first lecture is to introduce the axiomatic framework that makes it possible to develop the basic theory of $\infty$-categories without knowing precisely which variety of infinite-dimensional categories we are working with. In the second installment, we will continue with a discussion of limits and colimits of diagrams valued inside an $\infty$-category, which we relate to the adjunctions between $\infty$-categories that we will define now. In the third lecture, we will introduce \emph{comma $\infty$-categories}, which will be the essential ingredient in our proofs that basic category theory  is invariant under change of models. In the first two lectures, we will see how basic categorical notions can be developed ``model agnostically,'' i.e., simultaneously for all models. In lecture three, we will be able to prove that this theory is both preserved and reflected by certain change-of-model functors. Finally, in the final talk we will introduce \emph{modules} between $\infty$-categories, which we use to develop the theory of pointwise Kan extensions.

The work to date on this project can be found in a series of five papers:
\begin{enumerate}[label=\Roman*.]
\item The 2-category theory of quasi-categories \cite{RiehlVerity:2012tt}, \href{http://arxiv.org/abs/1306.5144}{ar\-Xiv:1306.5144}
\item Homotopy coherent adjunctions and the formal theory of monads \cite{RiehlVerity:2012hc},\\ \href{http://arxiv.org/abs/1310.8279}{arXiv:1310.8279}
\item Completeness results for quasi-categories of algebras, homotopy limits, and related general constructions \cite{RiehlVerity:2013cp}, \href{http://arxiv.org/abs/1401.6247}{arXiv:1401.6247}
\item Fibrations and Yoneda's lemma in an $\infty$-cosmos \cite{RiehlVerity:2015fy},  \href{http://arxiv.org/abs/1506.05500}{arXiv:1506.05500}
\item Kan extensions and the calculus of modules for $\infty$-categories \cite{RiehlVerity:2015ke}, \\ \href{http://arxiv.org/abs/1507.01460}{arXiv:1507.01460}
\end{enumerate}
references within which will have the form, e.g., \refII{thm:monadiccomparisonadj}.

\subsection{Formal category theory in a 2-category}

Recall the term $\infty$-\emph{categories} refers to the reader's choice of one of several listed models of infinite-dimensional categories. Each fixed model of $\infty$-categories has an accompanying notion of $\infty$-\emph{functors} between them. The reason we are able to develop the basic theory of $\infty$-categories without knowing exactly what the $\infty$-categories are is that 
the $\infty$-categories and $\infty$-functors of each type are packaged inside an axiomatic framework that describes the surrounding category with these objects and morphisms. We will say more about this in a moment, but first we want to illustrate how closely our development of the theory of $\infty$-categories parallels the development of the theory of ordinary categories.  One thing that is provided by the axiomatization that we will describe in a moment is a reasonable notion of $\infty$-\emph{natural transformations} between $\infty$-functors.

Categories, functors, and natural transformations naturally assemble into a \emph{2-category} $\h\Cat$, a structure which records the various ways in which functors and natural transformations can be composed. Similarly, we will see that the $\infty$-categories of each fixed model, the $\infty$-functors between them, and the $\infty$-natural transformations between $\infty$-functors also assemble into a 2-category, which we call the \emph{homotopy 2-category}. Each variety of $\infty$-categories will have their own homotopy 2-category --- one example being $\h\Cat$. A \emph{homotopy 2-category} is a strict 2-category whose: 
\begin{itemize} 
\item objects $A$, $B$, and $C$ are \emph{$\infty$-categories};
\item 1-cells $f \colon A \to B$, $g \colon B \to C$, are ($\infty$-)\emph{functors};
\item 2-cells $\xymatrix{ A \ar@/^2ex/[r]^f \ar@/_2ex/[r]_g \ar@{}[r]|{\Downarrow\alpha}& B,}$ written in line as $\alpha \colon f \To g \colon A \to B$, are ($\infty$-)\emph{natural transformations}.
\end{itemize}

\begin{rec}[pasting diagrams in 2-categories]
The objects and 1-cells in a 2-category define its \emph{underlying 1-category}. In particular, the 1-cells have the familiar associative and unital composition law. The 2-cells can be composed in two ways:
\begin{itemize}
\item The \emph{vertical composite} of  $\vcenter{\xymatrix@R=30pt{  A \ar@/^2ex/[r]^f \ar@/_2ex/[r]_g \ar@{}[r]|{\Downarrow\alpha}& B \\  A \ar@/^2ex/[r]^g \ar@/_2ex/[r]_h \ar@{}[r]|{\Downarrow\beta}& B}}$ defines a 2-cell $\xymatrix{  A \ar@/^2ex/[r]^f \ar@/_2ex/[r]_h \ar@{}[r]|{\Downarrow\beta\cdot\alpha}& B}$.

\item The \emph{horizontal composite} of  $ \xymatrix{  A \ar@/^2ex/[r]^f \ar@/_2ex/[r]_g \ar@{}[r]|{\Downarrow\alpha}& B \ar@/^2ex/[r]^h \ar@/_2ex/[r]_k \ar@{}[r]|{\Downarrow\gamma} & C}$  defines a 2-cell $\xymatrix{ A \ar@/^2ex/[r]^{hf} \ar@/_2ex/[r]_{kg} \ar@{}[r]|{\Downarrow\gamma\alpha}& C}$.
\end{itemize}
A degenerate special case of horizontal composition, in which all but one of the 2-cells is an identity $\id_f$ on its boundary 1-cell $f$, is called \emph{whiskering}: 
\[ \vcenter{\xymatrix{  A \ar@/^2ex/[r]^f \ar@/_2ex/[r]_f \ar@{}[r]|{\Downarrow\id_f}& B \ar@/^2ex/[r]^h \ar@/_2ex/[r]_k \ar@{}[r]|{\Downarrow\gamma} & C \ar@/^2ex/[r]^\ell \ar@/_2ex/[r]_\ell \ar@{}[r]|{\Downarrow\id_{\ell}} & D} } =  \vcenter{\xymatrix{  A \ar[r]^f & B \ar@/^2ex/[r]^h \ar@/_2ex/[r]_k \ar@{}[r]|{\Downarrow\gamma} & C \ar[r]^\ell & D} } =  \vcenter{\xymatrix{   A \ar@/^2ex/[r]^{\ell hf} \ar@/_2ex/[r]_{\ell kf} \ar@{}[r]|{\Downarrow\ell \gamma f} & D} }\]

Vertical composition is associative and unital with respect to identity 2-cells:  there is a category of 1-cells and 2-cells between each fixed pair of objects. Horizontal composition is associative and unital with respect to those identity 2-cells $\id_{\id_A}$ on identity 1-cells. Moreover, the vertical and horizontal composition operations commute by the law of \emph{middle four interchange} which says that any \emph{pasting diagram}, examples of which are displayed below, has a unique composite 2-cell.
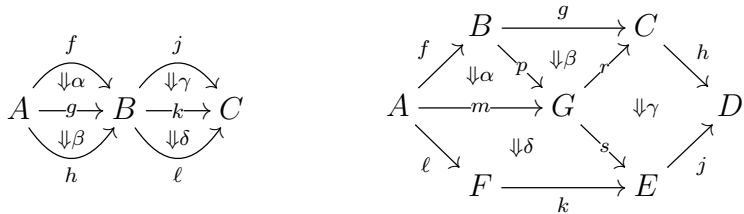
\[ \vcenter{\xymatrix{ A \ar@/^3.5ex/[r]^f_{\Downarrow\alpha} \ar[r]|g \ar@/_3.5ex/[r]_h^{\Downarrow\beta} & B  \ar@/^3.5ex/[r]^j_{\Downarrow\gamma} \ar[r]|k \ar@/_3.5ex/[r]_\ell^{\Downarrow\delta} & C}}
\qquad\qquad \vcenter{\xymatrix@=15pt{  & B \ar[rr]^g \ar@{}[d]|(.6){\Downarrow\alpha} \ar[dr]|p & \ar@{}[d]|(.4){\Downarrow\beta} & C \ar@{}[dd]|{\Downarrow\gamma} \ar[dr]^h \\ A \ar[ur]^f \ar[dr]_\ell \ar[rr]|m & \ar@{}[dr]|{\Downarrow\delta} & G \ar[ur]|r \ar[dr]|s & & D \\ & F \ar[rr]_k & & E \ar[ur]_j}}
\]
\end{rec}

Now let us explore how such a framework can be used to develop the basic category theory of the $\infty$-categories that define the objects of the homotopy 2-category. For instance, the homotopy 2-category provides a convenient framework in which to define the notion of adjunction between $\infty$-categories.

\begin{defn}\label{defn:adjunction} An \emph{adjunction} between $\infty$-categories consists of:
\begin{itemize}
\item a pair of $\infty$-categories $A$ and $B$;
\item a pair of functors $f \colon B \to A$ and $u \colon A \to B$; and
\item a pair of natural transformations $\eta \colon \id_B \To uf$ and $\epsilon \colon fu \To \id_A$
\end{itemize}
so that the triangle identities hold:
\[\xymatrix@=1.5em{ & B \ar[dr]|f \ar@{}[d]|(.6){\Downarrow\epsilon} \ar@{=}[rr] &  \ar@{}[d]|(.4){\Downarrow\eta} & B \ar@{}[d]^*+{=} & B &&   B \ar@{=}[rr] \ar[dr]_f & \ar@{}[d]|(.4){\Downarrow \eta} & B \ar[dr]^f \ar@{}[d]|(.6){\Downarrow\epsilon} & {\mkern40mu}\ar@{}[d]^*+{=} &  B \ar@/^2ex/[d]^f \ar@/_2ex/[d]_f \ar@{}[d]|(.4){\id_f}|(.6){\Leftarrow} & \\A \ar[ur]^u \ar@{=}[rr] & &  A \ar[ur]_u & {\mkern40mu} & A \ar@/^2ex/[u]^u \ar@/_2ex/[u]_u \ar@{}[u]|(.4){\Rightarrow}|(.6){\id_u}  && &A \ar[ur]|u \ar@{=}[rr] & & A & A }\] The left-hand equality of pasting diagrams asserts that $u\epsilon \cdot \eta u = \id_u$, while the right-hand equality asserts that $\epsilon f \cdot f\eta = \id_f$.
\end{defn}

We write $f \dashv u$ to assert that the functor $f \colon B \to A$ is \emph{left adjoint} to the functor $u \colon A \to B$, its \emph{right adjoint}. From the standpoint of this definition, we can easily prove some basic properties of adjunctions. Note that in the special case of the homotopy 2-category $\h\Cat$, these proofs are exactly the familiar ones.

\begin{prop}\label{prop:adjunctions-compose} Adjunctions compose: given adjoint functors
\[ \vcenter{\xymatrix{ C \ar@<1ex>[r]^{f'} \ar@{}[r]|\perp & B \ar@<1ex>[r]^f \ar@<1ex>[l]^{u'} \ar@{}[r]|\perp & A  \ar@<1ex>[l]^u}} \qquad \rightsquigarrow\qquad  \vcenter{\xymatrix{ C \ar@<1ex>[r]^{ff'} \ar@{}[r]|\perp & A \ar@<1ex>[l]^{u'u}}}\] the composite functors are adjoint.
\end{prop}
\begin{proof}
Writing $\eta \colon \id_B \To uf$, $\epsilon \colon fu \To \id_A$, $\eta' \colon \id_C \To u'f'$, and $\epsilon' \colon fu \To \id_B$ for the respective units and counits, the pasting diagrams 
\[ \xymatrix@=1.5em{  C \ar[dr]_{f'} \ar@{=}[rrrr] & & \ar@{}[d]|{\Downarrow \eta'} & & C & & C \ar@{}[d]|(.6){\Downarrow\epsilon'} \ar[dr]^{f'} \\ & B \ar@{=}[rr] \ar[dr]_f & \ar@{}[d]|(.4){\Downarrow\eta} & B \ar[ur]_{u'} & & B \ar[ur]^{u'} \ar@{=}[rr] & \ar@{}[d]|{\Downarrow\epsilon} & B \ar[dr]^f \\ & & A \ar[ur]_u & & A \ar[ur]^u \ar@{=}[rrrr] & & & & A}\] define the unit and counit of $ff' \dashv u'u$ so that the triangle identities
\[ \xymatrix@=.70em{  C \ar[dr]_{f'} \ar@{=}[rrrr] & & \ar@{}[d]|{\Downarrow \eta'} & & C  \ar@{}[d]|(.6){\Downarrow\epsilon'} \ar[dr]^{f'}& & && C   \ar@/^3ex/[dd]^(.3){ff'} \ar@/_3ex/[dd]_(.3){ff'} \ar@{}[dd]|(.4){\id_{ff'}}|(.6){\Leftarrow} &   &&  & C \ar@{}[d]|(.6){\Downarrow\epsilon'} \ar[dr]|{f'} \ar@{=}[rrrr] & & \ar@{}[d]|{\Downarrow \eta'} & & C & & C  \\ & B \ar@{=}[rr] \ar[dr]_f & \ar@{}[d]|(.4){\Downarrow\eta} & B \ar[ur]|{u'}  \ar@{=}[rr] & \ar@{}[d]|{\Downarrow\epsilon} & B \ar[dr]^f & & \ar@{}[l]|{\displaystyle =} & &&   & B \ar[ur]^{u'} \ar@{=}[rr] & \ar@{}[d]|{\Downarrow\epsilon} & B \ar[dr]|f   \ar@{=}[rr] & \ar@{}[d]|(.4){\Downarrow\eta} & B \ar[ur]_{u'} & & \ar@{}[l]|{\displaystyle =}  \\ & & A \ar[ur]|u  \ar@{=}[rrrr] & & & & A & &  A & &   A \ar[ur]^u \ar@{=}[rrrr] & & & & A \ar[ur]_u & & & &  A \ar@/^3ex/[uu]^(.3){u'u} \ar@/_3ex/[uu]_(.3){u'u} \ar@{}[uu]|(.6){\id_{u'u}}|(.4){\Rightarrow} & &  } \] 
hold.
\end{proof}

It is also straightforward to show that the existence of adjoint functors is equivalence-invariant, in the sense of the standard 2-categorical notion of equivalence.

\begin{defn}\label{defn:equivalence} An \emph{equivalence} between $\infty$-categories consists of:
\begin{itemize}
\item a pair of $\infty$-categories $A$ and $B$;
\item a pair of functors $f \colon B \to A$ and $g \colon A \to B$; and
\item a pair of natural isomorphisms  $\eta \colon \id_B \cong gf$ and $\epsilon \colon fg \cong \id_A$.
\end{itemize}
An ($\infty$-)\emph{natural isomorphism} is a 2-cell in the homotopy 2-category that admits a vertical inverse 2-cell.
\end{defn}

We write $A \simeq B$ and say that $A$ and $B$ are \emph{equivalent} if there exists an equivalence between $A$ and $B$. The direction for the natural isomorphisms comprising an equivalence is immaterial. Our notation is chosen to suggest the connection with adjunctions conveyed by the following exercise.

\begin{exs}\label{exs:adjoint-equivalence} Show that an equivalence, defined in any 2-category, can always be promoted to an \emph{adjoint equivalence} by modifying one of the 2-cell isomorphisms. That is, show that the 2-cell isomorphisms in an equivalence can be chosen so as to satisfy the triangle identities.
\end{exs}

Combining Exercise \ref{exs:adjoint-equivalence} with the symmetry in the definition of an equivalence, we have:

\begin{cor}\label{cor:equivs-are-adjoints} Any functor $f \colon A \to B$ that defines an equivalence of $\infty$-categories admits both a left and a right adjoint.
\end{cor}

\begin{prop}\label{prop:equiv-invar-adjunction} If $\xymatrix{ B \ar@<1ex>[r]^f \ar@{}[r]|\perp & A \ar@<1ex>[l]^u}$ is an adjunction and $A \simeq A'$ and $B \simeq B'$ are any equivalences, then the equivalent functors $\xymatrix{ B' \ar@<1ex>[r]^{f'} \ar@{}[r]|\perp & A' \ar@<1ex>[l]^{u'}}$ are again adjoints.
\end{prop}
\begin{proof}
Promoting the equivalences to adjoint equivalences we have the composite adjunction
\[ \xymatrix{ B' \ar@<1ex>[r]^\sim \ar@{}[r]|\perp & B \ar@<1ex>[r]^f \ar@<1ex>[l]^\sim \ar@{}[r]|\perp & A \ar@<1ex>[r]^\sim \ar@{}[r]|\perp \ar@<1ex>[l]^u & A' \ar@<1ex>[l]^\sim,}\] defining $f' \dashv u'$.
\end{proof}

An analogous result to Proposition \ref{prop:equiv-invar-adjunction} holds for other notions of equivalent functors:

\begin{exs}\label{exs:iso-invar-adjunction} In any 2-category show that:
\begin{enumerate}[label=(\roman*)]
\item If $f' \cong f$ and $f \dashv u$ then $f' \dashv u$.
\item If $f \dashv u$ and $f' \dashv u$ then $f \cong f'$.
\item If $f' \cong f$ and $f$ is an equivalence then so is $f'$.
\end{enumerate}
\end{exs}

\subsection{\texorpdfstring{$\infty$}{infinity}-cosmoi}

The results in the previous section illustrate some of the basic formal category theory that can be developed internally to any 2-category. As our terminology suggests, Definitions \ref{defn:adjunction} and \ref{defn:equivalence} indeed define appropriate notions of adjunction and equivalence for many varieties of $(\infty,1)$-categories in such a way that \ref{prop:adjunctions-compose}, \ref{exs:adjoint-equivalence}, \ref{cor:equivs-are-adjoints}, \ref{prop:equiv-invar-adjunction}, and \ref{exs:iso-invar-adjunction} all hold. To understand this, we will now explain how to define homotopy 2-categories whose objects are, respectively, quasi-categories, complete Segal spaces, Segal categories, or naturally marked simplicial sets, among other examples. Each homotopy 2-category arises as a quotient of an $\infty$-\emph{cosmos}, which records the properties of the model category ---  typically used to encode the ``homotopy theory'' of each type of $(\infty,1)$-category --- that will be required to develop their basic category theory.

It is not necessary to know anything about model categories in order to understand the definition of an $\infty$-cosmos, but for those who are familiar with them, we briefly say a few words to motivate this axiomatization. If our aim was to establish the basic homotopy theory of some type of objects, experience suggests that a good setting in which to work is a  simplicial model category, that is in a simplicially enriched category equipped with a model structure that is enriched over the Kan/Quillen model structure on simplicial sets. In a simplicial model category, there is a formula for the homotopy limit or colimit of a diagram of any shape.  In the subcategory of fibrant-cofibrant objects, the model-theoretic notion of ``weak equivalence'' coincides with a more symmetrically defined notion of ``homotopy equivalence.'' Also, mapping-spaces between objects that are fibrant and cofibrant  have the ``correct homotopy type'' and are invariant under these notions of equivalence.

By analogy, we posit that a good setting in which to establish the basic category theory of  some type of objects (the ``$\infty$-categories'') is in a simplicially enriched category equipped with a model structure that is enriched over the Quasi/Joyal model structure on simplicial sets. Each of the models of $(\infty,1)$-categories listed above arise as precisely the fibrant objects (which are also cofibrant) in a model structure of this kind. In fact, to do our work, we need only a portion of the axioms defining a Quasi/Joyal-enriched model structure, which define what we call an $\infty$-\emph{cosmos}.

\begin{defn}[$\infty$-cosmos]\label{qcat.ctxt.cof.def}
An $\infty$-\emph{cosmos} is a simplicially enriched category $\lcat{K}$ whose 
\begin{itemize}
\item objects we refer to as the \emph{$\infty$-categories} in the $\infty$-cosmos, whose
\item hom simplicial sets $\Fun(A,B)$ are all  quasi-categories\footnote{A quasi-category is a particular type of simplicial set. See \ref{rec:equiv-qcats}.}, 
\end{itemize} and that is equipped with a specified subcategory of \emph{isofibrations}, denoted by ``$\tfib$'',
satisfying the following axioms:
 \begin{enumerate}[label=(\alph*)]
    \item\label{qcat.ctxt.cof:a} (completeness) As a simplicially enriched category,  $\lcat{K}$ possesses a terminal object $1$, cotensors $A^U$ of  objects $A$ by all\footnote{It suffices to require only cotensors with finitely presented simplicial sets (those with only finitely many non-degenerate simplices).} simplicial sets $U$, and pullbacks of isofibrations along any functor.\footnote{For the theory of homotopy coherent adjunctions and monads developed in \cite{RiehlVerity:2012hc}, retracts and limits of towers of isofibrations are also required, with the accompanying stability properties of \ref{qcat.ctxt.cof:b}. These limits are present in all of the $\infty$-cosmoi we are aware of, but will not be required for any results discussed here.}
    \item\label{qcat.ctxt.cof:b} (isofibrations) The class of isofibrations contains the isomorphisms and all of the functors $!\colon A \tfib 1$ with codomain $1$; is stable under pullback along all functors; and if $p\colon E\tfib B$ is an isofibration in $\lcat{K}$ and $i\colon U\inc V$ is an inclusion of  simplicial sets then the Leibniz cotensor $i\leib\pwr p\colon E^V\tfib E^U\times_{B^U} B^V$ is an isofibration. Moreover, for any object $X$ and isofibration $p \colon E \tfib B$, $\Fun(X,p) \colon \Fun(X,E) \tfib \Fun(X,B)$ is an isofibration of quasi-categories.
\end{enumerate}
The underlying category of an $\infty$-cosmos $\lcat{K}$ has a canonical subcategory of (representably-defined) \emph{equivalences}, denoted by ``$\we$'', satisfying the 2-of-6 property. A functor $f \colon A \to B$ is an \emph{equivalence} just when the induced functor $\Fun(X,f) \colon \Fun(X,A) \to \Fun(X,B)$ is an equivalence of quasi-categories for all objects $X \in \lcat{K}$.  The  \emph{trivial fibrations}, denoted by ``$\trvfib$'', are those functors that are both equivalences and isofibrations. 

It follows from \ref{qcat.ctxt.cof.def}\ref{qcat.ctxt.cof:a}-\ref{qcat.ctxt.cof:b} that: 
 \begin{enumerate}[label=(\alph*), resume]
    \item\label{qcat.ctxt.cof:c} (cofibrancy) All objects are \emph{cofibrant}, in the sense that they enjoy the left lifting property with respect to all trivial fibrations in $\lcat{K}$. 
\[ \xymatrix{ & E \ar@{->>}[d]^{\rotatebox{90}{$\displaystyle\sim$}} \\ A \ar[r] \ar@{-->}[ur]^{\exists} & B}\] 
    \item\label{qcat.ctxt.cof:d} (trivial fibrations) The trivial fibrations define a subcategory containing  the isomorphisms; are stable under pullback along all functors; and the Leibniz cotensor $i\leib\pwr p\colon E^V\trvfib E^U\times_{B^U}  B^V$ of an isofibration $p\colon E\tfib B$ in $\lcat{K}$ and a monomorphism $i\colon U\inc V$ between simplicial sets   is a trivial fibration when $p$ is a trivial fibration in $\lcat{K}$ or $i$ is trivial cofibration in the Joyal model structure on $\sSet$ (see \refV{lem:triv.fib.stab}). Moreover, for any object $X$ and trivial fibration $p \colon E \trvfib B$, $\Fun(X,p) \colon \Fun(X,E) \trvfib \Fun(X,B)$ is a trivial fibration of quasi-categories.\footnote{The surjectivity on vertices of such maps proves \ref{qcat.ctxt.cof:c}.}
\item\label{qcat.ctxt.cof:e} (factorization) Any functor $f \colon A \to B$ may be factored as $f = p j$ 
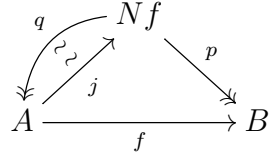
\[ \xymatrix{ & N_f \ar@{->>}[dr]^p \ar@{->>}@/_3ex/[dl]_{q}^*-{\rotatebox{45}{$\labelstyle\sim$}} \\ A \ar[rr]_f \ar[ur]^*-{\rotatebox{45}{$\labelstyle\sim$}}_j & & B}\] where $p \colon N_f \tfib B$ is an isofibration and $j \colon A \we N_f$ is right inverse to a trivial fibration $q \colon N_f \trvfib A$ (see \refIV{lem:Brown.fact}).
\end{enumerate}

Every $\infty$-cosmos has products by \ref{qcat.ctxt.cof:a}.  An $\infty$-cosmos is \emph{cartesian closed} if it satisfies the extra axiom:
 \begin{enumerate}[label=(\alph*), resume]
    \item\label{qcat.ctxt.cof:g} (cartesian closure) The product bifunctor $-\times - \colon \lcat{K} \times \lcat{K} \to \lcat{K}$ extends to a simplicially enriched two-variable adjunction
\[ \Fun(A \times B,C) \cong \Fun(A, C^B) \cong \Fun(B,C^A).\]
\end{enumerate}
\end{defn}

\begin{rec}[quasi-categories]\label{rec:equiv-qcats}
A \emph{quasi-category} is a simplicial set satisfying the weak Kan condition, i.e., in which every inner horn $\Horn^{n,k} \inc \Del^n$, $0 < k< n$,  has a filler. The quasi-categories define the fibrant objects in a model structure on simplicial sets due to Joyal in which the cofibrations are the monomorphisms and the fibrations between fibrant objects are maps we call \emph{isofibrations}. An \emph{isofibration} is a map that has the right lifting property against the inner horn inclusions and also against the inclusion $\Del^0 \to \iso$ of either endpoint into the ``interval'' $\iso$, defined to be the nerve of the free category $\bullet \cong \bullet$ containing an isomorphism.\footnote{Joyal refers to these maps as \emph{quasi-fibrations} \cite{Joyal:2002:QuasiCategories} and Lurie calls them \emph{categorical fibrations} \cite{Lurie:2009ht}. We prefer isofibrations because this term indicates the corresponding 2-categorical property; see \refIV{lem:isofib.are.representably.so}.}

The weak equivalences between quasi-categories are precisely the \emph{equivalences between quasi-categories}, which  can be understood as a type of ``simplicial homotopy equivalence'' with respect to the interval $\iso$. That is, a map $f \colon A \to B$ of quasi-categories is an equivalence just when there exists a map $g \colon B \to A$ together with maps $A \to A^\iso$ and $B \to B^\iso$ that restrict along the vertices of $\iso$ to the maps $\id_A$, $gf$, $fg$, and $\id_B$ respectively: 
\[ f \colon A \we B\quad \mathrm{iff}\quad \exists g \colon B \we A\quad \mathrm{and} \quad \vcenter{\xymatrix{ & A \\ A \ar[r] \ar[ur]^{\id_A} \ar[dr]_{gf} & A^\iso \ar@{->>}[u]_{p_0} \ar@{->>}[d]^{p_1} \\ & A}} \quad\mathrm{and}\quad \vcenter{\xymatrix{ & B \\ B \ar[r]\ar[ur]^{fg} \ar[dr]_{\id_B} & B^{\iso} \ar@{->>}[u]_{p_0} \ar@{->>}[d]^{p_1} \\ & B}}\]
\end{rec}

The subcategory of fibrant objects in a model category that is enriched over the Joyal model structure on simplicial sets defines an $\infty$-cosmos --- assuming all fibrant objects are cofibrant.\footnote{This hypothesis is not essential; see \refIV{qcat.ctxt.def}.} In such examples, the convention will be to define the isofibrations to be the fibrations between fibrant objects. It follows that the equivalences are precisely the weak equivalences between fibrant objects in the model category. This is the source of each of the following examples of $\infty$-cosmoi. 

\begin{ex} There exist $\infty$-cosmoi:
\begin{itemize}
\item $\Cat$, whose objects are ordinary categories, with isofibrations and equivalences the usual categorical isofibrations and equivalences (\refIV{ex:cat-cosmos});
\item $\qCat$, whose objects are quasi-categories (\refIV{ex:qcat-qcat-ctxt}), with isofibration and equivalences as in \ref{rec:equiv-qcats};
\item $\CSS$, whose objects are complete Segal spaces (\refIV{ex:CSS-cosmos});
\item $\Segal$, whose objects are Segal categories (\refIV{ex:segal-cosmos}); 
\item $\sSet_+$, whose objects are naturally marked simplicial sets (\refIV{ex:marked-cosmos});
\end{itemize}
all of which are cartesian closed. Thus each of these varieties of $(\infty,1)$-categories are examples of $\infty$-\emph{categories}, in our sense; the associated $\infty$-\emph{functors} are just the usual functors, maps of simplicial sets, maps of bisimplicial sets, and maps of marked simplicial sets, respectively.  For ordinary categories, the isofibrations and equivalences coincide with the usual categorical notions bearing these names. For the quasi-categories, complete Segal spaces, Segal categories, and marked simplicial sets, the equivalences of $\infty$-categories are exactly the weak equivalences between fibrant-cofibrant objects in the model structure that is used to present the basic homotopy theory of each variety of $(\infty,1)$-category.
\end{ex}

So, each model of $(\infty,1)$-categories mentioned in the introduction has an $\infty$-cosmos. These are not the only examples, however:

\begin{ex} There exist $\infty$-cosmoi:
\begin{itemize}
\item $\theta_n$-$\Sp$, whose objects are  \emph{$\theta_n$-spaces}, a simplicial presheaf model of $(\infty,n)$-categories (\refIV{ex:theta-n-cosmos})
\item $\Rezk_\lcat{M}$, whose objects are \emph{Rezk objects} in a sufficiently nice model category $\lcat{M}$.\footnote{Here, ``sufficiently nice'' means permitting left Bousfield localization. With the definition of $\infty$-cosmos presented in \ref{qcat.ctxt.cof.def} we also need to require that the resulting fibrant objects are all cofibrant.} Rezk objects are used to define iterated complete Segal spaces, another simplicial presheaf model of $(\infty,n)$-categories (\refIV{prop:rezk-cosmos}).
\end{itemize}
Moreover, if $\lcat{K}$ is any $\infty$-cosmos and $B \in \lcat{K}$, then there is a \emph{sliced $\infty$-cosmos} $\lcat{K}_{/B}$, whose objects are isofibrations with codomain $B$. Sliced $\infty$-cosmoi will play a big role in  \ref{sec:modules}.
\end{ex}

In summary, $\infty$-categories, for us, are the objects in a universe called an $\infty$-cosmos that is suitable for the development of their basic category theory --- much like a simplicial model category is a suitable environment in which to develop the basic homotopy theory of its objects. In our definition of an $\infty$-cosmos, we are not seeking to axiomatize the universe surrounding any particular variety of infinite-dimensional category, in contrast to \cite{Toen:2005vu}  or   \cite{BSP:2011ot}. Rather, the axioms outline what is needed to prove our theorems. The axiomatization presented here could also be made more general --- indeed, \cite{RiehlVerity:2015fy} uses a weaker definition of $\infty$-cosmos than will be considered here and further weakenings are also possible. Our aim is to optimize for simplicity of presentation, while applying sufficiently broadly. A perpetual challenge in category theory, or in many areas of abstract mathematics, is  to find the right level of generality, which is often not the maximal level of generality.

\subsection{The homotopy 2-category of an \texorpdfstring{$\infty$}{infinity}-cosmos}

In fact most of our work to develop the basic theory of $\infty$-categories takes place not in their ambient $\infty$-cosmos, but in a quotient of the $\infty$-cosmos that we call the \emph{homotopy 2-category}. Each $\infty$-cosmos has an underlying 1-category whose objects are the $\infty$-categories of that $\infty$-cosmos and whose morphisms, which we call $\infty$-\emph{functors} or more often simply \emph{functors}, are the vertices of the mapping quasi-categories. 

\begin{defn}[the homotopy 2-category of $\infty$-cosmos] The \emph{homotopy 2-category} of an $\infty$-cosmos $\lcat{K}$ is a strict 2-category $\h\lcat{K}$ so that 
\begin{itemize}
\item the objects of $\h\lcat{K}$ are the objects of $\lcat{K}$, i.e., the $\infty$-categories;
\item the 1-cells $f \colon A \to B$ of $\h\lcat{K}$ are the vertices $f \in \Fun(A,B)$ in the mapping quasi-categories of $\lcat{K}$, i.e., the $\infty$-functors;
\item a 2-cell  $\xymatrix{ A \ar@/^2ex/[r]^f \ar@/_2ex/[r]_g \ar@{}[r]|{\Downarrow\alpha}& B}$ in $\h\lcat{K}$ is represented by a 1-simplex $\alpha \colon f \to g \in \Fun(A,B)$, where a parallel pair of 1-simplices in $\Fun(A,B)$ represent the same 2-cell if and only if they bound a 2-simplex whose remaining outer face is degenerate.
\end{itemize}
Put concisely, the homotopy 2-category is the 2-category $\h\lcat{K} \defeq \ho_*\lcat{K}$ defined by applying the homotopy category functor $\ho \colon \qCat \to \Cat$ to the mapping quasi-categories of the $\infty$-cosmos; the hom-categories in $\h\lcat{K}$  are defined by the formula  \[\hom(A,B)\defeq \ho(\Fun(A,B))\] to be the homotopy categories of the mapping quasi-categories in $\lcat{K}$. 
\end{defn}

The homotopy 2-category $\h\qCat$ of the $\infty$-cosmos of quasi-categories was first introduced by Joyal in his work on the foundations of quasi-category theory.

\begin{obs}[functors representing (isomorphic) 2-cells]
We write $\cattwo$ for the simplicial set $\Del^1$, which is the nerve of the walking arrow $\bullet \to \bullet$. A natural transformation $\xymatrix{ A \ar@/^2ex/[r]^f \ar@/_2ex/[r]_g \ar@{}[r]|{\Downarrow\alpha}& B}$ in the homotopy 2-category of an $\infty$-cosmos $\lcat{K}$ is represented by a map of simplicial sets $\alpha \colon \cattwo \to \Fun(A,B)$, which transposes to define a functor in $\lcat{K}$  that composes with the two projections to the maps $g$ and $f$ respectively. 
\[ \xymatrix{ & B \\ A \ar[r]^-{\alpha} \ar[ur]^f \ar[dr]_g & B^\cattwo \ar@{->>}[u]_{p_0} \ar@{->>}[d]^{p_1} \\ & B}\]

The 2-cell $\xymatrix{ A \ar@/^2ex/[r]^f \ar@/_2ex/[r]_g \ar@{}[r]|{\Downarrow\alpha}& B}$ is an isomorphism in the homotopy 2-category $\h\lcat{K}$ if and only if the arrow in $\hom(A,B)$ represented by the map $\ho(\alpha) \colon  \ho\cattwo \to \ho\Fun(A,B)$  is an isomorphism. This is the case if an only if the representing simplicial map $\alpha \colon \cattwo \to \Fun(A,B)$ extends to a simplicial map $\alpha' \colon \iso \to \Fun(A,B)$, which transposes to define a functor in $\lcat{K}$
\[ \xymatrix{ & B \\ A \ar[r]^-{{\alpha}'} \ar[ur]^f \ar[dr]_g &  B^\iso \ar@{->>}[u]_{p_0} \ar@{->>}[d]^{p_1} \\ & B}\]
\end{obs}

A priori, it is a bit of a surprise that the homotopy 2-category remembers enough information from the $\infty$-cosmos to develop the basic category theory of its objects. The first result that shows why this might be the case is the following.

\begin{prop}[{\refIV{prop:equiv.are.weak.equiv}}]\label{prop:equiv.are.weak.equiv} A functor $ f\colon A \to B$ is an equivalence in the $\infty$-cosmos $\lcat{K}$ if and only if it is an equivalence in the homotopy 2-category $\h\lcat{K}$.
\end{prop}
\begin{proof}
By definition, any equivalence $f \colon A \we B$ in the $\infty$-cosmos induces an equivalence $\Fun(X,A) \we \Fun(X,B)$ of quasi-categories for any $X$, which becomes an equivalence of categories $\hom(X,A) \we \hom(X,B)$ upon applying the homotopy category functor $\ho \colon \qCat \to \Cat$. Applying the Yoneda lemma in the homotopy 2-category $\h\lcat{K}$, it follows easily that $f$ is an equivalence in the sense of Definition \ref{defn:equivalence}.

Conversely, as the map $\iso \to \Del^0$ of simplicial sets is a weak equivalence in the Joyal model structure, the cotensor $ B^\iso$ defines a path object for the $\infty$-category $B$. 
\[ \xymatrix{ & B^\iso \ar@{->>}[dr]^-{(p_1,p_0)}  \\ B \ar[rr]_\Delta \ar[ur]_*-{\rotatebox{45}{$\labelstyle\sim$}}^-{\Delta} & & B \times B}\] 
It follows from the 2-of-3 property that any functor that is isomorphic in the homotopy 2-category to an equivalence in the $\infty$-cosmos is again an equivalence in the $\infty$-cosmos. Now it follows immediately from the 2-of-6 property for equivalences in the $\infty$-cosmos and the fact that the class of equivalences includes the identities, that any 2-categorical equivalence in the sense of Definition \ref{defn:equivalence} is an equivalence in the $\infty$-cosmos.
\end{proof}

The upshot is that any categorical notion defined up to equivalence in the homotopy 2-category is also characterized up to equivalence in the $\infty$-cosmos.

Axioms \ref{qcat.ctxt.cof.def}\ref{qcat.ctxt.cof:a} and \ref{qcat.ctxt.cof:b} imply that an $\infty$-cosmos has finite products satisfying a simplicially enriched universal property. Consequently:

\begin{prop}\label{prop:cartesian-closure}
The homotopy 2-category of an $\infty$-cosmos has finite products, and if the $\infty$-cosmos is cartesian closed, then so is its homotopy 2-category.
\end{prop}
\begin{proof}
The homotopy category functor $\ho \colon \qCat \to \Cat$ preserves finite products. Applying this to the  defining isomorphisms $\Fun(X,1)\cong \Del^0$ and $\Fun(X, A\times B)\cong\Fun(X,A)\times\Fun(X,B)$ for the simplicially enriched terminal object and binary products of $\lcat{K}$ yields isomorphisms $\hom(X,1)\cong \catone$ and $\hom(X, A\times B)\cong\hom(X,A)\times\hom(X,B)$. These demonstrate that $1$ and $A\times B$ are also the 2-categorical terminal object and binary products in $\h\lcat{K}$.

In this case where $\lcat{K}$ is cartesian closed, applying the homotopy category functor to the defining isomorphisms on mapping quasi-categories yields the required natural isomorphisms
\[ \hom(A \times B, C) \cong \hom(A,C^B) \cong \hom(B,C^A)\] of hom-categories.
\end{proof}

\begin{defn}\label{defn:qcat-ctxt-functor} A \emph{functor of $\infty$-cosmoi} $F \colon \lcat{K} \to \lcat{L}$ is a simplicial functor that preserves isofibrations and the limits listed  in~\ref{qcat.ctxt.cof.def}\ref{qcat.ctxt.cof:a}. Simplicial functoriality implies that a functor of $\infty$-cosmoi  preserves equivalences and hence also trivial fibrations. 
\end{defn}

For any $\infty$-cosmoi that arise as the fibrant objects in a Joyal-enriched model category, a simplicially enriched right Quillen adjoint will define a functor of $\infty$-cosmoi. This is the source of many of the following examples.

\begin{ex}
The following define functors of $\infty$-cosmoi:
\begin{itemize}
\item $\Fun(X,-) \colon\lcat{K} \to \qCat$ for any object $X \in \lcat{K}$ (see Proposition \refIV{prop:representable-functors}). 
\item As a special case, the \emph{underlying quasi-category functor} $\Fun(1,-) \colon \lcat{K} \to \qCat$. Examples include the functors $\CSS \to \qCat$ and $\Segal \to \qCat$ that take a complete Segal space or Segal category to its 0th row (see \refIV{ex:CSS-cosmos} and  \refIV{ex:segal-cosmos}) and the functor $\sSet_+\to \qCat$ that carries a naturally marked simplicial set to its underlying quasi-category (see \refIV{ex:marked-cosmos}).
\item $(-)^U \colon \lcat{K} \to \lcat{K}$ for any simplicial set $U$, by \ref{qcat.ctxt.cof.def}\ref{qcat.ctxt.cof:b} and the fact that simplicially enriched limits commute with each other.
\item The inclusion $\Cat\to \qCat$ of categories into quasi-categories that identifies a category with its nerve (see  \refIV{ex:cat-cosmos}).
\item The functor $t^! \colon \qCat \to \CSS$ defined in example \refIV{ex:other-CSS-functor}.
\item The functor $\CSS \to \Segal$ that ``discretizes'' the 0th space of a complete Segal space. 
\end{itemize}
\end{ex}

The appropriate notion of functor between 2-categories is called a \emph{2-functor}, preserving all of the structure on the nose. A functor $F \colon \lcat{K} \to \lcat{L}$ of $\infty$-cosmoi induces a 2-functor $\h{F} \defeq \ho_*F \colon \h\lcat{K} \to \h\lcat{L}$ between their homotopy 2-categories. Because adjunctions and equivalences in a 2-category are defined equationally, they are preserved by any 2-functor; in particular, the 2-functor between homotopy 2-categories induced by a functor of $\infty$-cosmoi preserves adjunctions and equivalences. Hence:

\begin{prop}\label{prop:induced-adjunctions}
If $\xymatrix{ B \ar@<1ex>[r]^f \ar@{}[r]|\perp & A \ar@<1ex>[l]^u}$ is an adjunction between $\infty$-categories then \begin{enumerate}[label=(\roman*)]
\item For any $\infty$-category $X$,  $\xymatrix{ \Fun(X,B) \ar@<1ex>[r]^{\Fun(X,f)} \ar@{}[r]|\perp & \Fun(X,A) \ar@<1ex>[l]^{\Fun(X,u)}}$ defines an adjunction of quasi-categories.
\item\label{itm:induced-adjunctions:ii} For any $\infty$-category $X$,  $\xymatrix{ \hom(X,B) \ar@<1ex>[r]^{\hom(X,f)} \ar@{}[r]|\perp & \hom(X,A) \ar@<1ex>[l]^{\hom(X,u)}}$ defines an adjunction of categories.
\item For any  simplicial set $U$,  $\xymatrix{ B^U \ar@<1ex>[r]^{ f^U} \ar@{}[r]|\perp &  A^U \ar@<1ex>[l]^{u^U}}$ defines an adjunction of $\infty$-categories.
\item If the $\infty$-cosmos is cartesian closed, then  for any $\infty$-category $C$ the pre- and post-composition functors define adjunctions of $\infty$-categories:
\[ \xymatrix{ B^C \ar@<1ex>[r]^{f_*} \ar@{}[r]|\perp & A^C \ar@<1ex>[l]^{u_*} & C^B \ar@<1ex>[r]^{u^*} \ar@{}[r]|\perp & C^A \ar@<1ex>[l]^{f^*}}\]
\end{enumerate}
\end{prop}

Taking $X = 1$ in \ref{itm:induced-adjunctions:ii} yields an adjunction between the \emph{homotopy categories} associated to the $\infty$-categories $A$ and $B$.

\begin{proof}
The adjunction $f \dashv u$ in $\h\lcat{K}$ is preserved by the 2-functors $\Fun(X,-)\colon \h\lcat{K} \to \h\qCat$, $\hom(X,-) \colon \h\lcat{K} \to \h\Cat$, $(-)^U\colon \h\lcat{K} \to \h\lcat{K}$, $(-)^C \colon \h\lcat{K} \to \h\lcat{K}$, and $C^{(-)} \colon \h\lcat{K}\op \to \h\lcat{K}$.
\end{proof}

Via the simplicial cotensor and the embedding $\Cat\inc\sSet$, for any $\infty$-category $A$ in an $\infty$-cosmos $\lcat{K}$, there is also a 2-functor $A^{(-)} \colon \h\Cat\op\to \h\lcat{K}$, which is another source of adjunctions between $\infty$-categories:

\begin{ex}[{\refI{ex:comp.ident.adj}}]
For any $\infty$-category $A$, there is an adjunction     
\[\xymatrix@C=10em{
      *+[l]{ A^\cattwo\times_AA^\cattwo}
      \ar[r]|*+{\scriptstyle m} &
      {A^\cattwo}
      \ar@/^2.5ex/[l]^{i_1}_{}="l" \ar@/_2.5ex/[l]_{i_0}^{}="u"
      \ar@{} "u";"l" |(0.2){\bot} |(0.8){\bot} 
    }\] between the ``composition'' functor $m$ and the pair of functors that ``extend an arrow into a composable pair'' by using the identities at its domain and codomain.

To prove this, first note that there exists a pair of adjunctions
  \begin{equation*}
    \xymatrix@C=10em@R=1ex{
      {{\cattwo}}\ar[r]|*+{\scriptstyle \face^1} & {{\catthree}}
      \ar@/^2.5ex/[l]^{\degen^0}_{}="l" \ar@/_2.5ex/[l]_{\degen^1}^{}="u"
      \ar@{} "u";"l" |(0.2){\bot} |(0.8){\bot}
    }
  \end{equation*}
between ordinal categories so that the counit of the top adjunction and unit of the bottom adjunction are identities. Applying $A^{(-)} \colon \h\Cat\op\to \h\lcat{K}$ converts these into  adjunctions
  \begin{equation*}
    \xymatrix@C=10em{
      {A^\catthree}
      \ar[r]|*+{\scriptstyle A^{\face^1}} &
      {A^\cattwo}
      \ar@/^2.5ex/[l]^{A^{\degen^1}}_{}="l" \ar@/_2.5ex/[l]_{A^{\degen^0}}^{}="u"
      \ar@{} "u";"l" |(0.2){\bot} |(0.8){\bot} 
    } 
  \end{equation*}
in which the upper adjunction has identity unit and the lower adjunction has identity counit. We write $A^\cattwo \defeq A^{\Del^1}$ and refer to this as the \emph{arrow $\infty$-category} associated to $A$ on account of a weak 2-categorical universal property that we will describe in  \ref{sec:comma}.

The horn inclusion $\Horn^{2,1}\inc\Del^2$ is a trivial cofibration in Joyal's model structure, inducing an equivalence $p \colon A^\catthree \trvfib A^\cattwo \times_A A^\cattwo$ of $\infty$-categories, whose codomain we identify from the left-hand pushout in simplicial sets, which induces the right-hand pullback in the $\infty$-cosmos:
\begin{equation*}
  \xymatrix@=2em{ \Horn^{2,1} \pbexcursion & \Del^1 \ar[l]_-{\face^2} & & 
    {A^{\Horn^{2,1}}}\pbexcursion
    \ar[r]^-{\pi_0}\ar[d]_-{\pi_1} & {A^\cattwo}\ar@{->>}[d]^-{p_1} \\ \Del^1 \ar[u]^{\face^0} & \Del^0 \ar[u]_{\face^0}\ar[l]^-{\face^1} & &
    {A^\cattwo}\ar@{->>}[r]_-{p_0} & A
  }
\end{equation*} 
Proposition \ref{prop:equiv.are.weak.equiv} tells us that $p \colon A^\catthree \trvfib A^\cattwo \times_A A^\cattwo$ defines an equivalence in the homotopy 2-category. In particular, by Corollary \ref{cor:equivs-are-adjoints}, $p$ admits an equivalence inverse $p'$ that is simultaneously a left and a right adjoint. 
Composing $p \dashv p' \dashv p$ with the displayed adjunction, we obtain the adjunctions $i_0 \dashv m \dashv i_1$.

In fact, these adjunctions can be defined so that the unit and counit 2-cells, and not just the functors, are fibered over the endpoint evaluation functors $A^\cattwo\times_AA^\cattwo \tfib A \times A$ and $A^\cattwo \tfib A \times A$. The proof makes use of the fact that these maps are isofibrations; see \refI{ex:comp.ident.adj}.
\end{ex}

\renewcommand\thesection{Lecture~\arabic{section}}
\section{Limits and colimits in \texorpdfstring{$\infty$}{infinity}-categories}\label{sec:limits}
\renewcommand\thesection{\arabic{section}}

Recall that we use the term $\infty$-\emph{category} to refer to any variety of infinite-dimensional category that inhabits an $\infty$-cosmos.  An $\infty$-\emph{cosmos} is a simplicially enriched category $\lcat{K}$, whose homs $\Fun(A,B)$ are quasi-categories, that admits certain simplicially-enriched limit constructions and whose specified class of \emph{isofibrations} enjoy certain closure properties. The objects and morphisms of the underlying category of $\lcat{K}$ define the $\infty$-\emph{categories} and $\infty$-\emph{functors} of the $\infty$-cosmos. There are $\infty$-cosmoi for quasi-categories, complete Segal spaces, Segal categories, and naturally marked simplicial sets, each of these being a model of $(\infty,1)$-categories.

Our development of the basic theory of adjunctions and equivalences between $\infty$-cat\-e\-gor\-ies takes place entirely within the \emph{homotopy 2-category} $\h\lcat{K}$ of an $\infty$-cosmos $\lcat{K}$. The objects and morphisms in the homotopy 2-category are again the $\infty$-categories and $\infty$-functors. A 2-cell $\xymatrix{ A \ar@/^2ex/[r]^f \ar@/_2ex/[r]_{g} \ar@{}[r]|{\Downarrow \alpha} &B}$ between a parallel pair of functors between $\infty$-categories is represented by a 1-simplex $\alpha \colon f \to g \in \Fun(A,B)$, where two parallel 1-simplices are equivalent if and only if they bound a 2-simplex whose third edge is degenerate. More concisely, the hom-category $\hom(A,B)$ between two objects in the homotopy 2-category is the homotopy category of the mapping quasi-category $\Fun(A,B)$.

Proposition \ref{prop:cartesian-closure} demonstrates that the homotopy 2-category of any $\infty$-cosmos has finite products, satisfying a 2-dimensional universal property. For the terminal $\infty$-category 1, this says that $\hom(X,1) \cong \catone$, that is there is a unique $\infty$-functor $! \colon X \to 1$ for any $\infty$-category $X$ and this functor admits no non-identity endomorphisms.  Proposition \ref{prop:equiv.are.weak.equiv} demonstrates that the notions of equivalence between $\infty$-categories defined at the level of the $\infty$-cosmos coincide precisely with the notions of equivalence in the homotopy 2-category, the upshot being that equivalence-invariant 2-categorical constructions are appropriately homotopical.

Fix an ambient $\infty$-cosmos $\lcat{K}$. We will work inside this universe for the remainder of this lecture. Our aim now is to define appropriate notions of limit and colimit of diagrams taking values in an $\infty$-category $A$ inside the $\infty$-cosmos. To define limits and colimits, we need a way to look inside the $\infty$-category $A$ without leaving the comfort of the $\infty$-cosmos axiomatization. For this, we make use of the terminal $\infty$-category $1$. A functor $a \colon 1\to A$ will be called an \emph{element}\footnote{Synonyms include \emph{point} or \emph{object}. The term ``element'' is perhaps less traditional but also less likely to be confused with other mathematical notions currently under consideration.} of $A$; by analogy, a functor $a \colon X \to A$ is a \emph{generalized element} of $A$. Elements of $A$ and the 2-cells between them define the \emph{homotopy category} of the $\infty$-category $A$, as we record in passing.

\begin{defn}\label{defn:homotopy-category}
The \emph{homotopy category}  of an $\infty$-category $A$ is the category $\hom(1,A)$ whose objects are the elements of $A$ and whose morphisms $\alpha \colon a \to a'$ are 2-cells $\xymatrix{ 1 \ar@/^2ex/[r]^a \ar@/_2ex/[r]_{a'} \ar@{}[r]|{\Downarrow \alpha} &A}$.
\end{defn}

\subsection{Terminal elements}

Before introducing the general notion of limits, we will warm up with a special case of terminal objects, which we will call \emph{terminal elements}.

\begin{defn}\label{defn:terminal-element} A \emph{terminal element} in an $\infty$-category $A$ is a right adjoint $t \colon 1 \to A$ to the unique functor $! \colon A \to 1$. Explicitly, the data consists of:
\begin{itemize}
\item an element $t \colon 1 \to A$ and
\item a natural transformation $\eta \colon \id_A \To t!$ whose component $\eta t$ at the element $t$ is an isomorphism.\footnote{If $\eta$ is the unit of the adjunction $! \dashv t$, then the triangle identities demand that $\eta t =\id_t$.  However, by a 2-categorical trick, to show that such an adjunction exists, it suffices to find a 2-cell $\eta$ so that $\eta t$ is an isomorphism (see \refI{lem:min-term-pres}).}
\end{itemize}
\end{defn}

Several basic facts about terminal elements can be deduced immediately from the general theory of adjunctions.

\begin{prop}$\quad$
\begin{enumerate}[label=(\roman*)]
\item\label{itm:term-i} An element $t \colon 1 \to A$ is terminal if and only if it is representably terminal, i.e., if for all $f \colon X \to A$ there exists a unique 2-cell $\vcenter{\xymatrix@=1em{ X \ar[rr]^f \ar[dr]_{!} & \ar@{}[d]|(.4){\Downarrow\exists!} & A \\ & 1 \ar[ur]_t}}$
\item\label{itm:term-ii} Terminal elements are preserved by right adjoints and by equivalences.
\item\label{itm:term-iii} If  $A' \simeq A$ then $A$ has a terminal element if and only if $A'$ does.
\end{enumerate}
\end{prop}
\begin{proof}
For \ref{itm:term-i}, Proposition \ref{prop:induced-adjunctions}\ref{itm:induced-adjunctions:ii} proves that terminal elements are representably terminal; the converse follows from the Yoneda lemma. \ref{itm:term-ii} is a special case of Proposition \ref{prop:adjunctions-compose}, via Corollary \ref{cor:equivs-are-adjoints}; \ref{itm:term-iii} follows.
\end{proof}

\subsection{Limits}

Terminal elements are limits indexed by the empty set. We now turn to limits of generic diagrams whose indexing shapes are given by simplicial sets. We have a 2-category $\h\sSet$ of simplicial sets, extending in the evident way the definition of the homotopy 2-category $\h\qCat \subset \h\sSet$ of quasi-categories. The 2-category of categories sits as a full subcategory $\h\Cat\subset\h\qCat \subset \h\sSet$, with categories identified with the simplicial sets defining their nerves. In this way, diagrams indexed by categories are among the diagrams indexed by simplicial sets. 

For any $\infty$-category $A$ in an $\infty$-cosmos $\lcat{K}$, there is a simplicial functor $A^{(-)} \colon \sSet\op \to \lcat{K}$, which descends to a 2-functor $A^{(-)} \colon \h\sSet\op \to \h\lcat{K}$. These simplicial cotensors are used to define $\infty$-categories of diagrams. 

\begin{defn}[diagram $\infty$-categories]\label{defn:diagram-cats} If $J$ is a simplicial set and $A$ is an $\infty$-category, then the $\infty$-category $A^J$ is the \emph{$\infty$-category of $J$-indexed diagrams in $A$}.
\end{defn}

\begin{rmk}
In the case where the $\infty$-cosmos is cartesian closed, in the sense of Definition \ref{qcat.ctxt.cof.def}\ref{qcat.ctxt.cof:g}, we could instead take the indexing shape $J$ to be an $\infty$-category, in which case the internal hom $A^J$ is the \emph{$\infty$-category of $J$-indexed diagrams in $A$}. The development of the theory of limits indexed by an $\infty$-category in a cartesian closed $\infty$-cosmos entirely parallels the development for limits indexed by a simplicial set. The conflated notation of \ref{qcat.ctxt.cof.def}\ref{qcat.ctxt.cof:a} and \ref{qcat.ctxt.cof.def}\ref{qcat.ctxt.cof:g} is intended to further highlight this parallelism.
\end{rmk}

In analogy with Definition \ref{defn:terminal-element}, we have:

\begin{defn}\label{defn:all-limits} An $\infty$-category $A$ \emph{admits all limits of shape $J$} if the constant diagram functor $\Delta \colon A \to A^J$, induced by the unique functor $!\colon  J \to 1$, has a right adjoint:
\[ \xymatrix{ A \ar@<1ex>[r]^-\Delta \ar@{}[r]|-\perp & A^J \ar@<1ex>[l]^-{\lim}}\] 
\end{defn}

From the vantage point of Definition \ref{defn:all-limits}, the following result is easy:

\begin{exs} Show, using \ref{prop:adjunctions-compose} and \ref{exs:iso-invar-adjunction},  that a right adjoint functor $u \colon A \to B$ between $\infty$-categories that admit all limits of shape $J$  necessarily preserves them, in the sense that the functors
\[ \xymatrix{ A^J \ar[d]_{\lim} \ar[r]^{u^J} & B^J \ar[d]^{\lim} \ar@{}[dl]|\cong \\ A \ar[r]_u & B}\] commute up to isomorphism.
\end{exs}

The problem with Definition \ref{defn:all-limits} is that it is insufficiently general: many $\infty$-categories will have certain, but not all, limits of diagrams of a particular indexing shape. With this aim in mind, we will now re-express Definition \ref{defn:all-limits} in a form that permits its extension to cover this sort of situation. For this, we make use of the following 2-categorical notion.

\begin{defn}[absolute right lifting] Given a cospan $C \xrightarrow{g} A \xleftarrow{f} B$, a functor $\ell \colon C \to B$ and a 2-cell \begin{equation}\label{eq:abs-right-lifting} \xymatrix{ \ar@{}[dr]|(.7){\Downarrow\lambda} & B \ar[d]^f \\ C \ar[ur]^\ell \ar[r]_g & A}\end{equation} define an \emph{absolute right lifting of $g$ through $f$} if any 2-cell as displayed below-left factors uniquely through $\lambda$ as displayed below-right
\[    \vcenter{\xymatrix{ X \ar[d]_c \ar[r]^b \ar@{}[dr]|{\Downarrow\chi} & B \ar[d]^f \\ C \ar[r]_g & A}} \mkern20mu = \mkern20mu \vcenter{\xymatrix{ X \ar[d]_c \ar[r]^b \ar@{}[dr]|(.3){\exists !\Downarrow}|(.7){\Downarrow\lambda} & B \ar[d]^f \\ C \ar[ur]|(.4)*+<2pt>{\scriptstyle \ell} \ar[r]_g & A}}
\]
\end{defn}

We refer to the 2-cell \eqref{eq:abs-right-lifting} as an \emph{absolute right lifting diagram}. In category theory, the term ``absolute'' typically means ``preserved by all functors.'' An absolute right lifting diagram is a right lifting diagram $\lambda \colon f\ell \To g$ so that the restriction of $\lambda$ along any generalized element $c \colon X \to C$ again defines a right lifting diagram.

\begin{exs}\label{exs:adj-as-abs-lifting} Show that in any 2-category, a 2-cell $\epsilon \colon fu \To \id_A$ defines the counit of an adjunction $f \dashv u$ if and only if
\[ \xymatrix{ \ar@{}[dr]|(.7){\Downarrow\epsilon} & B \ar[d]^f \\ A \ar[ur]^u \ar@{=}[r] & A}\] defines an absolute right lifting diagram.
\end{exs}

Applying Exercise \ref{exs:adj-as-abs-lifting}, Definition \ref{defn:all-limits} is equivalent to the assertion that the \emph{limit cone}, our term for the counit of $\Delta \dashv \lim$  defines an absolute right lifting diagram:
\begin{equation}\label{eq:all-limits-abs-lifting} \xymatrix{ \ar@{}[dr]|(.7){\Downarrow\epsilon} & A\ar[d]^\Delta \\ A^J \ar@{=}[r] \ar[ur]^{\lim} & A^J}\end{equation}  This motivates the following definition.

\begin{defn}[limit]\label{defn:limit} A \emph{limit} of a $J$-indexed diagram in $A$ is an absolute right lifting of the diagram $d$ through the constant diagram functor $\Delta \colon A \to A^J$
\begin{equation}\label{eq:lim-diagram-defn} \xymatrix{ \ar@{}[dr]|(.7){\Downarrow\lambda} & A\ar[d]^\Delta \\ 1 \ar[r]_d \ar[ur]^{\lim d} & A^J}\end{equation} the 2-cell component of which  defines the \emph{limit cone} $\lambda \colon \Delta \lim d \To d$.
\end{defn}

If $A$ has all $J$-indexed limits, then the restriction of the absolute right lifting diagram \eqref{eq:all-limits-abs-lifting} along the element $d \colon 1\to A^J$ defines a limit for $d$. Interpolating between Definitions \ref{defn:limit} and \ref{defn:all-limits}, we can define a \emph{limit of a family of diagrams} to be an absolute right lifting of the family $d \colon K \to A^J$ through $\Delta \colon A \to A^J$. For instance:

\begin{thm}[{\refI{thm:splitgeorealizations}}]\label{thm:totalization} For every cosimplicial object in an $\infty$-category that admits an coaugmentation and a splitting, the coaugmentation defines its limit. That is, for every $\infty$-category $A$, the functors
\[ \xymatrix{ \ar@{}[dr]|(.7){\Downarrow\lambda} & A\ar[d]^\Delta \\ A^{\Del[b]} \ar[r]_{\mathrm{res}} \ar[ur]^{\ev_{[-1]}} & A^{\Del}}\] define an absolute right lifting diagram.
\end{thm}
Here $\Del$ is the usual simplex category of finite non-empty ordinals and order-preserving maps. It defines a full subcategory of $\Del+$, which freely appends an initial object $[-1]$, and this in turn defines a wide subcategory\footnote{A \emph{wide} subcategory is a subcategory containing all of the objects.} of $\Del[b]$, which adds an ``extra degeneracy'' map between each pair of consecutive ordinals. Diagrams indexed by $\Del \subset \Del+\subset \Del[b]$ are, respectively, called \emph{cosimplicial objects}, \emph{coaugmented cosimplicial objects}, and \emph{split cosimplicial objects}. The limit of a cosimplicial object is often called its \emph{totalization}.

\begin{proof}[Proof sketch]
In $\h\Cat$, there is a canonical 2-cell
\[ \xymatrix{ \Del \ar@{^(->}[r] \ar[d]_{!} & \Del[b] \\ \catone \ar[ur]_{[-1]} & \ar@{}[ul]|(.7){\Uparrow\lambda}}\]
because $[-1] \in \Del[b]$ is initial. This data defines an absolute right extension diagram that is moreover preserved by any 2-functor, as the universal property of the functor $[-1] \colon \catone \to \Del[b]$ and the 2-cell $\lambda$ is exhibited by a pair of adjunctions (see \refI{lem:doms2catlemma}). The 2-functor  $A^{(-)} \colon \h\Cat\op \to \h\lcat{K}$ converts this into the absolute right lifting diagram of the statement.
\end{proof} 

The most important result relating adjunctions and limits is of course:

\begin{thm}[{\refI{prop:RAPL}}]\label{thm:RAPL} Right adjoints preserve limits.
\end{thm}
Our proof will closely follow the classical one. Given a diagram $d\colon 1 \to A^J$ and a right adjoint $u \colon A \to B$ to some functor $f$, a cone with summit $b \colon 1 \to B$ over $u^J d$ transposes to define a cone with summit $fb$ over $d$, which factors uniquely through the limit cone. This factorization transposes back across the adjunction to show that $u$ carries the limit cone over $d$ to a limit cone over $u^Jd$.

\begin{proof}
Suppose that $A$ admits limits of a diagram $d\colon 1\to A^J$ as witnessed by an absolute right lifting diagram \eqref{eq:lim-diagram-defn}. By Proposition \ref{prop:induced-adjunctions}, an adjunction $f \dashv u$ induces an adjunction $f^J \dashv u^J$. We must show that \[\xymatrix{ \ar@{}[dr]|(.7){\Downarrow\lambda} & A \ar[d]^-\Delta \ar[r]^u & B \ar[d]^-\Delta \\ 1 \ar[ur]^{\lim d} \ar[r]_d& A^J \ar[r]_{u^J} & B^J}\] is again an absolute right lifting diagram. Given a square
\[\xymatrix{ X \ar[d]_{!} \ar[rr]^b \ar@{}[drr]|{\Downarrow\chi} & & B \ar[d]^-{\Delta} \\ 1 \ar[r]_-{d} & A^J \ar[r]_{u^J} & B^J} \] we first ``transpose across the adjunction,'' by composing with $f$ and the counit. 
\[\vcenter{\xymatrix{ X \ar[d]_-{!} \ar[rr]^b \ar@{}[drr]|{\Downarrow\chi} & & B \ar[d]^-\Delta \ar[r]^f & A \ar[d]^-\Delta  \\ 1 \ar[r]_-{d} & A^J \ar@{=}@/_3.5ex/[rr]^{\Downarrow\epsilon^J} \ar[r]^{u^J} & B^J \ar[r]^{f^J} & A^J}} = \vcenter{\xymatrix{ X \ar@{}[drr]|(.3){\exists !\Downarrow\zeta}|(.7){\Downarrow\lambda} \ar[d]_-{!} \ar[r]^b & B \ar[r]^f & A \ar[d]^-\Delta \\ 1 \ar[urr]_(0.4){\lim d} \ar[rr]_{d} & & A^J}} \] The universal property of the absolute right lifting diagram $\lambda \colon \Delta \lim \To d$ induces a unique factorisation $\zeta$, which may then be ``transposed back across the adjunction'' by composing with $u$ and the unit.
\[  \vcenter{\xymatrix{ X \ar@{}[drr]|(.3){\exists !\Downarrow\zeta}|(.7){\Downarrow\lambda} \ar[d]_-{!} \ar[r]^b & B \ar@{=}@/^3.5ex/[rr]_{\Downarrow\eta} \ar[r]|f & A \ar[d]^-\Delta \ar[r]_u & B \ar[d]^-\Delta \\ 1 \ar[urr]_(0.4){\lim d} \ar[rr]_-{d} & & A^J \ar[r]_{u^J} & B^J}}= \vcenter{\xymatrix{ X \ar[d]_-{!} \ar[rr]^b \ar@{}[drr]|{\Downarrow\chi} & & B \ar[d]^-\Delta \ar@{=}@/^3.5ex/[rr]_{\Downarrow\eta} \ar[r]_f & A \ar[d]^-\Delta \ar[r]_u & B \ar[d]^-\Delta  \\ 1 \ar[r]_-{d} & A^J \ar@{=}@/_3.5ex/[rr]^{\Downarrow\epsilon^J} \ar[r]^{u^J} & B^J \ar[r]^{f^J} & A^J
 \ar[r]_{u^J} & B^J}}  \] \[ = \vcenter{\xymatrix{ X \ar[d]_-{!} \ar[rr]^b \ar@{}[drr]|{\Downarrow\chi} & & B \ar[d]^-\Delta \ar@{=}@/^3.5ex/[rr]  &  & B \ar[d]^-\Delta  \\ 1 \ar[r]_-{d} & A^J \ar@{=}@/_3.5ex/[rr]^{\Downarrow\epsilon^J} \ar[r]^{u^J} & B^J \ar[r]|{f^J}  \ar@{=}@/^3.5ex/[rr]_{\Downarrow\eta^J}& A^J \ar[r]_{u^J} & B^J}}  = \vcenter{\xymatrix{ X \ar[d]_-{!} \ar[rr]^b \ar@{}[drr]|{\Downarrow\chi} & & B \ar[d]^-\Delta \\ 1 \ar[r]_-{d} & A^J \ar[r]_{u^J} & B^J}}\] Here the second equality is a consequence of the 2-functoriality of the simplicial cotensor, while the third is an application of a triangle identity for the adjunction $f^J \dashv u^J$. The pasted composite of $\zeta$ and $\eta$ is the desired factorisation of $\chi$ through $\lambda$. 

The proof that this factorization is unique, which again parallels the classical argument, is left to the reader: the essential point is that the transposes defined via these pasting diagrams are unique.
\end{proof}

The same argument also shows that a right adjoint preserves the limit of a family of diagrams $d \colon K \to A^J$. On account of Exercise \ref{exs:adjoint-equivalence}, we have the immediate corollary:

\begin{cor}\label{cor:equiv-pres-lim} Equivalences preserve limits.
\end{cor}

Moreover, ``completeness'' of $\infty$-categories is transferred along equivalences.

\begin{prop}\label{prop:completeness-equivalence} If $A \simeq B$ then any family of diagrams in $A$ that admits limits in $B$ also admits limits in $A$.
\end{prop}
\begin{proof}
Consider a family of diagrams $d\colon K \to A^J$ of shape $J$ in $A$ that admits limits in $B$ after composing with the equivalence $A \simeq B$, given by an absolute right lifting diagram:
\[ \xymatrix{ \ar@{}[drr]|(.7){\Downarrow \lambda} & & B \ar[d]^\Delta \\ K \ar[urr]^\lim  \ar[r]_d & A^J \ar[r]_{\sim} & B^J}\] We claim that the composite 2-cell 
\[ \xymatrix{ \ar@{}[drr]|(.7){\Downarrow \lambda} & & B \ar[d]^\Delta \ar[r]^\sim & A \ar[d]^\Delta \\ K \ar[urr]^\lim  \ar[r]_d & A^J\ar@{=}@/_3.5ex/[rr]^{\Downarrow\cong}  \ar[r]_{\sim} & B^J \ar[r]^\sim & A^J }\] again defines an absolute right lifting diagram, proving that the original family $d \colon K \to A^J$ admits limits in $A$. Promoting the equivalence to an adjoint equivalence, as in Exercise \ref{exs:adjoint-equivalence}, the proof of this universal property is very similar to the proof of Theorem \ref{thm:RAPL}; the remaining details are left as an exercise.
\end{proof}

\subsection{Colimits}

The theory of colimits of $J$-indexed diagrams in an $\infty$-category $A$  is dual to the theory of limits by reversing the direction of the 2-cells but not the 1-cells in the ambient homotopy 2-category $\h\lcat{K}$.

\begin{defn}[absolute left lifting] Given a cospan $C \xrightarrow{g} A \xleftarrow{f} B$, a functor $\ell \colon C \to B$ and a 2-cell \[\xymatrix{ \ar@{}[dr]|(.7){\Uparrow\lambda} & B \ar[d]^f \\ C \ar[ur]^\ell \ar[r]_g & A}\] define an \emph{absolute left lifting of $g$ through $f$} if any 2-cell as displayed below-left factors uniquely through $\lambda$ as displayed below-right
\[    \vcenter{\xymatrix{ X \ar[d]_c \ar[r]^b \ar@{}[dr]|{\Uparrow\chi} & B \ar[d]^f \\ C \ar[r]_g & A}} \mkern20mu = \mkern20mu \vcenter{\xymatrix{ X \ar[d]_c \ar[r]^b \ar@{}[dr]|(.3){\exists !\Uparrow}|(.7){\Uparrow\lambda} & B \ar[d]^f \\ C \ar[ur]|(.4)*+<2pt>{\scriptstyle \ell} \ar[r]_g & A}}
\]
\end{defn}

\begin{defn}[colimit]\label{defn:colimit} A \emph{colimit} of a $J$-indexed diagram in $A$ is an absolute left lifting of the diagram $d$ through the constant diagram functor $\Delta \colon A \to A^J$
\[ \xymatrix{ \ar@{}[dr]|(.7){\Uparrow\lambda} & A\ar[d]^\Delta \\ 1 \ar[r]_d \ar[ur]^{\colim d} & A^J}\] the 2-cell component of which  defines the \emph{colimit cone} $\lambda \colon d \To \Delta \colim d$.
\end{defn}

We leave the formulation of the evident duals of \ref{thm:totalization}, \ref{thm:RAPL}, \ref{cor:equiv-pres-lim}, and \ref{prop:completeness-equivalence} as an exercise.

\subsection{Arrow \texorpdfstring{$\infty$}{infinity}-categories}

By design, our definitions of limits and colimits of diagrams in an $\infty$-category and of adjunctions between $\infty$-categories are in a form that can be easily expressed internally to a 2-category. Consequently, we have not yet described  the universal properties encoded by the $\infty$-functors represented by limit or colimit elements.  Our aim in the next lecture will be to present equivalent definitions of limits, colimits, and adjunctions that can be given in these terms. These will require a more substantial use of the $\infty$-cosmos axiomatization than we have needed thus far and will also form the basis for our ``model independence'' results.

To prepare the way for this discussion, we now briefly introduce the $\infty$-categories that will serve as  a vehicle to encode these sorts of representable universal properties, at least in the very simplest case: these are \emph{arrow $\infty$-categories}. 

Here is the idea that motivates their importance. Recall that we use the terminal $\infty$-category $1$ to probe inside an $\infty$-category $A$, in the sense that an element of $A$ is defined to be a functor $a \colon 1 \to A$. We would like to be able to probe similarly for arrows in $A$. An approximate notion is given by a 2-cell of the form $\xymatrix{1 \ar@/^2ex/[r]^a \ar@/_2ex/[r]_{a'} \ar@{}[r]|{\Downarrow f} & A}$, but these correspond to arrows in the homotopy category of $A$, as in Definition \ref{defn:homotopy-category}, rather than arrows in $A$ itself. If we had a ``walking arrow'' $\infty$-category $\cattwo$, then we could define an arrow in $A$ to be a functor $f\colon \cattwo \to A$. The axioms of an $\infty$-cosmos do not guarantee this, but we can represent functors of this type in ``transposed'' form, as elements $f \colon 1 \to A^\cattwo$ in the arrow $\infty$-category that we now define.

\begin{defn}[arrow $\infty$-categories] For any $\infty$-category $A$, the simplicial cotensor 
\[ \xymatrix@C=30pt{ A^\cattwo \defeq A^{\Del^1} \ar@{->>}[r]^-{(p_1,p_0)} & {A^{\boundary\Delta^1}} \cong A \times A}\] defines the \emph{arrow $\infty$-category} $A^\cattwo$, equipped with an isofibration $(p_1,p_0)\colon A^\cattwo \tfib A \times A$, where $p_1 \colon A^\cattwo \tfib A$ denotes the codomain projection and $p_0 \colon A^\cattwo \tfib A$ denotes the domain projection.
\end{defn}

Using the notation $\cattwo \defeq \Delta^1$, the defining universal property of the simplicial cotensor asserts that the canonical map defines an isomorphism of quasi-categories
\[ \Fun(X, A^\cattwo) \stackrel{\cong}{\longrightarrow} \Fun(X,A)^{\cattwo}.\]  In particular, taking $X=A^\cattwo$, the identity functor $\id_{A^\cattwo}$ transposes to define a vertex in $\Fun(A^\cattwo,A)^\cattwo$ which represents a 2-cell
\begin{equation}\label{eq:generic-arrow} \xymatrix{ A^\cattwo \ar@{->>}@/^2ex/[r]^{p_0} \ar@{->>}@/_2ex/[r]_{p_1} \ar@{}[r]|-{\Downarrow\phi} & A}\end{equation} in the homotopy 2-category $\h\lcat{K}$.

There is an analogous categorical notion of cotensor with $\cattwo$: if $C$ is an ordinary 1-category, then $C^\cattwo$ is the category whose objects are morphisms in $C$ and whose morphisms are commutative squares. In particular, for any pair of $\infty$-categories, we can form the arrow category $\hom(X,A)^\cattwo$ of the hom-category between them. There is a canonical functor $\hom(X,A^\cattwo) \to \hom(X,A)^\cattwo$ obtained by applying $\hom(X,-)$ to the 2-cell \eqref{eq:generic-arrow} and appealing to the universal property of $\hom(X,A)^\cattwo$. Now the arrow $\infty$-category $A^\cattwo$ would be a strict $\cattwo$-cotensor in the homotopy 2-category $\h\lcat{K}$ if this map defined an isomorphism  $\hom(X,A^\cattwo) \stackrel{\cong}{\longrightarrow} \hom(X,A)^\cattwo$, analogous to the isomorphism of mapping quasi-categories that demonstrates that $A^\cattwo$ is a strict $\Delta^1$-cotensor in $\lcat{K}$. Instead: 

\begin{prop}[{\refI{prop:weak-cotensors}}]\label{prop:weak-cotensors} For any $\infty$-categories $X$ and $A$, the canonical functor
\[ \hom(X, A^\cattwo) \longrightarrow \hom(X,A)^{\cattwo}\] is \emph{smothering}: that is, surjective on objects, full, and conservative.\footnote{A functor is \emph{conservative} when it reflects isomorphisms.}
\end{prop}

Surjectivity on objects asserts that any 2-cell $\alpha \colon f \To g \colon X \to A$ is represented by a functor $\hat{\alpha} \colon X \to A^\cattwo$ with the property that the whiskered composite of the representing functor with the canonical 2-cell $\phi \colon p_0 \To p_1$ equals $\alpha$:
\[ \vcenter{\xymatrix{ X \ar[r]^-{\hat{\alpha}} & A^\cattwo \ar@{->>}@/^2ex/[r]^{p_0} \ar@{->>}@/_2ex/[r]_{p_1} \ar@{}[r]|-{\Downarrow\phi} & A}} = \vcenter{ \xymatrix{X \ar@/^2ex/[r]^f \ar@/_2ex/[r]_g \ar@{}[r]|{\Downarrow\alpha} & A}}\]
These representatives are not unique: 

\begin{exs} Show that any parallel pair of functors $\hat\alpha, \hat{\alpha}' \colon X \to A^\cattwo$ over $A \times A$ that are connected by an invertible 2-cell that projects to the identity along $(p_1,p_0) \colon A^\cattwo \tfib A \times A$ will necessarily represent the same 2-cell $\alpha \colon f \To g$.
\[ \xymatrix@=1.5em{ X \ar@/^2ex/[rr]^{\hat\alpha} \ar@/_2ex/[rr]_{{\hat\alpha}'} \ar@{}[rr]|{\Downarrow\cong} \ar[dr]_{(g,f)} & & A^\cattwo \ar@{->>}[dl]^{(p_1,p_0)} \\ & A \times A}\] Then use Proposition \ref{prop:weak-cotensors} to prove the converse: that any parallel pair of functors that represent the same 2-cell $\alpha \colon f \To g$ are connected by an isomorphism of this form. 

Lemma \ref{lem:ess.unique.1-cell.ind}, proven below, generalizes this result.
\end{exs}

This non-uniqueness of representing functors implies, in particular, that the functor of Proposition \ref{prop:weak-cotensors} cannot define an isomorphism of categories. Nonetheless, we are able to make substantial  use of the \emph{weak 2-dimensional universal property} of the arrow construction, as expressed by Proposition \ref{prop:weak-cotensors}, as we shall shortly discover.

\renewcommand\thesection{Lecture~\arabic{section}}
\section{Comma \texorpdfstring{$\infty$}{infinity}-categories and model independence}\label{sec:comma}
\renewcommand\thesection{\arabic{section}}

In the first two parts of this series, we have introduced notions of adjunctions between and limits and colimits of diagrams valued within $\infty$-categories, which are objects in some well-behaved universe we call an $\infty$-cosmos. In our treatment of adjunctions between $\infty$-categories and limits and colimits of diagrams valued in an $\infty$-category we have privileged definitions of these basic categorical notions that can be defined internally to any 2-category. Consequently, we have not yet seen the analogues of, for instance, the idea that an adjunction encodes a natural correspondence between certain arrows in a pair of $\infty$-categories, or the idea that a limit defines a terminal object in the $\infty$-category of cones. Our aim in this lecture is to present equivalent definitions in terms of these universal properties, utilizing additional structures in the homotopy 2-category guaranteed by  the $\infty$-cosmos axioms. The axioms imply that the homotopy 2-category of an $\infty$-cosmos admits the construction of the \emph{comma $\infty$-category} for any cospan of functors $C \xrightarrow{g} A \xleftarrow{f} B$. 

Immediately from their constructions, comma $\infty$-categories are preserved by functors of $\infty$-cosmoi. For a certain special class of functors, which we will call \emph{weak equivalences of $\infty$-cosmoi}, equivalences and comma $\infty$-categories are also created and reflected. Having encoded limits, colimits, and adjunctions as equivalences between comma $\infty$-categories, a corollary will be that these notions are preserved, reflected, and created by weak equivalences of $\infty$-cosmoi. In particular, these notions are invariant under change of models between quasi-categories, complete Segal spaces, Segal categories, and naturally marked simplicial sets.

\subsection{The weak 2-universal property of comma \texorpdfstring{$\infty$}{infinity}-categories}

Proposition \ref{prop:cartesian-closure} demonstrates that the homotopy 2-category of an $\infty$-cosmos has finite products, in the 2-categorical sense. Unlike the special case of $\h\Cat$, a general homotopy category will admit few 2-dimensional limit notions. However, certain simplicially-enriched limits in the $\infty$-cosmos provided by the axioms  \ref{qcat.ctxt.cof.def}\ref{qcat.ctxt.cof:a} and \ref{qcat.ctxt.cof:b} descend to a particular variety of weak 2-dimensional limits.  Arrow $\infty$-categories define a special case of a more general \emph{comma construction} that we now introduce.

\begin{defn}[comma $\infty$-categories]\label{defn:comma} Any pair of functors  $f\colon B\to A$ and $g\colon C\to A$ in an $\infty$-cosmos $\lcat{K}$ has an associated \emph{comma $\infty$-category}, constructed by the following pullback, formed in $\lcat{K}$:
  \begin{equation}\label{eq:comma-as-simp-pullback}
    \xymatrix@=2.5em{
      {f\comma g}\pbexcursion \ar[r]\ar@{->>}[d]_{(p_1,p_0)} &
      {A^\cattwo} \ar@{->>}[d]^{(p_1,p_0)} \\
      {C\times B} \ar[r]_-{g\times f} & {A\times A}
    }
  \end{equation}
The top horizontal functor represents a 2-cell 
\[
    \xymatrix@=10pt{
      & f \downarrow g \ar@{->>}[dl]_{p_1} \ar@{->>}[dr]^{p_0} \ar@{}[dd]|(.4){\phi}|{\Leftarrow}  \\ 
      C \ar[dr]_g & & B \ar[dl]^f \\ 
      & A}
\]
in the homotopy 2-category $\h\lcat{K}$, called the \emph{comma cone}, which is defined by composing with \eqref{eq:generic-arrow}. Note that, by construction, the map $(p_1,p_0) \colon f \comma g \tfib C \times B$ is an isofibration.
\end{defn}

As a simplicially-enriched limit in $\lcat{K}$, the $\infty$-category $f \comma g$ has a universal property expressed via a natural isomorphism of quasi-categories
\[ \Fun(X, f\comma g)\stackrel{\cong}{\longrightarrow} \Fun(X,f) \comma \Fun(X,g)\] for any $X \in \lcat{K}$, where the right-hand side is computed by the analogous pullback to \eqref{eq:comma-as-simp-pullback}, formed in $\qCat$. 

There is an analogous 2-categorical notion of comma object. The comma $\infty$-category $f \comma g$ would be a strict comma object in the homotopy 2-category $\h\lcat{K}$ if there were an isomorphism between the hom-category $\hom(X,f \comma g)$ and the comma category $\hom(X,f) \comma \hom(X,g)$ constructed from the pair of functors 
\[ \hom(X,C) \xrightarrow{\hom(X,g)} \hom(X,A) \xleftarrow{\hom(X,f)} \hom(X,B).\] Instead:

\begin{prop}[{\refI{prop:weakcomma}}]\label{prop:comma-smothering} For any object $X$,  the induced comparison functor of hom-categories
  \begin{equation*}
    \hom(X,f\comma g)\longrightarrow\hom(X,f)\comma\hom(X,g)
  \end{equation*}
is \emph{smothering}: surjective on objects, locally surjective on arrows, and conservative.
\end{prop}

\begin{obs}[the weak universal property of commas]\label{obs:comma-UP}
Explicitly, the weak universal property expressed by Proposition \ref{prop:comma-smothering} supplies three operations in the homotopy 2-category:
  \begin{enumerate}[label=(\roman*)]
    \item (1-cell induction) Given a 2-cell $\alpha \colon fb \To gc$
\[    \vcenter{ \xymatrix@=10pt{
      & X \ar[dl]_{c} \ar[dr]^{b} \ar@{}[dd]|(.4){\alpha}|{\Leftarrow}  \\ 
      C \ar[dr]_g & & B \ar[dl]^f \\ 
      & A}} \quad =\quad \vcenter{     \xymatrix@=10pt{ & X \ar@{-->}[d]^{a} \ar@/^2ex/[ddr]^b \ar@/_2ex/[ddl]_c \\ 
      & f \downarrow g \ar[dl]_{p_1} \ar[dr]^{p_0} \ar@{}[dd]|(.4){\phi}|{\Leftarrow}  \\ 
      C \ar[dr]_g & & B \ar[dl]^f \\ 
      & A}}
\] 
over the pair of functors $f$ and $g$, there exists a 1-cell $a\colon X\to f\comma g$, defined by \emph{1-cell induction}, so that $p_0a = b$, $p_1a = c$, and $\alpha = \phi a$.
    \item (2-cell induction) Given a pair of functors $a,a'\colon X\to f\comma g$ and a pair of 2-cells 
\[
    \vcenter{\xymatrix@=10pt{
      & {X}\ar[dl]_{a'}\ar[dr]^{a}
      \ar@{}[dd]|(.4){\tau_1}|{\Leftarrow} & \\
      {f\comma g}\ar[dr]_{p_1} & & 
      {f\comma g}\ar[dl]^{p_1} \\
      & C &
    }}
    \mkern30mu\text{and}\mkern30mu
    \vcenter{\xymatrix@=10pt{
      & {X}\ar[dl]_{a'}\ar[dr]^{a}
      \ar@{}[dd]|(.4){\tau_0}|{\Leftarrow} & \\
      {f\comma g}\ar[dr]_{p_0} & & 
      {f\comma g}\ar[dl]^{p_0} \\
      & B &
    }}
\]
  with the property that 
\[
      \xymatrix@=10pt{ & X \ar[dl]_{a'} \ar[dr]^a \ar@{}[dd]|(.4){\tau_1}|{\Leftarrow} & &  & \ar@{}[dd]|{\displaystyle =} & &  & X \ar[dl]_{a'} \ar[dr]^a  \ar@{}[dd]|(.4){\tau_0}|{\Leftarrow} \\ f \downarrow g  \ar[dr]_{p_1} & & f \downarrow g \ar[dl]|{p_1} \ar[dr]^{p_0}   \ar@{}[dd]|(.4){\phi}|{\Leftarrow} & & & &   f \downarrow g \ar[dl]_{p_1} \ar[dr]|{p_0}  \ar@{}[dd]|(.4){\phi}|{\Leftarrow}  & & f \downarrow g \ar[dl]^{p_0} \\ & C \ar[dr]_g & & B \ar[dl]^f & & C \ar[dr]_g & & B \ar[dl]^f & &  \\ & & A & &  &  & A}
\]
  then there exists a 2-cell $\tau \colon a \Rightarrow a'$, defined by \emph{2-cell induction}, satisfying the equalities 
  \[     \vcenter{\xymatrix@=10pt{
      & {X}\ar[dl]_{a'}\ar[dr]^{a}
      \ar@{}[dd]|(.4){\tau_1}|{\Leftarrow} & \\
      {f\comma g}\ar[dr]_{p_1} & & 
      {f\comma g}\ar[dl]^{p_1} \\
      & C &
    }}     \mkern10mu = \mkern10mu
        \vcenter{\xymatrix@=30pt{ X \ar@/^3ex/[d]^a \ar@/_3ex/[d]_{a'} \ar@{}[d]|(.4){\tau}|{\Leftarrow}  \\ f \downarrow g \ar[d]^{p_1} \\ C}}
      \mkern20mu\text{and}\mkern20mu
  \vcenter{\xymatrix@=10pt{
      & {X}\ar[dl]_{a'}\ar[dr]^{a}
      \ar@{}[dd]|(.4){\tau_0}|{\Leftarrow} & \\
      {f\comma g}\ar[dr]_{p_0} & & 
      {f\comma g}\ar[dl]^{p_0} \\
      & B &
    }}    \mkern10mu = \mkern10mu
    \vcenter{\xymatrix@=30pt{ X \ar@/^3ex/[d]^a \ar@/_3ex/[d]_{a'} \ar@{}[d]|(.4){\tau}|{\Leftarrow}  \\ f \downarrow g \ar[d]^{p_0} \\ B}}  
  \]

    \item (conservativity) Any 2-cell $\xymatrix{ X \ar@/^2ex/[r]^{a} \ar@/_2ex/[r]_{a'} \ar@{}[r]|-{\Downarrow\tau} & { f \comma g}}$  with the property that the whiskered 2-cells $p_0\tau$ and $p_1\tau$ are both isomorphisms is also an isomorphism.
  \end{enumerate}
\end{obs}

\begin{lem}[{\refI{lem:1cell-ind-uniqueness}}]\label{lem:ess.unique.1-cell.ind}
A parallel pair of functors over $C \times B$
\[ \xymatrix@=1.5em{ X \ar@<.5ex>[rr]^a \ar@<-.5ex>[rr]_{a'} \ar[dr]_{(c,b)} && f \comma g \ar@{->>}[dl]^{(p_1,p_0)} \\ & C \times B}\] 
are isomorphic over $C \times B$ if and only if $a$ and $a'$  both enjoy the same defining properties as 1-cells induced by the weak 2-universal property of $f\comma g$, i.e., if these functors satisfy $\phi a = \phi a'$.\footnote{For  $a$ and $a'$ to define a parallel pair of functors over $C \times B$ we must have $p_0a =p_0a'$ and $p_1 a = p_1a'$.} That is, 2-cells of the form displayed on the left
\[ \vcenter{\xymatrix@=10pt{
      & X \ar[dl]_{c} \ar[dr]^{b} \ar@{}[dd]|(.4){\alpha}|{\Leftarrow}  \\ 
      C \ar[dr]_g & & B \ar[dl]^f \\ 
      & A}} \qquad \leftrightsquigarrow \qquad \vcenter{     \xymatrix@=10pt{ & X \ar@{->}[dd]^{a} \ar[dr]^b \ar[dl]_c \\ C & & B \\ & f \comma g \ar@{->>}[ul]^{p_1} \ar@{->>}[ur]_{p_0}}}\]
stand in bijection with isomorphism classes of maps of spans, as displayed on the right.
\end{lem}
\begin{proof} An isomorphism $\tau$ between a parallel pair $a,a' \colon X \to f \comma g$ with the same defining conditions is induced by 2-cell induction and conservativity from the pair of identity 2-cells $p_1\tau = \id_c$ and $p_1 \tau = \id_b$. Conversely, by uniqueness of composites of pasting diagrams, if there is any isomorphism $\tau \colon a \cong a'$ over an identity on $C \times B$ then whiskered composites $\phi a$ and $\phi a'$ are equal.
\end{proof}

\begin{lem}[{\refV{lem:equiv-to-comma}}]\label{lem:equiv-to-comma} 
For any pair of functors $C \xrightarrow{g} A \xleftarrow{f} B$, the universal property of \ref{obs:comma-UP} characterizes a unique equivalence class of isofibrations $E \tfib C \times B$. That is:
\begin{enumerate}[label=(\roman*)]
\item\label{itm:equiv-to-comma-i} Any pair of isofibrations $E \tfib C \times B$ and $E' \tfib C \times B$ equipped with a comma cone that enjoys the weak universal property of \ref{obs:comma-UP} are equivalent, over $C \times B$.
\item\label{itm:equiv-to-comma-ii} Any isofibration $E \tfib C \times B$ that is equivalent over $C\times B$ to the comma $f \comma g \tfib C \times B$, has the universal property of \ref{obs:comma-UP}.
\end{enumerate}
\end{lem}
\begin{proof}
The proof of \ref{itm:equiv-to-comma-i} is an elementary exercise in the application of the universal property of \ref{obs:comma-UP}, paralleling the standard proof that a strictly-defined limit is unique up to isomorphism. For \ref{itm:equiv-to-comma-ii}, an equivalence $E \we f\comma g$ defines a representable equivalence in the homotopy 2-category. Thus, for any object $X$, we have a composable pair of functors
\begin{equation}\label{eq:equiv-to-comma-smothering} \hom(X,E) \we \hom(X,f \comma g) \longrightarrow \hom(X,f) \comma \hom(X,g),\end{equation} the first being an equivalence and the second being smothering, and so the composite is full, conservative, and essentially surjective on objects. To show that it is in fact surjective on objects, we make use of the fact that the equivalence $E \simeq f \comma g$ is fibered over $C \times B$. Any object in the codomain of \eqref{eq:equiv-to-comma-smothering} lifts to a representing functor $X \to f \comma g$, which is isomorphic over $C \times B$ to a functor of the form $X \xrightarrow{e} E \we f \comma g$. By Lemma \ref{lem:ess.unique.1-cell.ind}, $e$ defines the desired preimage in $\hom(X,E)$. 
\end{proof}

\subsection{Hom-spaces and groupoidal objects}

The comma construction allows us to define hom-spaces between a pair of elements in an $\infty$-category.

\begin{defn}[hom-spaces]\label{defn:hom-space} Given a pair of elements $a,a' \colon 1\to A$ in an $\infty$-category $A$, their \emph{hom-space} is the comma $\infty$-category $a \comma a'$ defined by the pullback
\[ \xymatrix{ a \comma a' \ar@{->>}[d] \ar[r] \pbexcursion & A^\cattwo \ar@{->>}[d]^{(p_1,p_0)} \\ 1 \ar[r]_-{(a',a)} & A \times A}\] 
\end{defn}

We refer to the $\infty$-category $a \comma a'$ as a hom-space because, as an easy corollary of the conservativity of its weak universal property,  it defines a \emph{groupoidal object} in the $\infty$-cosmos.

\begin{defn}\label{defn:groupoidal-object} We say an object $E$ in an $\infty$-cosmos $\lcat{K}$ is \emph{groupoidal} if the following equivalent conditions are satisfied:
\begin{enumerate}[label=(\roman*)]
\item\label{itm:groupoidal-object:i} $E$ is a groupoidal object in the homotopy 2-category $\h\lcat{K}$, that is,  every 2-cell with codomain $E$ is invertible.
\item\label{itm:groupoidal-object:ii} For each $X \in \lcat{K}$, the hom-category $\hom(X,E)$ is a groupoid. 
\item\label{itm:groupoidal-object:iii}  For each $X \in \lcat{K}$, the mapping quasi-category $\Fun(X,E)$ is a Kan complex.
\item\label{itm:groupoidal-object:iv} The isofibration $E^\iso \tfib E^\cattwo$, induced by the inclusion of simplicial sets $\cattwo\inc\iso$, is a trivial fibration.
\end{enumerate}
Here \ref{itm:groupoidal-object:ii} is an unpacking of \ref{itm:groupoidal-object:i}. The equivalence of \ref{itm:groupoidal-object:ii} and \ref{itm:groupoidal-object:iii} is a well-known result of Joyal \cite[1.4]{Joyal:2002:QuasiCategories}. Condition \ref{itm:groupoidal-object:iv} is equivalent to the assertion that $\Fun(X,E)^\iso \tfib \Fun(X,E)^\cattwo$ is a trivial fibration between quasi-categories for all $X$. If this is a trivial fibration, then surjectivity on vertices implies that every 1-simplex in $\Fun(X,E)$ is an isomorphism, proving \ref{itm:groupoidal-object:iii}. As $\cattwo \inc\iso$ is a weak homotopy equivalence, \ref{itm:groupoidal-object:iii} implies  \ref{itm:groupoidal-object:iv}.
\end{defn}

\begin{rmk}
In the $\infty$-cosmoi whose objects model $(\infty,1)$-categories, we posit that the groupoidal objects are precisely the corresponding $\infty$-\emph{groupoids}. For instance, in the $\infty$-cosmos for quasi-categories, an object is groupoidal if and only if it is a Kan complex. In the $\infty$-cosmos for naturally marked simplicial sets, an object is groupoidal if and only if it is a Kan complex with every edge marked.
\end{rmk}

\subsection{Commas representing functors}

For any functors between $\infty$-categories, e.g., $f \colon B \to A$ and $u \colon A \to B$, the following pullbacks in $\lcat{K}$
\[ \xymatrix@=1.5em{ f \comma A \ar[d]_{(p_1,p_0)} \ar[r] \pbexcursion &  A^\cattwo \ar[d]  & & B \comma u \ar[d]_{(q_1,q_0)} \ar[r] \pbexcursion &  B^\cattwo \ar[d] \\ A \times B \ar[r]_{\id_A \times f} & A \times A & & A \times B \ar[r]_{u \times \id_B} & B \times B}\]
define the comma $\infty$-categories $f \comma A$ and $B \comma u$, each of which is equipped with an isofibration to $A \times B$. The horizontal functors represent 2-cells
\begin{equation}\label{eq:adj-comma-cone}
\xymatrix@=15pt{ & f \comma A \ar@{->>}[dl]_{p_1} \ar@{->>}[dr]^{p_0} \ar@{}[d]|(.6){\Leftarrow\alpha} & && & B \comma u \ar@{->>}[dl]_{q_1} \ar@{->>}[dr]^{q_0} \ar@{}[d]|(.6){\Leftarrow\beta} \\ A & & B \ar[ll]^f && A \ar[rr]_u & & B}
\end{equation}
in the homotopy 2-category which satisfy the following weak universal property derived from Lemma \ref{lem:ess.unique.1-cell.ind}:
\begin{itemize}
\item 2-cells $ \chi \colon fb \To a$ stand in bijection with isomorphism classes of maps of spans:
\[\vcenter{\xymatrix@=15pt{
      & X \ar[dl]_{a} \ar[dr]^{b} \ar@{}[d]|{\chi}|(.7){\Leftarrow}  \\ 
     A & & B \ar[ll]^f }} \qquad \leftrightsquigarrow \qquad     \vcenter{\xymatrix@=10pt{ & X \ar@{->}[dd]^{x} \ar[dr]^b \ar[dl]_a \\ A & & B \\ & f \comma A \ar@{->>}[ul]^{p_1} \ar@{->>}[ur]_{p_0}}}\]
\item 2-cells $\zeta \colon b \To ua$ stand in bijection with isomorphism classes of maps of spans:
\[\vcenter{\xymatrix@=15pt{
      & X \ar[dl]_{a} \ar[dr]^{b} \ar@{}[d]|{\zeta}|(.7){\Leftarrow}  \\ 
     A \ar[rr]_u & & B }}\qquad \leftrightsquigarrow \qquad   \vcenter{\xymatrix@=10pt{ & X \ar@{->}[dd]^{y} \ar[dr]^b \ar[dl]_a \\ A & & B \\ & B \comma u \ar@{->>}[ul]^{q_1} \ar@{->>}[ur]_{q_0}}}\]

\end{itemize}
Here the isomorphism classes are defined with respect to natural isomorphisms $x \cong x'$ projecting to an identity over $A \times B$. The bijection is implemented by whiskering a functor over $A \times B$ with the 2-cells of \eqref{eq:adj-comma-cone}.

If $f \dashv u$, then $\hom(X,f) \dashv \hom(X,u)$ defines an adjunction of categories for any $\infty$-category $X$. This tells us that there is a natural bijection between 2-cells 
\[\vcenter{\xymatrix@=15pt{
      & X \ar[dl]_{a} \ar[dr]^{b} \ar@{}[d]|{\chi}|(.7){\Leftarrow}  \\ 
     A & & B \ar[ll]^f }}\qquad \leftrightsquigarrow\qquad \vcenter{\xymatrix@=15pt{
      & X \ar[dl]_{a} \ar[dr]^{b} \ar@{}[d]|{\zeta}|(.7){\Leftarrow}  \\ 
     A \ar[rr]_u & & B }}\] implemented by pasting along the bottom with the unit or with the counit. By a Yoneda-style argument, this yields the following result, which we instead prove directly. 

\begin{prop}[{\refI{prop:adjointequiv}}]\label{prop:adjoint-comma-equiv-I} If $\xymatrix{ B \ar@<1ex>[r]^f \ar@{}[r]|\perp & A \ar@<1ex>[l]^u}$ is an adjunction between $\infty$-categories, then there is a fibered equivalence between the comma $\infty$-categories:
\begin{equation*}
  \xymatrix@=1em{
    {f\comma A}\ar@{->>}[dr]_(0.4){(p_1,p_0)}\ar[rr]^{w}_{\sim}
    && {B\comma u}\ar@{->>}[dl]^(0.4){(q_1,q_0)} \\
    & {A\times B}&
  }
\end{equation*}
\end{prop}

The proof of this result mirrors the standard construction of the adjoint-transpose bijection using the unit and counit of an adjunction. 

\begin{proof}
The composite 2-cells displayed on the left of the pasting equalities induce functors  $w'\colon B \comma u \to f \comma A$ and $w\colon f \comma A \to B \comma u$
\begin{equation*}\xymatrix@C=10pt{ & B \comma u \ar[dl]_{q_1} \ar[dr]^{q_0} \ar@{}[d]|(.6){\Leftarrow\beta} &  &  & B \comma u \ar[d]^{w'}  & && & f \comma A \ar[dl]_{p_1} \ar[dr]^{p_0} \ar@{}[d]|(.6){\Leftarrow\alpha} & & & f \comma A \ar[d]^{w}  \\ A \ar@{=}[dr] \ar[rr]^u & \ar@{}[d]|{\Leftarrow\epsilon} & B  \ar[dl]^f & = & f \comma A \ar[dl]_{p_1} \ar[dr]^{p_0} \ar@{}[d]|(.6){\Leftarrow\alpha} & & & A  \ar[dr]_u & \ar@{}[d]|{\Leftarrow\eta} & B \ar@{=}[dl] \ar[ll]_f & = & B \comma u \ar[dl]_{q_1} \ar[dr]^{q_0} \ar@{}[d]|(.6){\Leftarrow\beta}  \\ & A & & A  & & B \ar[ll]^f & & & B & & A \ar[rr]_u & & B}\end{equation*} 
that commute with the projections to $A \times B$.
\begin{equation*}
    \xymatrix@=1em{ f\comma A \ar@{->>}[dr]_{(p_1,p_0)} \ar@/^1ex/[rr]^{w} & & B \comma u \ar@{->>}[dl]^{(q_1,q_0)} \ar@/^1ex/[ll]^{w'} \\ & A \times B} 
\end{equation*}
Supposing $f \dashv u$, we have the following series of pasting equalities:
\begin{equation*}
\xymatrix@C=10pt{ & f \comma A \ar[d]^{w} & & & f \comma A \ar[d]^{w} & & & f \comma A \ar[dl]_{p_1} \ar[dr]^{p_0}\ar@{}[d]|(.6){\Leftarrow\alpha} & \ar@{}[dr]|*+{=} & & f \comma A  \ar[dl]_{p_1} \ar[dr]^{p_0} \ar@{}[d]|(.6){\Leftarrow\alpha} \\ & B \comma u \ar[d]^{w'} & \ar@{}[d]|*+{=} & & B \comma u \ar[dl]_{q_1} \ar[dr]^{q_0} \ar@{}[d]|(.6){\Leftarrow\beta}  & {=}  & A  \ar[drr]|u \ar@{=}[d] &\ar@{}[dr]|(.4){\Leftarrow\eta} & B \ar[ll]_f \ar@{=}[d]  & A & & B \ar[ll]^f \\ & f \comma A \ar[dl]_{p_1} \ar[dr]^{p_0} \ar@{}[d]|(.6){\Leftarrow\alpha} & & A \ar@{=}[dr] \ar[rr]^u & \ar@{}[d]|(0.4){\Leftarrow\epsilon}  & B  \ar[dl]^f  & A \ar@{}[ur]|(0.4){\Leftarrow\epsilon} &   & B \ar[ll]^f  \\ A & & B \ar[ll]^f & & A & & & &  }
\end{equation*} 
in which the last step is an application of one of the triangle identities.  This tells us that the endo-1-cells $w'w$ and $\id_{f\comma A}$ on the object $(p_1,p_0)\colon f\comma A\tfib A\times B$ both map to the same 2-cell $\alpha$ under the whiskering operation. By Lemma \ref{lem:ess.unique.1-cell.ind}, it follows that these are connected via a natural isomorphism  $w'w$ and $\id_{f\comma A}$ over $A \times B$. A dual argument provides a natural isomorphism $ww' \cong \id_{B \comma f}$ over $A \times B$, defining a fibered equivalence $f \comma A \simeq B \comma u$ over  $A\times B$.
\end{proof}

\begin{cor} If $f \colon B \to A$ and $u \colon A \to B$ are functors so that $f \dashv u$, then for any pair of elements $a \colon 1 \to A$ and $b \colon 1 \to B$, the hom-spaces $fa \comma b$ and $a \comma ub$ are equivalent.
\end{cor}
\begin{proof}
Fibered equivalences can be pulled back along any functor to define another fibered equivalence. Pulling back the equivalence of Proposition  \ref{prop:adjoint-comma-equiv-I} along a pair of elements $1 \xrightarrow{(a,b)} A \times B$
  \[  \begin{xy}
      0;<5pc,0pc>:
      *{\xybox{
        \POS(0,0)*+{1}="A",
        \POS(1.3,0)*+{A \times B}="B",
        \POS(-0.3,1.3)*+{fb \comma a}="F",
        \POS(-0.25,1.25)*+{\pbcorner},
        \POS(0.45,0.65)*+{\pbcorner},
        \POS(0.4,0.7)*+{a \comma ub}="F'",
        \POS(1,1.3)*+{f \comma A}="E",
        \POS(1.6,0.7)*+{B \comma u}="E'",
    \ar"A";"B"_-{(a,b)}
    \ar@{->>}@/_0.2pc/"F";"A"
     \ar@{->>}@/^0.2pc/"F'";"A"
     \ar@{->>}@/_0.2pc/"E";"B"_(0.7){(p_1,p_0)}|!{"F'";"E'"}\hole
       \ar@{->>}@/^0.2pc/"E'";"B"^(0.4){(p_1,p_0)}
        \ar"F";"E"
 \ar"F'";"E'"
 \ar"E";"E'"^*{\rotatebox{145}{$\labelstyle\sim$}}
 \ar"F";"F'"^*{\rotatebox{145}{$\labelstyle\sim$}}
      }}
    \end{xy}\]
we obtain an equivalence of hom-spaces $fb \comma a \simeq b \comma ua$, the former in the $\infty$-category $A$ and the latter in the $\infty$-category $B$.
\end{proof}

The converse to Proposition \ref{prop:adjoint-comma-equiv-I} follows from a general result that will have other applications.

\subsection{Commas and absolute lifting diagrams}

A 2-cell $\lambda \colon f \ell \To g$ induces a functor 
\begin{equation}\label{eq:abs-R-lifting} \vcenter{\xymatrix{ \ar@{}[dr]|(.7){\Downarrow\lambda} & B \ar[d]^f \\ C \ar[ur]^\ell \ar[r]_g & A}} \qquad \rightsquigarrow\qquad  \vcenter{\xymatrix@=1em{ B \comma \ell \ar[rr]^w \ar@{->>}[dr]_(.4){(p_1,p_0)} & & f \comma g\ar@{->>}[dl]^(.4){(p_1,p_0)} \\ & C \times B}}\end{equation}  between comma $\infty$-categories defined by 1-cell induction for the comma $f \comma g$ from the pasted composite of the comma cone for $B \comma \ell$ and $\lambda$.
\[   \vcenter{\xymatrix@=1.2em{
    & {B\comma\ell}\ar[dl]_{p_1}\ar[dr]^{p_0} & \\
    {C} \ar[dr]_{g}\ar[rr]|*+{\scriptstyle\ell} && {B}\ar[dl]^{f} \\  
    & A &  
    \ar@{} "1,2";"3,2" |(0.3){\Leftarrow\phi} |(0.7){\Leftarrow\lambda}
  }}
  \mkern 20mu = \mkern20mu
  \vcenter{\xymatrix@=1.2em{
    & {B\comma\ell}\ar[d]^{w}\ar@/_1.5ex/[ddl]_{p_1}\ar@/^1.5ex/[ddr]^{p_0} & \\
    & {f\comma g}\ar[dl]^{p_1}\ar[dr]_{p_0} & \\
    {C}\ar[dr]_{g} & & {B}\ar[dl]^{f} \\
    & {A} & 
    \ar@{} "2,2";"4,2" |{\Leftarrow\phi}
  }}\]

A Yoneda-style argument, making use of Lemmas \ref{lem:ess.unique.1-cell.ind} and \ref{lem:equiv-to-comma}, proves the following result:

\begin{prop}[{\refI{prop:absliftingtranslation}, \refI{prop:absliftingtranslation2}}]\label{prop:abs-lifting-via-commas} The data of \eqref{eq:abs-R-lifting} defines an absolute right lifting diagram in $\h\lcat{K}$ if and only if the induced map $w\colon B \comma \ell \to f \comma g$ is an equivalence over $C \times B$. Conversely, an equivalence \[ \xymatrix@=1em{ B \comma \ell \ar[rr]^w_\sim \ar@{->>}[dr]_(.4){(p_1,p_0)} & & f \comma g\ar@{->>}[dl]^(.4){(p_1,p_0)} \\ & C \times B}\] over $C \times B$ induces a canonical 2-cell $\lambda \colon f\ell \To g$ that defines an absolute right lifting of $g$ through $f$. 
\end{prop}

\subsection{Adjunctions, limits, and colimits via commas}

Special cases of Proposition \ref{prop:abs-lifting-via-commas} provide characterizations of adjunctions and (co)limits as equivalences between comma $\infty$-categories.

\begin{prop}[{\refI{prop:adjointequivconverse}}]\label{prop:adjoint-comma-equiv-II} If $f \colon B \to A$ and $u \colon A \to B$ are functors so that there exists a fibered equivalence between the comma $\infty$-categories:
\begin{equation*}
  \xymatrix@=1em{
    {f\comma A}\ar@{->>}[dr]_(0.4){(p_1,p_0)}\ar[rr]^{w}_{\sim}
    && {B\comma u}\ar@{->>}[dl]^(0.4){(q_1,q_0)} \\
    & {A\times B}&
  }
\end{equation*} then $f \dashv u$.
\end{prop}
\begin{proof}
By Proposition \ref{prop:abs-lifting-via-commas}, the fibered equivalence induces an absolute right lifting diagram \[ \xymatrix{ \ar@{}[dr]|(.7){\Downarrow\epsilon} & B \ar[d]^f \\ A \ar[ur]^u \ar@{=}[r] & A}\] Exercise \ref{exs:adj-as-abs-lifting} then implies that $f \dashv u$, with $\epsilon \colon fu \To \id_A$ as the counit.
\end{proof}

\begin{defn}[the $\infty$-category of cones] Given a $J$-indexed diagram $d \colon 1 \to A^J$ in an $\infty$-category $A$, the \emph{$\infty$-category of cones over $d$} is the comma $\infty$-category $\Delta \comma d$ formed over the cospan
\[  \xymatrix@=10pt{
      &\Delta \comma d \ar@{->>}[dl]_{p_1} \ar@{->>}[dr]^{p_0} \ar@{}[dd]|(.4){\phi}|{\Leftarrow}  \\ 
      1 \ar[dr]_d & & A \ar[dl]^\Delta \\ 
      & A^J}\]
      By the defining simplicial pullback \eqref{eq:comma-as-simp-pullback}, the data of an element in $\Delta \comma d$ is comprised of an element $a \colon 1 \to A$ (the summit) together with an element of the hom-space from $\Delta a$ to $d$ in $A^J$ (the cone).
\end{defn}

Specializing Proposition \ref{prop:abs-lifting-via-commas}, we have:

\begin{prop}\label{prop:limit-comma-equiv} An element $\ell \colon 1 \to A$ defines a limit for a diagram $d \colon 1 \to A^J$ if and only if there is a fibered equivalence  between the comma $\infty$-category represented by $\ell$ and the $\infty$-category of cones over $d$:
\begin{equation*}
  \xymatrix@=1em{
    {A \comma \ell }\ar@{->>}[dr]_{p_0}\ar[rr]^{\sim}
    && {\Delta \comma d}\ar@{->>}[dl]^{p_0} \\
    & {A}&
  }
\end{equation*} 
\end{prop}

The conclusion of Proposition \ref{prop:limit-comma-equiv} asserts that the $\infty$-category of cones over $d$ is \emph{represented by} the element $\ell \colon 1 \to A$.

\begin{exs} Specializing to the case $J= \emptyset$, show that an element $t \colon 1 \to A$ is terminal if and only if the projection $p_0 \colon A \comma t \trvfib A$ is a equivalence, and thus a trivial fibration.
\end{exs}

\begin{exs}\label{exs:terminal-in-representable} Use 1-cell induction, 2-cell induction, and 2-cell conservativity for the comma $A \comma \ell$ associated to an element $\ell \colon 1 \to A$ to show that the identity at $\ell$ defines a terminal element ${\id_\ell} \colon 1 \to A \comma \ell$, in the sense of Definition \ref{defn:terminal-element}.
\end{exs}

A fibered version of Exercise \ref{exs:terminal-in-representable} (\refI{prop:right.liftings.as.fibred.terminal.objects}) proves the following:

\begin{prop}[{\refI{prop:limits.as.terminal.objects}}]\label{prop:limits.as.terminal.objects} A limit of a diagram $d \colon 1 \to A^J$ defines a terminal element in the $\infty$-category $\Delta \comma d \tfib A$ of cones over $d$. Conversely, a terminal element in the $\infty$-category of cones defines a limit for $d$. 
\end{prop}

\begin{rmk} The $\infty$-category of \emph{cones in $A$ over any $J$-indexed diagram} is the comma $\Delta \comma A^J \tfib A^J \times A$. Pulling back along an element $d \colon 1 \to A^J$ defines the $\infty$-category $\Delta \comma d$ of cones over $d$.  The defining simplicial pullback \eqref{eq:comma-as-simp-pullback} for $\Delta \comma A^J$ reveals that it is isomorphic to the simplicial cotensor $A^{\Del^0 \diamond J}$, where ``$\diamond$'' is Joyal's ``fat join'' construction. For any pair of simplicial sets $I$ and $J$, there is a weak equivalence $I \diamond J \we I \join J$  in the Joyal model structure under the disjoint union $I \coprod J$ from the \emph{fat join} $I \diamond J$  to the \emph{join} $I \join J$. Taking cotensors, this induces a fibered equivalence $A^{\Del^0 \join J} \we \Delta \comma A^J$ over $A^J \times A$, which pulls back to define an equivalence $\slicer{A}{d} \we \Delta \comma d$ between Joyal's \emph{slice $\infty$-category} and the $\infty$-category of cones over $d$; see \S\refI{subsec:join}. This is the geometrical basis for the proof that characterization of the limit of a diagram valued in a quasi-category given in Definition \ref{defn:limit} and re-expressed by Proposition \ref{prop:limits.as.terminal.objects} agrees with the Joyal's original definition (\refI{prop:limits.are.limits}).
\end{rmk}

\subsection{Model independence of basic \texorpdfstring{$\infty$}{infinity}-category theory}

We have seen that comma $\infty$-categories can be used to encode various universal properties including:
\begin{itemize}
\item the existence of an adjunction between a pair of functors $u \colon A \to B$ and $f \colon B \to A$
\item the property that an element $\ell \colon 1 \to A$ defines a limit for a diagram $d \colon 1 \to A^J$
\end{itemize}
We will now see that any categorical property that can be captured by the existence of a fibered equivalence between comma $\infty$-categories is ``model independent'' in the sense that it is preserved by any functor of $\infty$-cosmoi and reflected by those functors that define \emph{weak equivalences of $\infty$-cosmoi}.

Recall, a \emph{functor of $\infty$-cosmoi} $F \colon \lcat{K} \to \lcat{L}$ is a simplicial functor that preserves the limits listed  in~\ref{qcat.ctxt.cof.def}\ref{qcat.ctxt.cof:a} and the class of isofibrations, and hence also the classes of equivalences and  trivial fibrations.  A functor $F \colon \lcat{K} \to \lcat{L}$ of $\infty$-cosmoi induces a 2-functor $\h{F} \defeq \ho_*F \colon \h\lcat{K} \to \h\lcat{L}$ between their homotopy 2-categories.

\begin{prop}[{\refV{prop:induced-2-functor}}]\label{prop:induced-2-functor} A functor $F \colon \lcat{K} \to \lcat{L}$ of $\infty$-cosmoi induces a 2-functor $\h{F} \colon \h\lcat{K} \to \h\lcat{L}$ between their homotopy 2-categories that preserves  adjunctions, equivalences, isofibrations, trivial fibrations, groupoidal objects, products, and comma objects.
\end{prop}
\begin{proof}
Any 2-functor preserves  adjunctions and equivalences. Preservation of isofibrations and products  are direct consequences of the hypotheses in Definition \ref{defn:qcat-ctxt-functor}; recall that the class of trivial fibrations in this intersection of the classes of isofibrations and equivalences. Preservation of groupoidal objects is a consequence of the characterization  \ref{defn:groupoidal-object}\ref{itm:groupoidal-object:iv}. Preservation of commas follows from the construction of \eqref{eq:comma-as-simp-pullback}, which is preserved by a functor of $\infty$-cosmoi, and Lemma \ref{lem:equiv-to-comma}\ref{itm:equiv-to-comma-i}, which says that all commas are equivalent an $\infty$-category constructed by the simplicial pullback formula.
\end{proof}

\begin{defn}[weak equivalences of $\infty$-cosmoi] A functor $F \colon \lcat{K} \to \lcat{L}$ of $\infty$-cosmoi is a \emph{weak equivalence} when it is:
\begin{enumerate}[label=(\alph*)]
\item surjective on objects up to equivalence: i.e., if for every $X \in \lcat{L}$, there is some $A \in \lcat{K}$ so that $FA\simeq X \in \lcat{L}$.
\item a local equivalence of quasi-categories: i.e., if for every pair $A,B \in \lcat{K}$, the map  $\Fun(A,B) \we \Fun(FA,FB)$ is an equivalence of quasi-categories.
\end{enumerate}
\end{defn}

\begin{ex}
The following define weak equivalences of $\infty$-cosmoi:
\begin{itemize}
\item The underlying quasi-category functor $\Fun(1,-)\colon \CSS \to \qCat$ that takes a complete Segal space to its $0\th$ row (see example \refIV{ex:CSS-cosmos}).
\item The functor $t^! \colon \qCat \to \CSS$ defined in example \refIV{ex:other-CSS-functor}.
\item The underlying quasi-category functor $\Fun(1,-)\colon \Segal \to \qCat$ that takes a Segal category to its $0\th$ row (see example \refIV{ex:segal-cosmos}). 
\item The functor $d_* \colon \qCat \to \Segal$ defined by Joyal and Tierney \cite{Joyal:2007hb}.
\item The underlying quasi-category functor $\Fun(1,-) \colon \sSet_+\to \qCat$ that carries a naturally marked simplicial set to its underlying quasi-category (see example \refIV{ex:marked-cosmos}).
\item The functor $(-)^\natural \colon \qCat \to \sSet_+$ that gives a quasi-category its ``natural'' marking.
\item The functor $\CSS \to \Segal$ that ``discretizes'' the 0th space of a complete Segal space. This commutes with the underlying quasi-category functors.
\end{itemize}
\end{ex}

\begin{prop}\label{prop:we} If $F$ is a weak equivalence of $\infty$-cosmoi, then the induced 2-functor $\h{F} \colon \h\lcat{K} \to \h\lcat{L}$
\begin{enumerate}[label=(\roman*)]
\item\label{itm:we-bi-equiv} defines a biequivalence $\h{F} \colon \h\lcat{K} \to \h\lcat{L}$: i.e., the 2-functor $\h{F}$ is surjective on objects up to equivalence and defines a local equivalence of categories $\hom(A,B) \stackrel{\simeq}{\longrightarrow}\hom(FA,FB)$ for all $A,B \in \lcat{K}$.
\item\label{itm:we-iso} induces a bijection between isomorphism classes of parallel functors: for all $A, B \in \lcat{K}$, the functor $\hom(A,B) \stackrel{\simeq}{\longrightarrow} \hom(FA,FB)$ induces a bijection on isomorphism classes of objects. 
\item\label{itm:we-groupoidal} preserves and reflects groupoidal objects: $A \in \lcat{K}$ is groupoidal if and only if $FA \in \lcat{L}$ is groupoidal.
\item\label{itm:we-equiv} preserves and reflects equivalence: $A \simeq B \in \lcat{K}$ if and only if $FA \simeq FB \in \lcat{L}$.
\item\label{itm:we-equivs}  preserves and reflects equivalences: $f \colon A \to B \in \lcat{K}$ is an equivalence if and only if $Ff \colon FA \to FB \in \lcat{L}$ is an equivalence.
\item\label{itm:we-comma} preserves and reflects comma objects: given $E \tfib C \times B$ and $C \xrightarrow{g} A \xleftarrow{f} B$ in $\lcat{K}$, then $E \simeq f \comma g$ over $C \times B$ if and only if  $FE \simeq Ff \comma Fg \cong F(f \comma g)$ over $FC \times FB$.
\end{enumerate}
\end{prop}
\begin{proof}
The homotopy category functor $\ho \colon \qCat \to \Cat$ carries equivalences of quasi-categories to equivalences of categories; thus, the local equivalence of quasi-categories $\Fun(A,B) \we \Fun(FA,FB)$ descends to an equivalence of hom-categories $\hom(A,B) \to \hom(FA,FB)$, proving \ref{itm:we-bi-equiv}. Any equivalence of categories induces a bijection between isomorphism classes of objects, proving \ref{itm:we-iso}. A category is a groupoid if and only if it is equivalent to a groupoid, so \ref{itm:we-groupoidal} follows similarly, via Definition \ref{defn:groupoidal-object}\ref{itm:groupoidal-object:ii}. 

The preservation halves of \ref{itm:we-equiv}-\ref{itm:we-comma} holds for any functor of $\infty$-cosmoi, as observed in Proposition \ref{prop:induced-2-functor}. The reflection halves of \ref{itm:we-equiv} and \ref{itm:we-equivs} hold for any biequivalence, by a standard argument. The proof of the remaining half of \ref{itm:we-comma} is similar to the proof of \ref{itm:we-equiv}, using the fact that the local equivalence of mapping quasi-categories pulls back to define a local equivalence of fibered mapping quasi-categories
\[ \xymatrix@=1em{ \Fun_B(E,E') \ar@{-->}[dr]^-{\rotatebox{155}{$\labelstyle\sim$}} \ar[rr] \pbexcursion \ar@{->>}[dd] & & \Fun(E,E') \ar[dr]^-{\rotatebox{155}{$\labelstyle\sim$}} \ar@{->>}'[d][dd] \\ & \Fun_{FB}(FE,FE') \pbexcursion \ar@{->>}[dd] \ar[rr] & & \Fun(FE,FE') \ar@{->>}[dd] \\ \Del^0 \ar'[r][rr] \ar@{=}[dr]  & & \Fun(E,B) \ar[dr]^-{\rotatebox{155}{$\labelstyle\sim$}} \\ & \Del^0 \ar[rr] & & \Fun(FE,FB)}\] 
\end{proof}

The assertion made in \ref{prop:we}\ref{itm:we-comma}  can be strengthened, using  Lemma \ref{lem:equiv-to-comma}:

\begin{exs}\label{exs:we-comma-redux} Show that: 
\begin{enumerate}[label=(\roman*)]
\item\label{itm:we-comma-redux:i} If $f \cong f' \colon B \to A$ and $g \cong g' \colon C \to A$, then $f \comma g \simeq f' \comma g'$ over $C \times B$.
\end{enumerate}
Combine this with Lemma \refV{lem:equiv-invariance-of-commas}, which says that  a commutative diagram
\[ \xymatrix{ C' \ar[r]^{g'} \ar[d]_{c}^{\rotatebox{90}{$\labelstyle\sim$}} & A' \ar[d]_a^{\rotatebox{90}{$\labelstyle\sim$}} & B' \ar[l]_{f'} \ar[d]^b_{\rotatebox{90}{$\labelstyle\sim$}} \\ C \ar[r]_g & A & B \ar[l]^f}\] induces an equivalence $f \comma g \we f' \comma g'$ over $c \times b \colon C' \times B' \we C \times B$, to conclude:
\begin{enumerate}[label=(\roman*),resume]
\item If  $E \tfib C \times B$ is an isofibration in $\lcat{K}$ whose image under a weak equivalence of $\infty$-cosmoi $F \colon \lcat{K} \to \lcat{L}$ defines a comma $\infty$-category for some pair of functors in $\lcat{L}$, then $E \tfib C \times B$ defines a comma $\infty$-category in $\lcat{K}$.
\end{enumerate}
\end{exs}

\begin{thm}[model independence of basic category theory I]\label{thm:model-independence} The following notions are preserved and reflected by any weak equivalence of $\infty$-cosmoi:
\begin{enumerate}[label=(\roman*)]
\item\label{itm:model-independence-i} The adjointness of a pair of $\infty$-functors $f \colon B \to A$ and $u \colon A \to B$.
\item\label{itm:model-independence-ii} The existence of a left or right adjoint to an $\infty$-functor $u \colon A \to B$.
\item\label{itm:model-independence-iii} The question of whether a given element $\ell \colon 1 \to A$ defines a limit or a colimit for a diagram $d \colon 1 \to A^J$.
\item\label{itm:model-independence-iv} The existence of a limit or a colimit for a $J$-indexed diagram $d \colon 1 \to A^J$ in an $\infty$-category $A$.
\end{enumerate}
\end{thm}
\begin{proof} \ref{itm:model-independence-i} and \ref{itm:model-independence-iii} follow directly from Proposition \ref{prop:we}\ref{itm:we-comma}, via the characterizations of Propositions \ref{prop:adjoint-comma-equiv-I}, \ref{prop:adjoint-comma-equiv-II}, and \ref{prop:limit-comma-equiv}. Then  \ref{itm:model-independence-ii} and \ref{itm:model-independence-iv} follow, using \ref{prop:we}\ref{itm:we-iso} to lift an adjoint or limit element from $\lcat{L}$ to $\lcat{K}$ and Exercise \ref{exs:we-comma-redux}\ref{itm:we-comma-redux:i} to transport the universal property encoded by an equivalence of commas along isomorphic functors. 
\end{proof}

\begin{rmk} By definition, if $F \colon \lcat{K} \to \lcat{L}$ is a functor of $\infty$-cosmoi, $A \in \lcat{K}$, and $J \in \sSet$, then $F(A^J) \cong FA^J \in \lcat{L}$. If $\lcat{K}$ and $\lcat{L}$ are cartesian closed $\infty$-cosmoi and $J \in \lcat{K}$, then there is a natural functor $F(A^J) \to FA^{FJ}$, the transpose of the image of the evaluation map $A^J \times J \to A$ under $F$. If $F$ is a weak equivalence of $\infty$-cosmoi, then this map induces a natural equivalence $\Fun(X, F(A^J)) \we \Fun(X, FA^{FJ})$ for all $X\in\lcat{L}$, and so $F(A^J) \simeq FA^{FJ}$. With more care, analogs of the assertions of Theorem \ref{thm:model-independence}\ref{itm:model-independence-iii} and \ref{itm:model-independence-iv} concerning limit and colimits of simplicial set-indexed diagrams can be proven for diagrams indexed by another $\infty$-category.
\end{rmk}

\renewcommand\thesection{Lecture~\arabic{section}}
\section{Fibrations, modules, and Kan extensions}\label{sec:modules}
\renewcommand\thesection{\arabic{section}}

A section of Saunders Mac Lane's \emph{Categories for the Working Mathematician}  is famously entitled ``All concepts are Kan extensions.'' Our aim in this final lecture is to develop the theory of Kan extensions for functors between $\infty$-categories.

At first glance, this might seem easy. After all, any 2-category has an internally-defined notion of extension diagram.

\begin{defn}\label{defn:right-ext} Given a span $C \xleftarrow{f} A \xrightarrow{k} B$, a functor $r \colon B \to C$ and a 2-cell 
\[ \xymatrix{ A \ar[r]^k \ar[d]_f & B \ar[dl]^r \\ C & \ar@{}[ul]|(.7){\Leftarrow\nu}}\]
define a \emph{right extension of $f$ along $k$} if any 2-cell as displayed below-left factors uniquely through $\nu$ as displayed below-right:
\[    \vcenter{\xymatrix{ A \ar[d]_f \ar[r]^k \ar@{}[dr]|(.3){\Leftarrow\chi} & B \ar[dl]^g \\ C  & }} \mkern20mu = \mkern20mu \vcenter{ \xymatrix{ A \ar[r]^k \ar[d]_f & B \ar[dl]|r \ar@/^4ex/[dl]^g \\ C & \ar@{}[ul]|(.3){\Leftarrow\exists!}|(.7){\Leftarrow\nu}}}
\] Dually, a functor $\ell \colon B \to C$ and a 2-cell $\lambda \colon f \To \ell k$ define a \emph{left extension of $f$ along $k$} if any 2-cell as displayed below-left factors uniquely through $\lambda$ as displayed below-right:
\[    \vcenter{\xymatrix{ A \ar[d]_f \ar[r]^k \ar@{}[dr]|(.3){\Rightarrow\chi} & B \ar[dl]^g \\ C  & }} \mkern20mu = \mkern20mu \vcenter{ \xymatrix{ A \ar[r]^k \ar[d]_f & B \ar[dl]|\ell \ar@/^4ex/[dl]^g \\ C & \ar@{}[ul]|(.3){\Rightarrow\exists!}|(.7){\Rightarrow\lambda}}}
\] 
\end{defn}

In the 2-category $\h\Cat$, \ref{defn:right-ext} defines the usual right and left Kan extensions for functors between ordinary categories. Special cases of these, in turn, define adjunctions (by a dual to Exercise \ref{exs:adj-as-abs-lifting}) and limits and colimit (as right or left extensions of the diagram $f \colon A \to C$ along the functor $!\colon A \to 1$). However, in general it turns out that the universal property expressed by this naive notion of right extension is insufficiently strong. In particular, Definition \ref{defn:right-ext}, interpreted in the homotopy 2-category, does not define the correct notion of Kan extension for $\infty$-functors. Instead, the correction notion will be of \emph{pointwise right Kan extensions}.

In fact, we will give two equivalent definitions of pointwise Kan extensions, both of which make use of comma $\infty$-categories, our vehicle for encoding $\infty$-categorical universal properties. One of these could be stated immediately, but we instead delay it in order to first develop the prerequisite theory for the other. Specifically, our aim will be to describe the full universal property of the comma $\infty$-category construction: namely that an isofibration $(p_1,p_0) \colon f \comma g \tfib C \times B$ constructed from functors $C \xrightarrow{g} A \xleftarrow{f} B$ encodes a \emph{module} from $C$ to $B$, with $C$ acting covariantly ``on the left'' and $B$ acting contravariantly ``on the right.'' The calculus of modules describes the 2-dimensional structure into which modules most naturally assemble, which turns out to be a familiar setting for formal category theory, the scope of which includes pointwise Kan extensions.

\subsection{Arrow \texorpdfstring{$\infty$}{infinity}-categories define modules}

The prototypical examples of modules are comma $\infty$-categories. For simplicity, we specialize to the case of the arrow $\infty$-category $(p_1,p_0)\colon A^\cattwo \tfib A \times A$. Lemma \ref{lem:ess.unique.1-cell.ind} describes its universal property: 2-cells $\xymatrix{X \ar@/^2ex/[r]^a \ar@/_2ex/[r]_b \ar@{}[r]|{\Downarrow g} & A}$ with codomain $A$ correspond to generalized elements $\hat{g} \colon X \to A^\cattwo$, up to isomorphism over $A \times A$. By Lemma \ref{lem:equiv-to-comma}, this universal property characterizes the arrow $\infty$-category up to equivalence of isofibrations over $A \times A$. However, it does not capture the additional fact that 2-cells from $X$ to $A$ can be composed vertically
\[\vcenter{\xymatrix@R=30pt{  X \ar@/^5.5ex/[r]^x_{\Downarrow f} \ar@/_5.5ex/[r]_y^{\Downarrow h} \ar@/^2ex/[r]|a \ar@/_2ex/[r]|b \ar@{}[r]|{\Downarrow g}& A}} \qquad \qquad (h \cdot g) \cdot f = h \cdot g \cdot f = h \cdot (g \cdot f),\] 
 defining commuting contravariant and covariant actions on the domains and codomains of the 2-cell $g \colon a \To b$. 
Observe that the domain and codomain 1-cells for the 2-cell $g \colon a \To b$ can be recovered as the composites of a representing functor $\hat{g} \colon X \to A^\cattwo$ with the projection functors: $p_0 \hat{g} = a$ and $p_1 \hat{g} = b$.
One way to express these actions is to note that the domain-projection functor $p_0 \colon A^\cattwo \tfib A$ and the codomain-projection functor $p_1 \colon A^\cattwo \tfib A$ respectively define a \emph{cartesian fibration} and a \emph{cocartesian fibration}:  any 2-cells as displayed on the left-hand side of the pasting equalities below admit lifts as displayed on the right-hand sides:
\[ \vcenter{\xymatrix{ X \ar[r]^{\hat{g}} \ar[dr]_x & A^\cattwo \ar@{->>}[d]^{p_0} \ar@{}[dl]|(.3){\Uparrow f} \ar@{}[dr]|{\displaystyle ~=} & X \ar@/^2ex/[r]^{\hat{g}} \ar@/_2ex/[r]_{f^*(x)} \ar@{}[r]|{\Uparrow \chi_{f}} & A^\cattwo \ar@{->>}[d]^{p_0}  \\ & A & & A}}  \qquad \vcenter{\xymatrix{ X \ar[r]^{\hat{g}} \ar[dr]_y & A^\cattwo \ar@{->>}[d]^{p_1} \ar@{}[dl]|(.3){\Downarrow h} \ar@{}[dr]|{\displaystyle ~= } & X \ar@/^2ex/[r]^{\hat{g}} \ar@/_2ex/[r]_{h_*(y)} \ar@{}[r]|{\Downarrow \chi^{h}} & A^\cattwo \ar@{->>}[d]^{p_1}  \\ & A & & A}}
\]
Moreover, the lifted 2-cell $\chi_f$ can be chosen to project along $p_1 \colon A^\cattwo \tfib A$ to $\id_b$ and the lifted 2-cell $\chi^h$ can be chosen to project along $p_0 \colon A^\cattwo \tfib A$ to $\id_a$. 

In summary, the arrow $\infty$-category $A^\cattwo$ defines a \emph{module} from $A$ to $A$. The definition will make use of the fact that slices of $\infty$-cosmoi are again $\infty$-cosmoi.

\begin{defn}[{\refV{defn:sliced-cosmoi}}]\label{defn:sliced-cosmoi} 
If $\lcat{K}$ is any $\infty$-cosmos and $B \in \lcat{K}$ is any object, then there is an $\infty$-cosmos $\lcat{K}_{/B}$, the \emph{sliced $\infty$-cosmos of $\lcat{K}$ over $B$}, whose:
\begin{itemize}
\item objects are isofibrations $p \colon E \tfib B$ with codomain $B$;
\item mapping quasi-category from $p \colon E \tfib B$ to $q \colon F \tfib B$ is defined by  taking the pullback
\[
    \xymatrix@=1.5em{
      {\Fun_B(p,q)}\pbexcursion\ar[r]\ar@{->>}[d] &
      {\Fun(E,F)}\ar@{->>}[d]^{\Fun(E,q)} \\
      {\Del^0}\ar[r]_-{p} & {\Fun(E,B)}
    }
\]
  in simplicial sets;
\item isofibrations, equivalences, and trivial fibrations are created by the forgetful functor $\lcat{K}_{/B} \to \lcat{K}$;
\end{itemize}
and in which the simplicial limits are defined in the usual way for sliced simplicial categories.
\end{defn}

\begin{defn}\label{defn:module} In an $\infty$-cosmos $\lcat{K}$, a \emph{module $E$ from $A$ to $B$}, denoted by $\dmod{E}{A}{B}$, is given by an isofibration $(q,p) \colon E \tfib A \times B$ such that
  \begin{enumerate}[label=(\roman*)]
\item\label{itm:cartesian-on-the-right} $\vcenter{\xymatrix@=1em{ E \ar@{->>}[dr]_q \ar[rr]^-{(q,p)} & & A \times B \ar@{->>}[dl]^{\pi_1} \\ & A}}$ is a \emph{cartesian fibration} in  $\lcat{K}_{/A}$; informally, ``$B$ acts on the right of $E$, over $A$.'' 
\item\label{itm:cocartesian-on-the-left}  $\vcenter{\xymatrix@=1em{ E \ar@{->>}[dr]_p \ar[rr]^-{(q,p)} & & A \times B \ar@{->>}[dl]^{\pi_0} \\ & B}}$  is a \emph{cocartesian fibration} in $\lcat{K}_{/B}$; informally,  ``$A$ acts on the left of $E$, over $B$.'' 
\item\label{itm:groupoidal-fibers} $(q,p) \colon E \tfib A \times B$ is groupoidal as an object in $\lcat{K}_{/A \times B}$; this  means that any 2-cell $\xymatrix{X \ar@/^2ex/[r]^e \ar@/_2ex/[r]_{e'} \ar@{}[r]|{\Downarrow} & E}$ over an identity in $A \times B$ is an isomorphism, which implies in particular that  $(q,p) \colon E \tfib A \times B$ has groupoidal fibers.
\end{enumerate}
\end{defn}

\begin{ex}[{\refV{prop:hom-is-a-module}}]\label{ex:hom-is-a-module} For any $\infty$-category $A$, the arrow $\infty$-category defines a module $\dmod{A^\cattwo}{A}{A}$. The fact that $A^\cattwo \tfib A \times A$ is groupoidal is related to but stronger than the fact that each fiber over a pair of  elements in $A$, the hom-spaces of Definition \ref{defn:hom-space}, is a groupoidal $\infty$-category.
\end{ex}

\subsection{Cartesian and cocartesian fibrations}

To explain Definition \ref{defn:module}, we need to define what it means for a functor in an $\infty$-cosmos to be a \emph{cartesian fibration} or \emph{cocartesian fibration}. We will not actually require any of these details for out ultimate aim in this lecture, to initiate the theory of pointwise Kan extensions, but we include them because these fibration notions are of independent interest.

\begin{defn}[cartesian 2-cells]
A 2-cell $\chi \colon e' \Rightarrow e \colon A \to E$ in the homotopy 2-category of an $\infty$-cosmos is \emph{cartesian} for an isofibration $p \colon E \tfib B$ if and only if
\begin{enumerate}[label=(\roman*)]
  \item\label{itm:weak.cart.i} (induction) for any pair of 2-cells $\tau \colon e'' \Rightarrow e$ and $\gamma\colon pe''\Rightarrow pe'$ with $p \tau = p\chi \cdot \gamma$ there is some $\overline{\gamma} \colon e'' \Rightarrow e'$ with $p\overline\gamma = \gamma$ ($\bar\gamma$ lies over $\gamma$) and the property that $\tau = \chi \cdot \bar\gamma$.
  \item\label{itm:weak.cart.ii} (conservativity) for any 2-cell $\gamma \colon e' \Rightarrow e'$ if $\chi \cdot \gamma = \chi$ and $p\gamma$ is an identity then $\gamma$ is an isomorphism.
\end{enumerate}
\end{defn}

All isomorphisms with codomain $E$ are $p$-cartesian. The class of $p$-cartesian 2-cells is stable under composition and left cancelation (Lemmas~\refIV{lem:cart-arrows-compose} and~\refIV{lem:cart-arrows-cancel}).

\begin{defn}[cartesian fibration]\label{defn:cart-fib}
An isofibration  $p\colon E\tfib B$ is a {\em cartesian fibration\/} if and only if:
\begin{enumerate}[label=(\roman*)]
  \item Every 2-cell $\alpha \colon b \To pe$ has a $p$-cartesian lift $\chi_\alpha \colon \alpha^*(e)\To e$:
\[\vcenter{\xymatrix{ X \ar[r]^{e} \ar[dr]_b & E \ar@{->>}[d]^{p} \ar@{}[dl]|(.3){\Uparrow \alpha} \ar@{}[dr]|{\displaystyle ~=} & X \ar@/^2ex/[r]^{e} \ar@/_2ex/[r]_{\alpha^*(e)} \ar@{}[r]|{\Uparrow \chi_\alpha } &E \ar@{->>}[d]^{p}  \\ & B & & B}} \]
  \item The class of $p$-cartesian 2-cells for $p$ is closed under pre-composition by all 1-cells.
  \end{enumerate}
\end{defn}

Importantly, there is a ``model independent'' characterization of cartesian fibrations given in terms of adjunctions between commas. Any functor $p \colon E \to B$ induces functors between comma $\infty$-categories 
  \begin{equation*}
 \vcenter{\xymatrix@C=0.8em@R=1.2em{
    & {E}\ar[d]^-{i} & \\
    & {B\comma p}\ar@{->>}[dl]_{p_1}\ar@{->>}[dr]^{p_0} & \\
    {E}\ar[rr]_{p} && {B}
    \ar@{} "2,2";"3,2" |(0.6){\Leftarrow\phi}
  }} = 
  \vcenter{\xymatrix@C=0.8em@R=1.2em{
    & {E}\ar@{=}[dl]\ar@{->>}[dr]^{p} & \\
    {E}\ar[rr]_{p} && {B}
    \ar@{} "1,2";"2,2" |(0.6){=}
  }}   
  \mkern50mu
  \vcenter{\xymatrix@C=0.8em@R=1.2em{
    & {E^\cattwo}\ar[d]^-{k} & \\
    & {B\comma p}\ar@{->>}[dl]_{p_1}\ar@{->>}[dr]^{p_0} & \\
    {E}\ar[rr]_{p} && {B}
    \ar@{} "2,2";"3,2" |(0.6){\Leftarrow\phi}
  }} = 
  \vcenter{\xymatrix@C=0.8em@R=1.2em{
    & {E^\cattwo}\ar@{->>}_{q_1}[dl]\ar@{->>}[dr]^{pq_0} & \\
    {E}\ar[rr]_{p} && {B}
    \ar@{} "1,2";"2,2" |(0.6){\Leftarrow p\psi}
  }}
  \end{equation*}
 that are well-defined up to isomorphism over $E \times B$. 

\begin{thm}[{\refIV{thm:cart.fib.chars}}]\label{thm:cart.fib.chars} For an isofibration $p \colon E \tfib B$, the following are equivalent:
  \begin{enumerate}[label=(\roman*)]
    \item\label{itm:cart.fib.chars.i} $p$ is a cartesian fibration.
    \item\label{itm:cart.fib.chars.ii} The functor $i\colon E\to B\comma p$ admits a right adjoint which is fibered over $B$.
\[
      \xymatrix@R=2em@C=3em{
        {B\comma p}\ar@{->>}[dr]_{p_0} \ar@{-->}@/_0.6pc/[]!R(0.5);[rr]_{r}^{}="a" & &
        {E}\ar@{->>}[dl]^{p} \ar@/_0.6pc/[ll]!R(0.5)_{i}^{}="b" 
        \ar@{}"a";"b"|{\bot} \\
        & B &
      }
    \]
    \item\label{itm:cart.fib.chars.iii} The functor $k\colon E^\cattwo\to B\comma p$ admits a right adjoint right inverse, i.e., a right adjoint with invertible counit.
\[
    \xymatrix@C=6em{
      {B\comma p}\ar@{-->}@/_0.8pc/[]!R(0.6);[r]!L(0.45)_{\bar{r}}^{}="u" &
      {E^\cattwo}\ar@/_0.8pc/[]!L(0.45);[l]!R(0.6)_{k}^{}="t"
      \ar@{}"u";"t"|(0.6){\bot}
    }
  \]
  \end{enumerate}
\end{thm}

\emph{Cocartesian fibrations} are defined dually, by reversing the direction of the 2-cells in Definition \ref{defn:cart-fib} and of the adjoints in Theorem \ref{thm:cart.fib.chars}. A cartesian or cocartesian fibration $p \colon E \tfib B$ that defines a groupoidal object in the sliced $\infty$-cosmos over $B$ is called a \emph{groupoidal cartesian fibration} or \emph{groupoidal cocartesian fibration}.\footnote{In the $\infty$-cosmos of quasi-categories, these coincide with the \emph{right fibrations} and \emph{left fibrations} introduced by Joyal \cite{Joyal:2002:QuasiCategories}; see \refIV{ex:quasi-groupoidal-cart}.} A groupoidal cartesian fibration is a cartesian fibration with groupoidal fibers. Condition \ref{itm:gpd.cart.fib.chars.iii} of the following theorem provides a ``model independent'' characterization of groupoidal cartesian fibrations.

\begin{thm}[{\refV{thm:gpd.cart.fib.chars}}]\label{thm:gpd.cart.fib.chars} For an isofibration $p \colon E \tfib B$, the following are equivalent:
  \begin{enumerate}[label=(\roman*)]
    \item\label{itm:gpd.cart.fib.chars.i} $p$ is a groupoidal cartesian fibration.
    \item\label{itm:gpd.cart.fib.chars.ii} Every 2-cell $\alpha\colon b\Rightarrow pe\colon X\to B$ has an essentially unique lift $\chi\colon e'\Rightarrow e\colon X\to E$, where the essential uniqueness is up to composition into the domain of $\chi$ with an invertible 2-cell that projects along $p$ to an identity.
     \item\label{itm:gpd.cart.fib.chars.iii} The functor $k\colon E^\cattwo\to B\comma p$ is an equivalence.
\end{enumerate}
\end{thm}

Recall that a functor $F \colon \lcat{K} \to \lcat{L}$ of $\infty$-cosmoi induces a 2-functor between the homotopy 2-categories that preserves equivalences, adjunctions, and commas, and if $F$ is a weak equivalence of $\infty$-cosmoi, then these notions are also reflected; see Proposition \ref{prop:induced-2-functor}, Proposition \ref{prop:we}, and Theorem \ref{thm:model-independence}. It is straightforward to extend  Theorem \ref{thm:model-independence}\ref{itm:model-independence-ii} to see that the existence of a right adjoint right inverse is preserved and reflected as well. By Theorem \ref{thm:cart.fib.chars}\ref{itm:cart.fib.chars.iii} and Theorem \ref{thm:gpd.cart.fib.chars}\ref{itm:gpd.cart.fib.chars.iii} we conclude:

\begin{cor}\label{cor:model-indep-fib} $\quad$ 
\begin{enumerate}[label=(\roman*)]
\item Functors of $\infty$-cosmoi preserve cartesian fibrations, cocartesian fibrations, group\-oid\-al cartesian fibrations, groupoidal cocartesian fibrations, and modules.
\item Weak equivalences of $\infty$-cosmoi both preserve and reflect cartesian fibrations, cocartesian fibrations, groupoidal cartesian fibrations, groupoidal cocartesian fibrations, and modules: an isofibration in the domain is a functor of this type if and only if its image is.
\end{enumerate}
\end{cor}
\begin{proof}
The preservation and reflection of a module from $A$ to $B$ follows from the statements concerning fibrations together with Propositions \ref{prop:induced-2-functor} and \ref{prop:we}\ref{itm:we-groupoidal} applied to the sliced functor of $\infty$-cosmoi $\lcat{K}_{/A\times B} \to \lcat{L}_{/FA \times FB}$.
\end{proof}

\subsection{The calculus of modules}

The calculus of modules between $\infty$-categories bears a strong resemblance to the calculus of (bi)modules between unital rings, with functors between $\infty$-categories playing the role of ring homomorphisms.\footnote{In more detail, unital rings, ring homomorphisms, bimodules, and module maps define a \emph{proarrow equipment}, in the sense of Wood \cite{wood:proI}. This can be seen as a special case of the prototypical equipment comprised of $\lcat{V}$-categories, $\lcat{V}$-functors, $\lcat{V}$-modules, and $\lcat{V}$-natural transformations between then, for any closed symmetric monoidal category $\lcat{V}$. The equipment for rings is obtained from the case where $\lcat{V}$ is the category of abelian groups by restricting to abelian group enriched categories with a single object.}

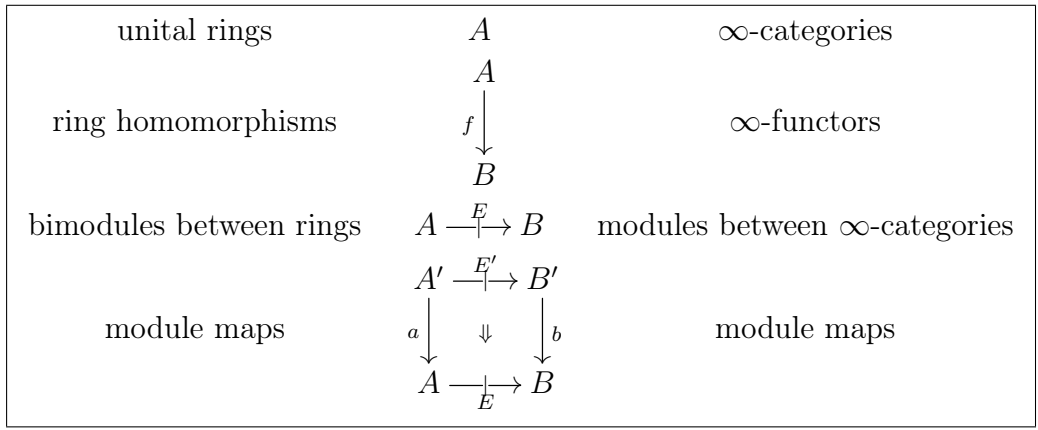
\begin{center}
\fbox{\begin{tabular}{ccc}
unital rings & $A$ & $\infty$-categories \\
ring homomorphisms & $\vcenter{\xymatrix{ A \ar[d]_f \\ B}}$ & $\infty$-functors \\
bimodules between rings & $\xymatrix{ A \ar[r]|\mid^E & B}$ & modules between $\infty$-categories \\
module maps & $\raisebox{.75cm}{ \xymatrix{A' \ar[d]_a \ar@{}[dr]|{\Downarrow} \ar[r]|\mid^{E'} & B' \ar[d]^b \\ A \ar[r]|\mid_E & B}}$ & module maps 
\end{tabular}}
\end{center}

Our first result, analogous to restriction of scalars, is that modules can be pulled back. 

\begin{prop}[{\refV{prop:two-sided-pullback}}]\label{prop:module-pullback} If $(q,p) \colon E \tfib A \times B$ defines a module from $A$ to $B$ and $a \colon A' \to A$ and $b \colon B' \to B$ are any functors then the pullback
\[ \xymatrix{ E(b,a) \ar@{->>}[d] \ar[r] \pbexcursion & E \ar@{->>}[d]^{(q,p)} \\ A' \times B' \ar[r]_{a \times b} & A \times B}\] defines a module $E(b,a)$ from $A'$ to $B'$.
\end{prop}

\begin{ex} Recall that comma $\infty$-categories are defined to be pullbacks of arrow $\infty$-categories.    \[ \xymatrix@=2.5em{
      {f\comma g}\pbexcursion \ar[r]\ar@{->>}[d]_{(p_1,p_0)} &
      {A^\cattwo} \ar@{->>}[d]^{(p_1,p_0)} \\
      {C\times B} \ar[r]_-{g\times f} & {A\times A}
    }\]
Since we know from Example \ref{ex:hom-is-a-module} that $A^\cattwo$ defines a module $\dmod{A^\cattwo}{A}{A}$, Proposition \ref{prop:module-pullback} implies that $f\comma g$ defines a module  $\dmod{f \comma g}{C}{B}$.
\end{ex}

As is the case for rings, horizontal composition of modules between $\infty$-categories is a complicated operation. For  general modules $\dmod{E}{A}{B}$ and $\dmod{F}{B}{C}$,  the pullback
\[ \xymatrix@!=5pt{ & & E \times_B F \pbdiamond \ar@{->>}[dl]_-{\pi_1} \ar@{->>}[dr]^-{\pi_0} \\  & E \ar@{->>}[dl]_q \ar@{->>}[dr]^p &  & F \ar@{->>}[dl]_s \ar@{->>}[dr]^r \\ A  & & B & & C}\] defines an isofibration $E \times_B F \tfib A \times C$ that is a groupoidal cartesian fibration in the slice over $A$ and a groupoidal cocartesian fibration in the slice over $C$. However, it fails to be a groupoidal object over $A \times C$, as can be seen by considering the case $E=F=A^\cattwo$. We use the notation $E \times_B F$ to denote this pullback construction, reserving $E \otimes_B F$ for special cases in which there is a module from $A$ to $C$ that can be recognized as the horizontal composite. For particular types of $\infty$-categories, general composite modules  $E \otimes_B F$  can be defined via a ``fiberwise coinverter'' construction, but to do so requires that we leave the $\infty$-cosmos axiomatization, which does not provide for any colimits.

Rather than leave the axiomatization in search of a general composition formula, it turns out to be simpler to do without it. There is a natural categorical framework into which modules assemble, even without general horizontal composites, that turns out to be sufficient for our real goal: developing a theory of pointwise Kan extensions. Rings, ring homomorphisms, modules, and module maps assemble into a 2-dimensional structure known as a \emph{double category}. By analogy, $\infty$-categories, $\infty$-functors, modules, and module maps, to be introduced assemble into a \emph{virtual double category}.

\begin{defn}[{\refV{defn:modules-virtual}}] The \emph{virtual double category of modules} $\MMod{K}$ in an $\infty$-cosmos $\lcat{K}$ consists of
\begin{itemize}
\item a category of \emph{objects} and \emph{vertical arrows}, here the $\infty$-categories and $\infty$-functors
\item for any pair of objects $A, B$, a class of \emph{horizontal arrows} $\dmod{~}{A}{B}$, here the modules from $A$ to $B$
\item \emph{cells}, with boundary depicted as follows
\begin{equation}\label{eq:generic-cell} \xymatrix{ A_0 \ar[d]_{f} \ar[r]|{\mid}^{E_1} & A_1 \ar[r]|{\mid}^{E_2}  \ar@{}[dr]|{\Downarrow} & \cdots \ar[r]|{\mid}^{E_n} & A_n \ar[d]^g \\ B_0 \ar[rrr]|{\mid}_{F} & & & B_1}\end{equation} including those whose horizontal source has length zero in the case $A_0 = A_n$, which we depict as
\begin{equation}\label{eq:generic-nullary-cell} \xymatrix{A \ar[d]_f \ar@{=}[r] \ar@{}[dr]|{\Downarrow} & A \ar[d]^g \\ B_0 \ar[r]|{\mid}_F & B_1}\end{equation}
 Here, a cell with boundary \eqref{eq:generic-cell} will be an isomorphism class of objects in the mapping quasi-category
\[ \xymatrix{ \Fun_{f,g}(E_1 \pbtimes{A_1} \cdots \pbtimes{A_{n-1}} E_n, F) \ar[r] \ar[d] \pbexcursion & \Fun(E_1 \pbtimes{A_1} \cdots \pbtimes{A_{n-1}} E_n, F) \ar[d] \\ \Delta^0 \ar[r]^-{(f,g)} & \Fun(E_1 \pbtimes{A_1} \cdots \pbtimes{A_{n-1}} E_n, B_0 \times B_1)}\] Similarly a cell with boundary \eqref{eq:generic-nullary-cell} is an isomorphism class of objects in $\Fun_{f,g}(A,F)$.
\item a \emph{composite cell}, for any configuration
\[ \xymatrix@C=35pt{ A_0 \ar[d]_{f_0} \ar@{..>}[r]|{\mid}^{E_{11},\dots, E_{1n_1}} \ar@{}[dr]|{\Downarrow}  & A_1 \ar@{..>}[r]|{\mid}^{E_{21},\ldots, E_{2n_2}} \ar[d]^{f_1}  \ar@{}[dr]|{\Downarrow} & \cdots \ar@{..>}[r]|{\mid}^{E_{n1},\ldots, E_{nn_n}}  \ar@{}[d]|\cdots \ar@{}[dr]|{\Downarrow} & A_n \ar[d]^{f_n} \\ B_0 \ar[d]_{g} \ar[r]|{\mid}^{F_1} & B_1 \ar[r]|{\mid}^{F_2}  \ar@{}[dr]|{\Downarrow} & \cdots \ar[r]|{\mid}^{F_n} & B_n \ar[d]^h \\ C_0 \ar[rrr]|{\mid}_{G} & & & C_n}\] 
\item an \emph{identity cell} for every horizontal arrow 
\[ \xymatrix{ A \ar@{=}[d] \ar[r]|{\mid}^E & B \ar@{=}[d] \ar@{}[dl]|{\Downarrow \id_E} \\ A \ar[r]|{\mid}_E & B}\] 
\end{itemize}
so that composition of cells is associative and unital in the usual multi-categorical sense.
\end{defn}

The following examples motivate our definition of \emph{module maps}, i.e., cells in the virtual double category $\MMod{K}$.

\begin{ex}\label{ex:cells-above-a-comma} Lemma  \ref{lem:ess.unique.1-cell.ind}, which expresses 1-cell induction as a bijection between isomorphism classes of maps of spans whose codomain is a comma span and certain 2-cells in the homotopy 2-category,  provides an alternate characterization of cells in the virtual double category of modules whose codomain is a comma module. Explicitly, for any cospan  $C \xrightarrow{g} A \xleftarrow{f} B$, there is a bijection between cells in $\MMod{K}$ whose codomain is the comma module $\dmod{f \comma g}{C}{B}$ and 2-cells in the homotopy 2-category $\h\lcat{K}$ under the pullback of the spans encoding the domain modules\footnote{In the case of a cell with nullary source $A$, this empty pullback corresponds to the identity span $A = A = A$.} and  over the cospan defining the comma module $f \comma g$. 
\[ \vcenter{ \xymatrix{ A_0 \ar[d]_{c} \ar[r]|{\mid}^{E_1} & A_1 \ar[r]|{\mid}^{E_2}  \ar@{}[dr]|{\Downarrow} & \cdots \ar[r]|{\mid}^{E_n} & A_n \ar[d]^b \\ C \ar[rrr]|{\mid}_{f \comma g} & & & B}}\qquad \leftrightsquigarrow\qquad \vcenter{ \xymatrix@!0@C=40pt@R=30pt{ & E_1 \pbtimes{A_1} \cdots \pbtimes{A_{n-1}} E_n \ar[dl] \ar[dr] \ar@{}[ddd]|{\displaystyle\Leftarrow} \\ A_0  \ar[d]_c & & A_n \ar[d]^b \\ C \ar[dr]_g & & B \ar[dl]^f \\ & A}}\] 
\end{ex}

\begin{exs}\label{exs:sample-cells} Use the correspondence described in Example \ref{ex:cells-above-a-comma} to define canonical nullary and unary cells in $\MMod{K}$ associated to any functor $f \colon A \to B$.
\[ \xymatrix {  A \ar@{=}[r] \ar@{=}[d] \ar@{}[dr]|{\Downarrow\nu} & A \ar[d]^f & & A \ar[d]_f \ar[r]|\mid^{B \comma f} \ar@{}[dr]|{\Downarrow\rho} & B \ar@{=}[d] & & A \ar@{=}[r] \ar[d]_f \ar@{}[dr]|{\Downarrow\nu} & A \ar@{=}[d] & & B \ar@{=}[d] \ar[r]|\mid^{f \comma B} \ar@{}[dr]|{\Downarrow\rho} & A \ar[d]^f   \\ A \ar[r]_{B \comma f}|{\mid}  & B & & B \ar[r]|\mid_{B^\cattwo} & B & & B \ar[r]|\mid_{f \comma B} & A & & B \ar[r]|\mid_{B^\cattwo} & B}\]
\end{exs}

\begin{ex}[{\refV{prop:two-sided-yoneda}}]\label{ex:yoneda}
Any functor $f \colon A \to B$ induces a map
\[  \vcenter{\xymatrix@C=0.8em@R=1.2em{
    & {A}\ar[d]^-{t} & \\
    & {B \comma f}\ar@{->>}[dl]_{p_1}\ar@{->>}[dr]^{p_0} & \\
    {A}\ar[rr]_{f} && {B}
    \ar@{} "2,2";"3,2" |(0.6){\Leftarrow\phi}
  }} = 
  \vcenter{\xymatrix@C=0.8em@R=1.2em{
    & {A}\ar@{=}[dl]\ar[dr]^{f} & \\
    {A}\ar[rr]_{f} && {B}
    \ar@{} "1,2";"2,2" |(0.6){=}
  }} \]
over $A \times B$. Then for any module $E$ from $A$ to $B$, pre-composition with $t \colon A \to B \comma f$  induces an equivalence of quasi-categories 
\[ \Fun_{A \times B} ( B \comma f , E) \stackrel{\simeq}{\longrightarrow} \Fun_{A \times B} ( A , E ).\]  This result is a direct application of the Yoneda lemma for groupoidal cartesian fibrations (\refIV{cor:groupoidal-yoneda}) in the slice $\infty$-cosmos $\lcat{K}_{/A}$. Passing to isomorphism classes of objects in the mapping quasi-categories, this result asserts that there is a bijection between cells in $\MMod{K}$
\[ \vcenter{\xymatrix{A \ar@{=}[d] \ar[r]|{\mid}^{B \comma f} \ar@{}[dr]|{\Downarrow} & B \ar@{=}[d] \\ A \ar[r]|\mid_E & B}} \qquad \stackrel{\cong}{\mapsto}\qquad  \vcenter{\xymatrix{A \ar@{=}[d]\ar@{=}[r] \ar@{}[dr]|{\Downarrow} & A \ar[d]^f \\ A \ar[r]|\mid_E & B}} \] 
implemented by restricting along the nullary cell $\nu$, as defined in Exercise \ref{exs:sample-cells},  represented by the functor $t \colon A \to B \comma f$. 
\end{ex}

Proposition \ref{prop:module-pullback} tells us that modules in an $\infty$-cosmos can be pulled back. Given $\dmod{E}{A}{B}$ and functors $a \colon A' \to A$ and $b \colon B' \to B$, the horizontal functor $\rho \colon E(b,a) \to E$ in the diagram defining the pullback module
\begin{equation}\label{eq:simplicial-pullback-of-module} \xymatrix{ E(b,a) \pbexcursion \ar@{->>}[d]_{(q',p')} \ar[r]^-\rho & E \ar@{->>}[d]^{(q,p)} \\ A' \times B' \ar[r]_{a \times b} & A \times B}\end{equation} defines a unary cell in the virtual double category of modules with a universal property that we now describe.

\begin{prop}[{\refV{prop:cartesian-pullback-cells}}]\label{prop:cartesian-pullback-cells} In $\MMod{K}$, the cell 
\[ \xymatrix{ A' \ar[d]_a \ar[r]|{\mid}^{E(b,a)} \ar@{}[dr]|{\Downarrow\rho} & B' \ar[d]^b \\ A \ar[r]|{\mid}_E & B}\]
defined by pulling back a module $\dmod{E}{A}{B}$ along functors $a \colon A' \to A$ and $b \colon B' \to B$ has the property that any cell as displayed on the left
\[ \vcenter{ \xymatrix{  X_0 \ar[d]_{af} \ar[r]|{\mid}^{E_1} & X_1 \ar[r]|{\mid}^{E_2}  \ar@{}[dr]|{\Downarrow} & \cdots \ar[r]|{\mid}^{E_n} & X_n \ar[d]^{bg} \\ A \ar[rrr]_E|{\mid} & & & B}} \quad = \quad \vcenter{ \xymatrix{  X_0  \ar[d]_{f} \ar[r]|{\mid}^{E_1} & X_1 \ar[r]|{\mid}^{E_2}  \ar@{}[dr]|{\Downarrow \exists !} & \cdots \ar[r]|{\mid}^{E_n} & X_n \ar[d]^{g} \\ A' \ar[d]_a\ar[rrr]^{E(b,a)}|{\mid} & \ar@{}[dr]|{\Downarrow\rho} & & B' \ar[d]^b \\  A \ar[rrr]_E|{\mid} & & & B}}  \]
factors uniquely as displayed on the right.
\end{prop}

Proposition \ref{prop:cartesian-pullback-cells} asserts that $\rho$ is a \emph{cartesian cell} in $\MMod{K}$.

\begin{proof}
The simplicial pullback \eqref{eq:simplicial-pullback-of-module}, induces an equivalence of hom quasi-categories
\[ \Fun_{af,bg}(E_1 \times \cdots \times E_n, E) \simeq \Fun_{f,g}(E_1 \times \cdots \times E_n, E(b,a)). \qedhere \] 
\end{proof}

Each module $\dmod{A^\cattwo}{A}{A}$ defined by the arrow construction comes with a canonical cell with nullary source. Under the identification of Example \ref{ex:cells-above-a-comma}, this cell corresponds via 1-cell induction to the isomorphism class of maps of spans representing the identity 2-cell at the identity 1-cell of the object $A$.
\[ \vcenter{\xymatrix {  A \ar@{=}[d] \ar@{=}[r] \ar@{}[dr]|{\Downarrow\iota} & A \ar@{=}[d]  \\ A \ar[r]_{A^\cattwo}|{\mid}  & A}} \quad \leftrightsquigarrow \quad    \vcenter{
  \xymatrix@R=2.5em@C=0.8em{
    {A}\ar@/^3ex/[]!D(0.4);[d]!U(0.5)^{\id_A}
    \ar@/_3ex/!D(0.4);[d]!U(0.5)_{\id_A}
    \ar@{}[d]|{=} \\ {A}
  }} \quad = \quad   \vcenter{
    \xymatrix@R=2.5em@C=0.8em{
      {A}\ar[d]^-{j} \\
      {A^\cattwo}\ar@{->>}@/^3ex/[]!D(0.4);[d]!U(0.5)^{p_0}
      \ar@{->>}@/_3ex/!D(0.4);[d]!U(0.5)_{p_1}
      \ar@{}[d]|{\Leftarrow\phi} \\ {A}
    }} \] 
This cell also has a universal property in the virtual double category of modules.

\begin{prop}[{\refV{prop:arrows-are-units}}]\label{prop:arrows-are-units} Any cell in the virtual double category of modules whose horizontal source includes the object $A$, as displayed on the left
\[  \vcenter{\xymatrix@C=15pt{ X \ar[d]_{f} \ar[r]|{\mid}^{E_1} & \cdots \ar[r]|-{\mid}^-{E_n} &  A  \ar@{}[d]|{\Downarrow}    \ar[r]|-{\mid}^-{F_1} & \cdots \ar[r]|{\mid}^{F_m} & Y \ar[d]^g \\ B \ar[rrrr]|{\mid}_{G} & &  &  & C}} \quad=\quad  \vcenter{ \xymatrix@C=20pt{ X \ar@{=}[d] \ar[r]|{\mid}^{E_1} \ar@{}[dr]|{\Downarrow\id_{E_1}{~}} & \cdots \ar@{=}@<-1.5ex>[d]^{\cdots} \ar@{=}@<1.5ex>[d] \ar[r]|-{\mid}^-{E_n} \ar@{}[dr]|{{~}{~}\Downarrow\id_{E_n}}&  A \ar@{=}[r] \ar@{=}[d] \ar@{}[dr]|{\Downarrow\iota}  & \ar@{=}[d]A \ar@{}[dr]|{\Downarrow\id_{F_1}{~}} \ar[r]|-{\mid}^-{F_1} & \cdots \ar@{=}@<-1.5ex>[d]^{\cdots} \ar@{=}@<1.5ex>[d] \ar[r]|{\mid}^{F_m} \ar@{}[dr]|{{~}{~}\Downarrow\id_{F_m}}& Y \ar@{=}[d]  \\ X \ar[d]_{f} \ar[r]|{\mid}^{E_1} & \cdots \ar[r]|-{\mid}^-{E_n} &  A \ar[r]|{\mid}^{A^\cattwo} \ar@{}[dr]|{\Downarrow\exists !}  & A  \ar[r]|-{\mid}^-{F_1} & \cdots \ar[r]|{\mid}^{F_m} & Y \ar[d]^g \\ B \ar[rrrrr]|{\mid}_{G} & & & &  & C}}\]
factors uniquely through $\iota$ as displayed on the right.
\end{prop}

Proposition \ref{prop:arrows-are-units} asserts that $\iota$ is a \emph{cocartesian cell} in $\MMod{K}$.

\begin{proof}
In the case where both of the sequences $E_i$ and $F_j$ are empty, the Yoneda lemma for modules, in the form described in Example \ref{ex:yoneda}, supplies  an equivalence of quasi-categories
\[ \xymatrix{ \Fun_{f,g}(A^\cattwo, G) \simeq \Fun_{A \times A}(A^\cattwo, G(g,f)) \ar[r]^-{j^*}_{\simeq} & \Fun_{A \times A} (A, G(g,f)) \simeq \Fun_{f,g}(A,G) .}\] This equivalence descends to a bijection between isomorphism classes of objects, i.e., to a bijection between cells
\[ \vcenter{\xymatrix{A \ar[d]_f \ar[r]|{\mid}^{A^\cattwo} \ar@{}[dr]|{\Downarrow} & A \ar[d]^g \\ B \ar[r]|\mid_G & C}} \qquad \stackrel{\cong}{\mapsto}\qquad  \vcenter{\xymatrix{A \ar[d]_f \ar@{=}[r] \ar@{}[dr]|{\Downarrow} & A \ar[d]^g \\ B \ar[r]|\mid_G & C}} \] 
implemented by restricting along the cocartesian cell $\iota$. See \refV{prop:arrows-are-units} for the proof in the general case.
\end{proof}

Propositions \ref{prop:cartesian-pullback-cells} and \ref{prop:arrows-are-units} imply that the virtual double category of modules is a \emph{virtual equipment} in the sense introduced by Cruttwell and Shulman \cite[\S 7]{CruttwellShulman:2010au}.

\begin{thm}\label{thm:virtual-equipment} The virtual double category $\MMod{K}$ of modules in an $\infty$-cosmos $\lcat{K}$ is a virtual equipment: i.e., $\MMod{K}$ is a virtual double category such that 
\begin{enumerate}[label=(\roman*)]
\item For any module and pair of functors as displayed on the left, there exists a module and cartesian cell as displayed on the right satisfying the universal property of Proposition \ref{prop:cartesian-pullback-cells}. \[ \xymatrix{ A' \ar[d]_a & B' \ar[d]^b & \ar@{}[d]|{\displaystyle\stackrel{\exists}{\rightsquigarrow}} & A' \ar[d]_a \ar[r]|\mid^{E(b,a)} \ar@{}[dr]|{\Downarrow\rho} & B' \ar[d]^b \\ A \ar[r]|\mid_E & B & & A \ar[r]|\mid_E & B}\]
\item Every object $A$ admits a \emph{unit} module equipped with a nullary cocartesian  cell satisfying the universal property of Proposition \ref{prop:arrows-are-units}.
\[ \xymatrix{ A \ar@{=}[r] \ar@{=}[d] \ar@{}[dr]|{\Downarrow\iota} & A \ar@{=}[d] \\ A \ar[r]|\mid_{A^\cattwo} & A}\]
\end{enumerate}
\end{thm}

The \emph{virtual equipment of modules} in $\lcat{K}$ has a lot of pleasant properties, which follow formally from the axiomatization of a virtual equipment. For instance, certain sequences of composable modules can be said to \emph{have composites}, witnessed by cocartesian cells as in Proposition \ref{prop:arrows-are-units} (see  \refV{lem:comp-with-unit}, \refV{lem:comp-with-unit-cells}, \refV{lem:one-sided-rep-comp}, and \refV{cor:two-sided-rep-comp}). Also, for any functor $f \colon A \to B$, the modules $\dmod{B \comma f}{A}{B}$ and $\dmod{f \comma B}{B}{A}$ behave like adjoints is a sense suitable to a virtual double category; more precisely, the module $\dmod{B \comma f}{A}{B}$ defines a \emph{companion} and the module $\dmod{f \comma B}{B}{A}$ defines a \emph{conjoint} to $f \colon A \to B$ (see \refV{thm:companion-conjoint} and \refV{cor:companion-conjoint}). Another formal consequence of Theorem \ref{thm:virtual-equipment} is the following:

\begin{lem}[{\refV{lem:yoneda-embedding}}]\label{lem:yoneda-embedding} For any pair of parallel functors there are natural bijections between 2-cells
$ \xymatrix{ A \ar@/^2ex/[r]^f \ar@/_2ex/[r]_g \ar@{}[r]|{\Downarrow} & B}$ in the homotopy 2-category and cells
\[ \xymatrix{  A \ar@{=}[d] \ar[r]|\mid^{B \comma f} \ar@{}[dr]|{\Downarrow} & B \ar@{=}[d] & \ar@{}[d]|{\displaystyle\leftrightsquigarrow} & A \ar[r]|\mid^{A^{\cattwo}} \ar[d]_g  \ar@{}[dr]|{\Downarrow} & A \ar[d]^f &  \ar@{}[d]|{\displaystyle\leftrightsquigarrow} & B \ar@{}[dr]|{\Downarrow} \ar@{=}[d] \ar[r]|\mid^{g \comma B} & A \ar@{=}[d] \\ A \ar[r]|\mid_{B \comma g} & B & & B \ar[r]|\mid_{B^\cattwo} & B & & B \ar[r]|\mid_{f \comma B} & A}\]  
in the virtual equipment of modules.
\end{lem}

It follows from Lemma \ref{lem:yoneda-embedding} and the cited results about composition of modules  that there are two locally-fully-faithful homomorphisms $\h\lcat{K} \hookrightarrow \MMod{K}$ and $\h\lcat{K}\coop \hookrightarrow\MMod{K}$ embedding the homotopy 2-category into the sub bicategory of $\MMod{K}$ comprised only of unary cells whose vertical boundaries are identities; the latter homomorphism is contravariant on both 1- and 2-cells as signaled by the ``coop.'' The modules in the image of the first homomorphism are the covariant representables and the modules in the image of the second homomorphism are the contravariant representables. We refer to these as the \emph{covariant} and \emph{contravariant embeddings}, respectively.

\subsection{Pointwise Kan extensions}

We are now close to achieving our goal, definitions of pointwise Kan extensions between $\infty$-categories. By Proposition \ref{prop:cartesian-pullback-cells}, any cell in the virtual equipment of modules  can be represented uniquely as a cell between parallel sequences of modules, which we display inline as $E_1 \times_{A_1} \cdots \times_{A_{n-1}} E_n \To E$, where $E$ is a module from $A$ to $B$ and $E_1,\ldots, E_n$ is a composable sequence of modules starting at $A$ and ending at $B$. 

\begin{defn}[right extension of modules]\label{defn:right-extension-mod}
In the virtual equipment $\MMod{K}$ of modules, a \emph{right extension} of a module $\dmod{F}{A}{C}$ along a module $\dmod{K}{A}{B}$ is given by a module $\dmod{R}{B}{C}$ together with a cell $\nu \colon K \times_B R \To F$  so that for any composable sequence of modules $E_1,\ldots, E_n$ from $B$ to $C$, composition with $\nu$ defines a bijection
 \[\vcenter{ \xymatrix@=1em{ A \ar[rr]|\mid^K \ar[dd]|{\rotatebox{90}{$\labelstyle\mid$}}_F &  & B \ar[d]|(.4){\rotatebox{90}{$\labelstyle\mid$}}^{E_1} \\ & &   A_1\ar@/^/@{..>}[dl]    \\C & A_{n-1} \ar[l]|(.6)\mid^{E_n} \ar@{}[uu]|{\Leftarrow\chi}}} \qquad = \qquad \vcenter{ \xymatrix@=1em{ A \ar[rr]|\mid^K  \ar@{}[ddrr]|(.3){\Leftarrow\nu}|(.7){\Leftarrow\exists!} \ar[dd]|{\rotatebox{90}{$\labelstyle\mid$}}_F &  & B \ar[ddll]|{\rotatebox{45}{$\labelstyle\mid$}}^R  \ar[d]|(.4){\rotatebox{90}{$\labelstyle\mid$}}^{E_1} \\ & &   A_1\ar@/^/@{..>}[dl]    \\C & A_{n-1} \ar[l]|(.6)\mid^{E_n} &  }} \]
\end{defn}

In the case where the modules $\dmod{K}{A}{B}$, $\dmod{F}{A}{C}$, and $\dmod{R}{B}{C}$ are all covariant representables, the Yoneda lemma, in the form of Lemma \ref{lem:yoneda-embedding}, implies that the binary cell $\nu \colon K \times_B R \To F$ arises from a 2-cell in the homotopy 2-category. The following lemma, whose proof is left as an exercise, asserts that this 2-cell defines  a right extension in $\h\lcat{K}$ in the sense of Definition \ref{defn:right-ext}.

\begin{lem}[{\refV{lem:mod-extensions-in-2-cat}}]\label{lem:mod-extensions-in-2-cat}
If $\nu \colon B \comma k \times_B C \comma r \To C \comma f$ displays $\dmod{C \comma r}{B}{C}$ as a right extension of $\dmod{C \comma f}{A}{C}$ along $\dmod{B \comma k}{A}{B}$ in $\MMod{K}$, then $\nu \colon rk \To f$ displays $r$ as the right extension of $f$ along $k$ in $\h\lcat{K}$.
\[ \xymatrix{ A \ar[r]|\mid^{B \comma k} \ar[d]|{\rotatebox{90}{$\labelstyle\mid$}}_{C \comma f} & B \ar[dl]|{\rotatebox{45}{$\labelstyle\mid$}}^{C \comma r} & \ar@{}[d]|{\displaystyle \rightsquigarrow} &  A \ar[r]^k \ar[d]_f & B \ar[dl]^r \\ C & \ar@{}[ul]|(.7){\Leftarrow\nu} & & C & \ar@{}[ul]|(.7){\Leftarrow\nu}}\]
\end{lem}

\begin{defn}[stability of extensions under pasting]
In any 2-category, a right extension diagram 
\[ \xymatrix{ A \ar[r]^k \ar[d]_f & B \ar[dl]^r \\ C & \ar@{}[ul]|(.7){\Leftarrow\nu}}\]
is said to be \emph{stable under pasting with a square}
\[ \xymatrix{ D \ar[d]_g \ar[r]^h \ar@{}[dr]|{\Leftarrow\lambda} & E \ar[d]^b \\ A \ar[r]_k & B}\]
if the pasted diagram
\[ \xymatrix{ D \ar[r]^h \ar[d]_g \ar@{}[dr]|{\Leftarrow\lambda}  & E \ar[d]^b\\
A \ar[r]^k \ar[d]_f & B \ar[dl]^r \\ C & \ar@{}[ul]|(.7){\Leftarrow\nu}}\]
  displays $rb$ as a right extension of $fg$ along $h$.
\end{defn} 

The following proposition defines \emph{pointwise right Kan extensions} for $\infty$-categories:

\begin{prop}[{\refV{prop:pointwise-kan}}]\label{prop:pointwise-kan} For a diagram
 \[ \xymatrix{ A \ar[r]^k \ar[d]_f & B \ar[dl]^r \\ C & \ar@{}[ul]|(.7){\Leftarrow\nu}}\]
in the homotopy 2-category  of an $\infty$-cosmos $\lcat{K}$ the following are equivalent. 
\begin{enumerate}[label=(\roman*)]
\item\label{itm:pointwise-comma-stable} $\nu \colon rk \To f$ defines a right extension in $\h\lcat{K}$ that is stable under pasting with comma squares.
\item\label{itm:pointwise-in-mod} The image $\nu \colon B \comma k \times_B C \comma r \To C \comma f$ of $\nu$ under the covariant embedding $\h\lcat{K}\hookrightarrow\MMod{K}$ defines a right extension in $\MMod{K}$.
\end{enumerate}
\end{prop}

There are other equivalent conditions, one of which says that pointwise right Kan extensions are right extension diagrams that are stable under pasting with a larger class of \emph{exact squares} in $\h\lcat{K}$. Exact squares, which are characterized using the virtual equipment $\MMod{K}$, can also be used to define fully faithful functors of $\infty$-categories (see \refV{defn:fully-faithful}) and initial and final functors between $\infty$-categories (see \refV{defn:initial-final}).

Definition \ref{defn:right-extension-mod} can be dualized, by turning around the 1-cells but not the 2-cells, to define a \emph{right lifting} $\dmod{R}{C}{B}$ of a module $\dmod{F}{C}{A}$ along a module $\dmod{K}{B}{A}$. In the case, where all three modules are contravariant representables,  the Yoneda lemma, in the form of Lemma \ref{lem:yoneda-embedding}, implies that the binary cell $\nu \colon r \comma C \times_B k \comma B \To f \comma C$ arises from a 2-cell $\nu \colon f \To rk$ in the homotopy 2-category, which defines  a left extension in $\h\lcat{K}$ in the sense of Definition \ref{defn:right-ext}.

The dual to Proposition \ref{prop:pointwise-kan} defines \emph{pointwise left Kan extensions} for $\infty$-categories:

\begin{prop} For a diagram
 \[ \xymatrix{ A \ar[r]^k \ar[d]_f & B \ar[dl]^r \\ C & \ar@{}[ul]|(.7){\Rightarrow\nu}}\]
in the homotopy 2-category  of an $\infty$-cosmos $\lcat{K}$ the following are equivalent. 
\begin{enumerate}[label=(\roman*)]
\item\label{itm:pointwise-comma-stable-L} $\nu \colon f \To rk$ defines a left extension in $\h\lcat{K}$ that is stable under pasting with commas squares.
\item\label{itm:pointwise-in-mod-L} The image $\nu \colon r \comma C \times_B k \comma B \To f \comma C$ of $\nu$ under the contravariant embedding $\h\lcat{K}\coop\hookrightarrow\MMod{K}$ defines a right lifting in $\MMod{K}$.
\end{enumerate}
\end{prop}

Recall Definition \ref{defn:limit}: in a cartesian closed $\infty$-cosmos, the \emph{limit} of a diagram $d \colon J \to A$ is a element $\ell \colon 1 \to A$ equipped with an  absolute right lifting diagram
\begin{equation}\label{eq:limit-abs-lifting} \xymatrix{ \ar@{}[dr]|(.7){\Downarrow\nu} & A \ar[d]^{\Delta} \\ 1 \ar[r]_-{d}  \ar[ur]^\ell & A^J}\end{equation} Here the 2-cell $\nu$ encodes the data of the limit cone. 

\begin{prop}[{\refV{prop:limits-as-pointwise-kan}}]\label{prop:limits-as-pointwise-kan} In a cartesian closed $\infty$-cosmos $\lcat{K}$, any limit \eqref{eq:limit-abs-lifting} defines a pointwise right Kan extension 
\[ \xymatrix{ J \ar[d]_d \ar[r]^{!} \ar@{}[dr]|(.3){\Leftarrow\nu} &  1 \ar[dl]^\ell \\  A& }\] Conversely, any pointwise right Kan extension of this form transposes to define a limit \eqref{eq:limit-abs-lifting} in $A$.
\end{prop}

Proposition \ref{prop:limits-as-pointwise-kan} suggests the way to extend the definition of limits and colimits of diagrams indexed by $\infty$-categories to non-cartesian closed $\infty$-cosmoi.

\subsection{Model independence revisited}

We conclude with some remarks concerning the model independence of these notions. Note that a functor of $\infty$-cosmoi $F \colon \lcat{K} \to \lcat{L}$ induces a functor of sliced $\infty$-cosmoi $F \colon \lcat{K}_{/B} \to \lcat{L}_{/FB}$ for any $B \in \lcat{K}$.

\begin{prop}\label{prop:sliced-we} Suppose $F \colon \lcat{K} \to \lcat{L}$ is a weak equivalence of $\infty$-cosmoi. Then the induced functor 
$F \colon \lcat{K}_{/B} \to \lcat{L}_{/FB}$ is also a weak equivalence of $\infty$-cosmoi.
\end{prop}
\begin{proof}
We first argue that the functor between slices defines a local equivalence of sliced mapping quasi-categories, as defined in \ref{defn:sliced-cosmoi}. Given a pair of isofibration $p \colon E \tfib B$ and $p' \colon E' \tfib B$ in $\lcat{K}$, the induced functor on mapping quasi-categories is defined by 
\[ \xymatrix@=1em{ \Fun_B(p,p') \ar@{-->}[dr]^-{\rotatebox{155}{$\labelstyle\sim$}} \ar[rr] \pbexcursion \ar@{->>}[dd] & & \Fun(E,E') \ar[dr]^-{\rotatebox{155}{$\labelstyle\sim$}} \ar@{->>}'[d][dd] \\ & \Fun_{FB}(FE,FE') \pbexcursion \ar@{->>}[dd] \ar[rr] & & \Fun(FE,FE') \ar@{->>}[dd] \\ \Del^0 \ar'[r][rr] \ar@{=}[dr]  & & \Fun(E,B) \ar[dr]^-{\rotatebox{155}{$\labelstyle\sim$}} \\ & \Del^0 \ar[rr] & & \Fun(FE,FB)}\] As the maps between the cospans in $\qCat$ are equivalences, so is the induced map between the pullbacks.

For surjectivity up to equivalence, consider an isofibration $q \colon L \tfib FB$ in $\lcat{L}$. As $F$ is surjective on objects up to equivalence,  there exists some $A \in \lcat{K}$ together with an equivalence $i \colon FA \we L \in \lcat{L}$. As $F$ defines a local equivalence of mapping quasi-categories, there is moreover a functor $f \colon A \to B$ in $\mathcal{L}$ so that $Ff \colon FA \to FB$ is isomorphic to $qi$ in $\h\lcat{L}$. The map $f$ need not be an isofibration, but \ref{qcat.ctxt.cof.def}\ref{qcat.ctxt.cof:e} allows us to factor $f$ as $A \we K \xtfib{p} B$. Choosing an equivalence inverse $j \colon K \we A$, the result defines a diagram in $\h\lcat{L}$ that does not commute on the nose but which does commute up to isomorphism.
\[ \xymatrix{FK \ar@{->>}[dr]_{Fp} \ar[r]^{Fj}_{\sim} & FA \ar[d]|{Ff} \ar[r]^{i}_{\sim} & L \ar@{->>}[dl]^q \\  \ar@{}[ur]|(.7){\cong} & FB & \ar@{}[ul]|(.7){\cong} }\]

Now a basic fact about isofibrations in an $\infty$-cosmos that we have not had occasion to mention is that they define isofibrations in the homotopy 2-category. An \emph{isofibration} in a 2-category is a 1-cell that has a lifting property for isomorphisms with one chosen endpoint; see \refIV{rec:trivial-fibration} and \refIV{lem:isofib.are.representably.so}. In particular, as $q \colon L \tfib FB$ defines an isofibration in $\h\lcat{L}$, we may lift the displayed isomorphism along $q$ to define a commutative triangle:
\[ \vcenter{\xymatrix{FK \ar@{->>}[dr]_{Fp} \ar[r]^{Fj}_{\sim} & FA \ar[d]|{Ff} \ar[r]^{i}_{\sim} & L \ar@{->>}[dl]^q \\  \ar@{}[ur]|(.7){\cong} & FB & \ar@{}[ul]|(.7){\cong} }} \qquad =  \qquad \vcenter{ \xymatrix{FK \ar@{->>}[dr]_{Fp} \ar@/^2ex/[rr]^{i \cdot Fj} \ar@{}[rr]|\cong \ar@/_2ex/[rr]_e &  & L \ar@{->>}[dl]^q \\  & FB & } }\]  As $e$ is isomorphic to an equivalence $i \cdot Fj$, it must also define an equivalence, whence we have shown that the isofibration $p \colon K \tfib B$ maps under $F$ to an isofibration that is equivalent  to our chosen $q \colon L \tfib FB$.
\end{proof}

By Proposition \ref{prop:induced-2-functor} and Corollary \ref{cor:model-indep-fib}, a functor $F \colon \lcat{K} \to \lcat{L}$ of $\infty$-cosmoi induces a functor of virtual equipments $F\colon \MMod{K} \to \MMod{L}$; the important point here is that modules and simplicial pullbacks are preserved.

\begin{prop}\label{prop:equipment-we} If $F \colon \lcat{K} \to \lcat{L}$ is a weak equivalence of $\infty$-cosmoi, then the induced functor $F \colon \MMod{K} \to \MMod{L}$ defines a biequivalence of virtual equipments: i.e., it is
\begin{enumerate}[label=(\roman*)]
\item\label{itm:equipment-we:i} surjective on objects up to equivalence;
\item\label{itm:equipment-we:ii} locally bijective on isomorphism classes of parallel vertical functors;
\item\label{itm:equipment-we:iii} locally bijective on equivalence classes of parallel modules;
\item\label{itm:equipment-we:iv} locally bijective on cells.
\end{enumerate}
\end{prop}
\begin{proof}
\ref{itm:equipment-we:i} and \ref{itm:equipment-we:ii} are restatements of Proposition \ref{prop:we}\ref{itm:we-bi-equiv} and \ref{itm:we-iso}. The proof of \ref{itm:equipment-we:iii} is more subtle. Corollary \ref{cor:model-indep-fib} tells us that the property of an isofibration $E \tfib A \times B$ in $\lcat{K}$ defining a module is both preserved and reflected by $F$. An argument similar to that given in the proof of Proposition \ref{prop:sliced-we} shows that the weak equivalence of $\infty$-cosmoi $F \colon \lcat{K}_{/A\times B} \to \lcat{L}_{/FA \times FB}$ creates modules, and of course also preserves and reflects equivalences. Finally \ref{itm:equipment-we:iv} is an application of Proposition \ref{prop:we}\ref{itm:we-iso} to this weak equivalence of $\infty$-cosmoi.
\end{proof}

\begin{thm}[model independence of basic category theory II] Any basic $\infty$-categorical notion that can be encoded as an equivalence-invariant proposition in the virtual equipment of modules is model invariant: preserved and reflected by weak equivalences of $\infty$-cosmoi.
\end{thm}

We are only just beginning to explore the consequences of this result but they appear to be quite strong. For instance, we can prove that any (large) quasi-category $E$ that admits limits and colimits of every diagram indexed by a small category defines a \emph{derivator} 
\[ \Cat\op \xrightarrow{E^{-}}  \qCat \xrightarrow{\ho} \Cat\] 
in the sense of \cite{Heller:1988ly}; see \refV{rmk:qcat-derivator}. The proof makes use of two (non-formal) facts that we  prove directly in $\qCat$:
\begin{enumerate}[label=(\roman*)]
\item\label{itm:derivator-i} For any pair of modules $\dmod{G}{A}{C}$ and  $\dmod{K}{A}{B}$ between quasi-categories, there exists a right extension $\dmod{R}{B}{C}$.
\item\label{itm:derivator-ii} A module $\dmod{R}{B}{C}$ is equivalent to a comma module $\dmod{C \comma r}{B}{C}$ for some functor $r \colon B \to C$ between quasi-categories if and only if this property is true for the modules obtained by pulling back along each vertex $b \colon 1 \to B$.
\end{enumerate}

Now if $\lcat{K} \to \qCat$ is any weak equivalence of $\infty$-cosmoi, the biequivalence $\MMod{K} \to \MMod{\qCat}$ implies that \ref{itm:derivator-i} also holds in $\lcat{K}$. Given a pair of modules in $\lcat{K}$ form the right extension of their images in $\MMod{\qCat}$ and use \ref{prop:equipment-we}\ref{itm:equipment-we:iii} and \ref{itm:equipment-we:iv} to lift this module and the universal cell to a right extension diagram in  $\MMod{K}$. Now if $\dmod{R}{B}{C}$ is a module in $\lcat{K}$ that pulls back along all elements $b \colon 1 \to B$ to a represented module, then its image in $\qCat$ also has this property, using \ref{prop:equipment-we}\ref{itm:equipment-we:ii} to see that elements in $B$ correspond to elements in $FB$ up to isomorphism. The representing functor for the module $\dmod{FR}{FB}{FC}$ between quasi-categories lifts to a representing functor for $\dmod{R}{B}{C}$. 

As a consequence, we conclude that in an $\infty$-cosmos that is weakly equivalent to $\qCat$, any $\infty$-category that admits limits and colimits of every small diagram defines a derivator. 

By a similar argument, a related result --- \refI{cor:pointwise} --- that says that universal properties in $\h\qCat$ are determined ``pointwise'' can also be generalized to weakly equivalent $\infty$-cosmoi.

\providecommand{\bysame}{\leavevmode\hbox to3em{\hrulefill}\thinspace}
\providecommand{\MR}{\relax\ifhmode\unskip\space\fi MR }
\providecommand{\MRhref}[2]{%
  \href{http://www.ams.org/mathscinet-getitem?mr=#1}{#2}
}
\providecommand{\href}[2]{#2}

\end{document}